\RequirePackage{fix-cm}
\documentclass{article}
\input xy
\xyoption{all}

\usepackage{latexsym, epsfig, epic, xcolor, upgreek}
\usepackage{amsmath, amsthm, amsopn}
\usepackage{amssymb}
\usepackage{amscd}
\usepackage{multicol}
\usepackage{array}
\usepackage{amsfonts}
\usepackage{palatino}
\usepackage{graphicx}
\usepackage[normalem]{ulem}
\usepackage{tikz}
\usepackage{hyperref}

\newtheorem{thm}{Theorem}%[section]
 \newtheorem{coro}[thm]{Corollary}
 \newtheorem{lemm}[thm]{Lemma}
 \newtheorem{prop}[thm]{Proposition}
 
 \newtheorem{defi}[thm]{Definition}
 
 \newtheorem{rem}[thm]{Remark}

  \newtheorem{example}[thm]{Example}

\newcommand{\IC}{\mathbb{C}}

\newcommand{\IZ}{\mathbb{Z}}

\begin{document}

\title{A positive formula for type A Peterson Schubert calculus \thanks{This work was supported by National Science Foundation grant, Award \#1201458. }
}


\author{Rebecca Goldin     
 \and
        Brent Gorbutt %etc.
}


%\institute{R. Goldin (corresponding author) \at
 %             George Mason University
  %            Dept. of Mathematical Sciences
 %             4400 University Dr.
  %            Fairfax, VA  22030 (USA)
  %            \email{rgoldin@gmu.edu}      
  %            \url{https://orcid.org/0000-0001-7082-0248}
  %         \and
   %        B. Gorbutt \at
   %           George Mason University
    %          Dept. of Mathematical Sciences
     %         4400 University Dr.
     %         Fairfax, VA  22030 (USA)
     %         \email{bgorbutt@gmu.edu}%
%}


\maketitle

\begin{abstract}
The Peterson variety is a special case of a nilpotent Hessenberg variety, a class of subvarieties of $G/B$ that have appeared in the study of quantum cohomology, representation theory and combinatorics. In type A, the Peterson variety $Y$ is a subvariety of 
%We study the type A Peterson variety $Y$ in 
$Fl(n;\IC)$, the set of complete flags in $\IC^n$, and comes equipped with an action by a one-dimensional torus subgroup $S$ of a standard torus $T$ that acts on $Fl(n;\IC)$. Using the {\em Peterson Schubert basis} introduced in \cite{HT} and obtained by restricting a specific set of Schubert classes from $H_T^*(Fl(n; \IC))$ to $H_S^*(Y)$, we describe the product structure of the equivariant cohomology $H_{S}^*(Y)$.
In particular, we show that the product is {\em manifestly positive} in an appropriate sense by  providing an explicit, positive, combinatorial formula for its structure constants. 
A key step in our proof requires a new 
combinatorial identity of binomial coefficients that generalizes Vandermonde's identity, and merits independent interest.
\end{abstract}

%\keywords{Peterson \and Schubert \and structure constants \and Vandermonde}
%\PACS{14M15 \and 05E15\and 05A10}

%\begin{acknowledgements}
% The authors would like to thank Geir Agnarsson, James Lawrence, and Julianna Tymoczko for many useful conversations about both Hessenberg varieties and combinatorics. We also thank the referees for excellent and insightful suggestions for this manuscript.
%\end{acknowledgements}

%\section*{Conflict of interest} The authors certify that they have no relevant financial or non-financial interests to disclose.

%\newpage
\section{Introduction}\label{se:intro}

Let $G= Gl(n, \IC)$, $B$ upper triangular matrices, and $B_-$ lower triangular matrices. The quotient $G/B = Fl(n;\IC)$ is the associated {\em flag variety}. Let $T$ be compact form of the set of diagonal matrices in $G$, i.e. diagonal matrices in which each entry has norm 1. Then $G/B$ has a left $T$ action with isolated fixed points, $(G/B)^T$. The fixed point set may be identified with the Weyl group $W \cong S_n$, the permutation group on $n$ letters. We denote by $\mathfrak t$ the  
Lie algebra of $T$ and by $\mathfrak t^*$ its dual. 
Let $x_i$ be the $i$th coordinate function on $T\cong (S^1)^n$, for $i = 1,\dots n$. 
Finally let $\{\alpha_i := x_i-x_{i+1}:\ i\in \{1,\dots, n-1\} \}$ denote a choice of positive simple roots, with the property that the roots spaces of the Lie algebra $\mathfrak b$ of $B$  are positive. 

The ordinary cohomology and the $T$-equivariant cohomology of $G/B$ have a linear basis given by {\em Schubert classes} $\sigma_w$ as $w$ varies over elements of $W$. Indeed, they are each free modules over the corresponding ordinary or equivariant cohomology of a point. We use cohomology with complex coefficients throughout, and
 identify the equivariant cohomology of a point, denoted $H_T^*$, with the polynomial ring $\IC[x_1,\dots, x_n]$.

The products of Schubert classes define coefficients $c_{u,v}^w\in H_T^*$ by expanding in the basis:
$$
\sigma_u\sigma_v = \sum\limits_{w\in W} c_{u,v}^w \sigma_w
$$
for all $u, v\in W$. The coefficients $c_{u,v}^w$ are polynomials in $\alpha_1,\dots, \alpha_{n-1}$ with {\em non-negative coefficients} \cite{gr}. 

This manuscript describes a similar story with a particular subvariety of $Fl(n;\IC)$, namely the {\em Peterson variety} $Y$.
The Peterson variety is a special nilpotent Hessenberg variety first introduced in unpublished work by Peterson \cite{peterson:notes}, in which he proposed a link with the quantum cohomology of $Fl(n;\mathbb C)$. There are multiple equivalent definitions that have been given for the Peterson, and we provide one that works in all Lie types.  In Definition~\ref{defi:Peterson}, we provide another definition specific to the case that $G/B = Fl(n;\IC)$.  Let $w_0$ denote the longest word in the Weyl group $W$, and $e \in \mathfrak{b}$ a principal nilpotent element in the Lie algebra of $B$. Define $G^e$ be the centralizer of $e$.
The Peterson variety is defined as the closure in $G/B$ of an orbit of $G^e$ on the point $w_0B$, as follows:
\[ Y:= \overline{G^ew_0B} \hookrightarrow G/B \/. \]

Kostant elaborated on the connection to integrable systems, showing that the quantum cohomology ring of $Fl(n;\mathbb C)$ is isomorphic to the coordinate ring of an open dense affine subvariety of the Peterson variety \cite{Ko96}.
Rietsch generalizes these results to $G/P$ for any parabolic $P$, 
and proved the Peterson variety is paved by these affine varieties as $P$ varies \cite{rietsch2003totally}. Her work revealed an explicit relationship among geometric, algebraic and combinatorial descriptions of quantum cohomology, which she subsequently generalized to equivariant quantum cohomology, noting that each stratum may also play the role of a ``mirror symmetry phenomenon" for $G/P$ \cite{rietsch2008}.

The Peterson variety $Y$ in $Fl(n;\IC)$ is invariant under the action of a one-dimensional subgroup $S$ of $T$ (specified in Section~\ref{sse:background1}).  
We describe the product structure of the $S$-equivariant cohomology $H_{S}^*(Y)$ in a specific linear basis, termed 
the {\em Peterson Schubert basis}. In particular, we show that the product is {\em positive} in an appropriate sense by  providing an explicit positive combinatorial formula for the $S$-equivariant and ordinary structure constants (see Theorems~\ref{general},  \ref{thm:AuBconsec}, \ref{thm:AuBnotconsec},  \ref{thm:disconnectedgeneral}, and their corollaries).

The (equivariant) cohomology of the Peterson variety has been formulated and described in several ways.  Tymoczko showed the Peterson variety has a paving by affine cells \cite{T1}, implying its cohomology groups are nonzero only in even degrees. Tymoczko and Insko explore the non-equivariant cohomology through the study of its homology groups \cite{insko.tymoczko:intersection.theory}.
The ring structure has been described both as a quotient ring and as a subring of a sum of polynomial rings in work by Brion and Carrel \cite{BC}, Harada, Horiguchi, and Masuda \cite{HHM}, and Fukukawa, Harada, and Masuda in \cite{Fukukawa2015}, and via a connection with hyperplane arrangements \cite{abe2019hessenberg}.
Harada and Tymoczko \cite {HT} introduced a Schubert-type basis for the $S$-equivariant cohomology of the Peterson variety as a module over the $S$-equivariant cohomology of a point and 
  proved a {\em manifestly positive} Chevalley-Monk formula for the equivariant cohomology of the Peterson variety of $Fl(n;\IC)$. 
 Drellich extended the Chevalley-Monk formula proved by Harada and Tymoczko to all Lie types as well as proved Giambelli's formula for $Y$ in all Lie types \cite {Dr}.  After the appearance of this manuscript on the arXiv, Abe, Horiguchi, Kuwata, and Zeng posted a paper that computes the structure constants for the ordinary cohomology of $Y$ \cite{Abe-Horiguchi-Kuwata-Zeng}.

Harada and Tymoczko's insight was to use a natural composition
$$j: H_T^*(Fl(n;\IC))\longrightarrow H_{S}^*(Fl(n;\IC))\longrightarrow H_{S}^*(Y)$$
to obtain a basis of $H_{S}^*(Y)$ (as a module over $H_{S}^*$) as the image of a specific subset of Schubert classes on $Fl(n;\IC)$ indexed by subsets  
$$A \subseteq [n-1]=\{1, \ldots, n-1\}.$$
More specifically, let $\alpha_1, \dots \alpha_{n-1}$ denote the simple roots ordered by adjacency in the Dynkin diagram, and $s_1,\dots, s_{n-1}$ the corresponding reflections. For   $A = \{a_1,\dots, a_k\}$ listed in increasing order, let
 $$
 v_A= s_{a_1}s_{a_2}\dots s_{a_k}
 $$
 and $\sigma_{v_A}$ the corresponding Schubert class. The {\em Peterson Schubert classes} $p_A$ are defined by 
 $$p_A = j(\sigma_{v_A}).$$
The set $\{p_A\}_{A\subseteq [n-1]}$ forms a module basis of  $H_{S}^*(Y)$. Thus the product of two Peterson Schubert classes is an $H_{S}^*$-linear combination of Peterson Schubert classes.  
For $A, B, C \subseteq \{1, \ldots, n-1\}$, define the {\em structure constant} $b_{A,B}^C \in H_{S}^*$ by
\begin{equation}\label{eq:structureconstants}
p_A  p_B =
\sum_{C \subseteq \{1, \ldots, n-1\}}
b_{A,B}^C \ p_C.
\end{equation}
Harada and Tymoczko show that $b_{A,B}^C$ is a non-negative integer multiple of a power of $t$ when $A=\{i\}$ consists of a single element, and provide a positive (counting) formula for the coefficients $b_{\{i\}, B}^C$. 

Their work raises the enticing question of whether the product structure is positive in the equivariant sense, i.e. whether the structure constants $b_{AB}^C$ are polynomials with nonnegative coefficients for all $A, B, C$.  Our main results are combinatorially positive formulas for these equivariant Peterson Schubert structure coefficients
when $G = Gl(n,\IC)$.  
The explicit formulas are found in Theorems~\ref{general},  \ref{thm:AuBconsec}, \ref{thm:AuBnotconsec}, and \ref{thm:disconnectedgeneral}, which together provide manifestly positive formulas for the equivariant structure constants of $H_{S}^*(Y)$ in the basis $\{p_A: A\subset\{1,\dots, n-1\}\}$ of Peterson Schubert classes.
As a result, we obtain both the statement that structure constants are nonnegative, as well as simple criteria for when they are positive.  The first author explores a geometric proof of positivity in all Lie types in separate work \cite{Goldin-Milhalcea-Singh}.

We call a subset $C_k\subset C\subset \{1,\dots, n-1\}$  {\em maximal consecutive} if $C_k$ is consecutive set such that 
\begin{align*}
(\min C_k-1)\not\in C && \mbox{and} &&  (\max{C_k}+1)\not\in C.
\end{align*}

\vspace{.1in}

\noindent {\bf  Corollary \ref{cor:nonnegativecoeff}, Theorem \ref{thm:positive}} 
{\em The equivariant structure constants $b_{A,B}^C$ defined by \eqref{eq:structureconstants}
are nonnegative, integral multiples of powers of $t$. They have  positive coefficients if and only if $A\cup B\subseteq C$ and each maximal consecutive subset $C_k$ of $C$ satisfies $|C_k|\leq |C_k\cap A| + |C_k\cap B|$.}
\vspace{.1in}

One consequence of these theorems is a manifestly positive formula for the structure constants in ordinary Peterson Schubert calculus (Corollary~\ref{cor:ordinary}).

The proofs in this paper are combinatorial rather than geometric. A crucial step for the proof is an unexpected combinatorial identity (Theorem~\ref{identity}), a generalization of Vandermonde's identity, which we prove using a technique we term {\em bike lock moves}.

The structure of the paper is as follows. In Section~\ref{se:positivity} we state the main positivity theorems which together provide a full picture of the positivity of the structure constants. In Section~\ref{se:background} we define the basics of equivariant cohomology, Peterson varieties, and positivity. We prove the main positivity theorems in Section \ref{se:proofs}, and the crucial combinatorial theorem in Section~\ref{se:proofofidentity}. 

\vspace{.1in}
\thanks{\noindent {\bf Acknowledgements} 
The authors would like to thank Geir Agnarsson, James Lawrence, and Julianna Tymoczko for many useful conversations about both Hessenberg varieties and combinatorics. We also thank the referees for excellent and insightful suggestions for this manuscript.}

\section{Positivity Theorems}\label{se:positivity}

In this section, we describe the main results on the structure constants for the equivariant cohomology $H_{S}^*(Y)$ of the Peterson variety $Y$  in $Fl(n;\IC)$ (both defined in Section~\ref{se:background}), which show directly their positivity. To each subset $A\subseteq \{1,2,\dots, n-1\}$, we define an element $p_A \in H_S^*(Y)$ in Section~\ref{sec:Petersonclasses} as the pullback of a specific Schubert class from $G/B$. We call $p_A$  a {\em Peterson Schubert class},  The collection $\{ p_A: A\subset \{1,\dots, n-1\}\}$ a free module basis for the equivariant cohomology $H_S^*(Y)$ over $H_S^*:=H_{S}^*(pt)$. Define the {\em structure constants} $b_{A,B}^C\in H_{S}^*$ by
\begin{equation}\label{babc2} 
p_A p_B = \sum_{C\subseteq \{1,2,\dots, n-1\}} b_{A,B}^C\ p_C.
\end{equation}
By construction, $p_\emptyset=1$, and thus the coefficients $b_{A,B}^C$ are easy to calculate when $A, B$ or $C$ is empty: $b_{A,\emptyset}^A=b_{\emptyset,A}^A=1$ for all $A\subseteq \{1,\dots, n-1\}$, and all other coefficients vanish. 

For $A, B, C$ nonempty, Theorem \ref{general} gives an explicit positive, integral formula for the coefficients $b_{A,B}^C$ 
when $A$ and $B$ are consecutive. Theorems~\ref{thm:AuBconsec}, \ref{thm:AuBnotconsec} and \ref{thm:disconnectedgeneral} describe the constants in the nonconsecutive cases.  
Nonvanishing conditions for the structure constants are described in Theorem~\ref{thm:positive}. Proofs are relegated to Section~\ref{se:proofs}.
 
We recall notation found in \cite{HT}. For $A \subseteq \{1, \ldots, n-1\}$ with $A$ nonempty and consecutive, let  $\mathcal T_A = \min\{a\in A\}$ and $\mathcal H_A = \max \{a\in A\}$, called the {\em tail} and {\em head} of $A$, respectively.

\begin{thm}[$A, B, C$ consecutive] \label{general} 
Let $A, B, C \subseteq \{1, \ldots, n-1\}$ be nonempty consecutive subsets. 
 If $C \supseteq A \cup B$ and $|C| \leq |A|+|B|,$ then
\begin{equation}\label{eq:general}
b_{A, B}^C = 
d!
{\mathcal H_A - \mathcal T_B +1 \choose d, \ \mathcal T_A - \mathcal T_C, \ \mathcal H_C - \mathcal H_B}
{\mathcal H_B - \mathcal T_A +1 \choose d, \ \mathcal T_B - \mathcal T_C, \ \mathcal H_C - \mathcal H_A}
t^d
\end{equation}
for $d:=|A|+|B|-|C|$.
\end{thm}

\begin{example} 
Let $A= \{1,2\}$, $B=\{2,3,4\}$ and $C=\{1,2,3,4\}$. Then $C$ is consecutive, contains $A\cup B$ and $|C|=4\leq |A|+|B|=5$, so that $b_{A,B}^C$ is given by \eqref{eq:general}. 
Observe 
\begin{align*}
\mathcal H_A=2 && \mathcal T_A =1  && 
\mathcal H_B =4 &&\mathcal T_B =2\\
\mathcal T_C = 1&& \mathcal H_C = 4 & & d=1
\end{align*}
so that $b_{A,B}^C= 1!{ 1 \choose 1, \ 0, \ 0}{ 4 \choose 1, \ 1, \ 2}t^1= \frac{4!}{2!}t=12t.$
\end{example}

An immediate consequence of Theorem~\ref{general} is a formula for the ordinary cohomology structure constants.
For degree reasons, the product $p_A p_B$ in ordinary cohomology requires simply summing over classes $p_C$ such that $|C| = |A| + |B|$.  

\begin{coro} \label{cor:ordinary}  Let $A, B, C \subseteq \{1, \ldots, n-1\}$ be nonempty consecutive subsets. Suppose $A \cup B \subseteq C$, and $|C| = |A| + |B|$. Without loss of generality, assume that $\mathcal T_A \leq \mathcal T_B$.  Then $b_{A,B}^C$ is the product of binomial coefficients:
$$
b_{A,B}^C 
=
{\mathcal H_A - \mathcal T_B +1 \choose \mathcal T_A - \mathcal T_C}
{\mathcal H_B - \mathcal T_A +1 \choose \mathcal T_B - \mathcal T_C}.
$$
\end{coro}

\begin{proof} By the degree assumption, 
$\mathcal H_C-\mathcal T_C+1 = (\mathcal H_A -\mathcal T_A+1) +  (\mathcal H_B -\mathcal T_B+1)$. Thus $ \mathcal H_A -\mathcal T_B+1 = (\mathcal T_A-\mathcal T_C)+(\mathcal H_C-\mathcal H_B)$ and 
$$  \mathcal H_B -\mathcal T_A+1 = (\mathcal T_B-\mathcal T_C)+(\mathcal H_C-\mathcal H_A).$$
 The corollary follows.   \end{proof}

 We successively loosen the restrictive demand of Theorem~\ref{general} that $A, B$ and $C$ are each sets with consecutive numbers, as follows:
\begin{itemize}
\item[$\bullet$] Sets $A\cup B$ and $C$ consecutive (Theorem \ref{thm:AuBconsec}),
\item[$\bullet$] The set $C$ is consecutive (Theorem~\ref{thm:AuBnotconsec}), and
\item[$\bullet$] No constraint on $A, B, C$ (Theorem~\ref{thm:disconnectedgeneral}).
\end{itemize}

When $A, B$ or $C$ are not consecutive, 
there are non-equivariant analogs for ordinary cohomology.  
We won't list them, however, as each result is identical to the corresponding theorem with an additional hypothesis to ensure the degree is correct: 
 any coefficient $b_{E, F}^G$ occurring in the formula are set to $0$ unless $|E|+|F|=|G|$.

\begin{thm}[$A\cup B$, $C$ consecutive]\label{thm:AuBconsec}
Let $A,B,C \subseteq \{1, \ldots, n-1\}$ be nonempty subsets with $A\cup B$  and $C$
 consecutive. 
Rename the maximal consecutive subsets of 
 $A$ and $B$ by  $E_1,\dots, E_v$
 ordered with increasing tails i.e. $\mathcal T_{E_1}\leq \mathcal T_{E_2}\leq \dots\leq \mathcal T_{E_v}$.
Then
\begin{equation}\label{eq:AuBconsec}
b_{A,B}^C =\sum_{(C_2,\dots, C_{v-1})}
b_{E_1,E_2}^{C_2} b_{C_2,E_3}^{C_3} b_{C_3,E_4}^{C_4}\dots. b_{C_{v-2}, E_{v-1}}^{C_{v-1}}
b_{C_{v-1},E_v}^C
\end{equation}
where the sum is over $v-2$-tuples 
of consecutive sets $C_i$.
\end{thm}
Note that, for each term in the sum of Theorem~\ref{thm:AuBconsec}, the factors $b_{E_1,E_2}^{C_2}$ and $b_{C_i, E_{i+1}}^{C_{i+1}}$ are each calculated using Theorem~\ref{general} (as $C_i, E_{i+1}$ and $C_{i+1}$ are all consecutive).

\begin{example} 
 Let $A = \{1,2, 4, 5\}$,  $B = \{2,3,4\}$ and $C=\{1,2,3,4,5,6\}$. 
We use Theorem~\ref{thm:AuBconsec} to compute $b_{A,B}^C$ noting that $A\cup B$ is consecutive.
 
 By ordering according to the smallest element in each maximal consecutive set, choose $E_1 = \{1,2\}, E_2=B, E_3=\{4,5\}$ and note $v=3$.
 Thus the sum \eqref{eq:AuBconsec} is
$$
b_{A,B}^C = \sum\limits_{(C_2)\atop C_2\ \mbox{ \tiny consecutive}} b_{E_1, E_2}^{C_2} b_{C_2,E_3}^C.
$$
By Theorem~\ref{general}, $b_{E_1, E_2}^{C_2} \neq 0$ implies  $C_2$ contains $E_1\cup E_2 = \{1,2,3,4\}$ and $|C_2|\leq |E_1|+|E_2| =5$.  
Since $C_2$ is consecutive, the two possibilities are $C_2 = \{1,2,3,4\}$ and $C_2 = \{1,2,3,4,5\}$. Thus by Theorem~\ref{thm:AuBconsec} 
$$
b_{A,B}^C =  b_{E_1, E_2}^{\{1,2,3,4\}}  b_{\{1,2,3,4\},E_3}^C + b_{E_1,E_2}^{\{1,2,3,4,5\}} b_{\{1,2,3,4,5\},E_3}^C.
$$
Each factor of each term can be computed using Theorem~\ref{general}:
\begin{align*}
b_{E_1,E_2}^{ \{1,2,3,4\}} &=
1!{1 \choose 1, \ 0, \ 0}{4 \choose 1, \ 1, \ 2}t^{1} = 12t \\
b_{\{1,2,3,4\},E_3}^C &=
0!{1 \choose 0, \ 0, \ 1}{5 \choose 0, \ 3, \ 2}t^{0}= 10 \\
b_{E_1,E_2}^{\{1,2,3,4,5\}}&=
0!{1 \choose 0, \ 0, \ 1} {4 \choose 0, \ 1, \ 3} t^{0} = 4 \\
b_{\{1,2,3,4,5\},E_3}^C  &=  
1!{2 \choose 1, \ 0, \ 1}{5 \choose 1, \ 3, \ 1}t^{1} = 40t.
\end{align*}
Therefore $b_{A,B}^C = 12t\cdot 10 + 4\cdot 40t = 280t.$
\end{example}

The following theorem is a complete description of the product when $C$ is consecutive.

\begin{thm}[$C$ consecutive]\label{thm:AuBnotconsec}
Let $A,B,C \subseteq \{1, \ldots, n-1\}$ be nonempty subsets with $C$ consecutive. 
Let $A\cup B = D_1\cup\dots\cup D_u$ be a union of maximal consecutive subsets. 
Write $A^i = D_i\cap A$ and $B^i = D_i\cap B$, and note that $D_i = A^i\cup B^i$.
Then $$
b_{A,B}^C = \sum_{(E_1,\dots, E_u): \ D_i\subseteq E_i,\atop{E_i\mbox{ {consecutive}}} }\left(\prod_{i=1}^u b_{A^i, B^i}^{E_i}\right)b_{E_1,\dots, E_u}^C,$$
where $b_{A^i, B^i}^{E_i}$ is calculated using Theorem~\ref{thm:AuBconsec}, and
$b_{E_1,\dots, E_u}^C$ is the coefficient of $p_C$ in the product $\prod_{i=1}^u p_{E_i}$.  

If $\cup_i E_i$ is consecutive, $b_{E_1,\dots, E_u}^C$ may be calculated by Theorems \ref{general} and \ref{thm:AuBconsec}.
If $\cup_i E_i$ is not consecutive, 
\begin{equation}\label{eq:productofmanyEs}
b_{E_1,\dots, E_u}^C = \sum_{\substack{(F^{(1)}, F^{(2)},\dots, F^{(u-2)})\\ \text{consecutive}}} b_{E^{(1)}_{j_1}, E^{(1)}_{k_1}}^{F^{(1)}} b_{E^{(2)}_{j_2},E^{(2)}_{k_2}}^{F^{(2)}} \dots b_{E^{(u-2)}_{j_{u-2}}, E^{(u-2)}_{k_{u-2}}}^{F^{(u-2)}}
 b_{E^{(u-1)}_{j_{u-1}}, E^{(u-1)}_{k_{u-1}}}^C
\end{equation}
where $E^{(1)}_i= E_i$,  and the sets $E^{(s)}_i$ for $s=2,\dots, u-1$ 
are defined inductively as follows. $E_{j_s}^{(s)}$ and $E_{k_s}^{(s)}$ are chosen so that their union is consecutive, the sum is over consecutive sets $F^{(s)}$ containing $E_{j_s}^{(s)}\cup E_{k_s}^{(s)}$, and 
the sets $E^{(s+1)}_i$ are a relabeling of the $u-s$ sets 
$$F^{(s)}, E^{(s)}_1,\dots, \widehat{E}^{(s)}_{j_{s}},\widehat{E}^{(s)}_{k_{s}}, \dots, E^{(s)}_{(u-s+1)}$$
 in which the two sets $E^{(s)}_{j_{s}}$ and $E^{(s)}_{k_{s}}$ have been excluded. The sum is independent of choices involved with ordering. Each term $b_{E^{(s)}_{j_s},E^{(s)}_{k_s}}^{F^{(s)}}$ may be calculated using Theorem~\ref{general} as $E^{(s)}_i$ is consecutive.  
\end{thm}
Note that the sum in \eqref{eq:productofmanyEs} is not independent of the order of $F^{(1)}, \dots F^{(u-2)}$. The set of possible $F^{(s)}$ depend on the term $F^{(s-1)}$ in the prior sum, as well as the choice of sets $E^{(s)}_{j_{s}}$ and $E^{(s)}_{k_{s}}$ whose union is consecutive. Theorem~\ref{thm:AuBnotconsec} guarantees that these sets exist for each $s$ when the coefficient is nonzero.

Finally, when $C$ is not consecutive, $b_{A,B}^C$ is a product of coefficients with consecutive superscripts. 
  \begin{thm}\label{thm:disconnectedgeneral} Let $A,B,C \subseteq \{1, \ldots, n-1\}$ be subsets such that $b_{A,B}^C\neq 0$.  Then
$$
b_{A,B}^C = \prod_{k=1}^m b_{A\cap C_k,B\cap C_k}^{C_k}.
$$
where  $C=C_1\cup \dots \cup C_m$ is written as a union of maximal consecutive subsequences. 
\end{thm}

An immediate corollary to these theorems is  that the structure constants for multiplication of $\{p_A\}$ in $H_{S}^*(Y)$, and hence  in $H^*(Y)$ are nonnegative. 

\begin{coro}\label{cor:nonnegativecoeff}
For any $A,B,C \subseteq \{1, \ldots, n-1\}$, $b_{A,B}^C$ is a nonnegative, integral multiple of a power of $t$.
\end{coro}
\begin{proof}
If $A$ and $B$ are consecutive, then this follows immediately from Theorem~\ref{general} as $b_{A,B}^C$ is 0, 1, or described by Equation~\ref{eq:general}.
If $A$ or $B$ is not consecutive, but $A\cup B$ is consecutive, then Theorem~\ref{thm:AuBconsec} implies that $b_{A,B}^C$ is a sum of products of the terms for consecutive $A$ and $B$.  Finally, Theorems~\ref{thm:AuBnotconsec} and \ref{thm:disconnectedgeneral} show that when $A\cup B$ is not consecutive, the terms associated with consecutive pieces are nonnegative and integral, and the terms associated with the product of those terms is also nonnegative and integral.
  \end{proof}

Finally, we state a nonvanishing result for arbitrary $A, B, C$.

\begin{thm}\label{thm:positive} Let $A,B,C \subseteq \{1, \ldots, n-1\}$ be arbitrary subsets. The structure constant 
$b_{A,B}^C\neq 0$ 
if and only if 
 \begin{itemize}
\item[$\bullet$]   $A\cup B \subseteq C$, and 
\item[$\bullet$]   For each maximal consecutive subset $C_k$ of $C$, $|C_k|\leq |C_k\cap A| + |C_k\cap B|$.
\end{itemize}
\end{thm}

Theorem~\ref{thm:positive} and Corollary~\ref{cor:nonnegativecoeff} imply these structure constants are positive (i.e. are monomials with positive coefficients) when they are non-vanishing.

The proof of Theorem \ref{general} relies heavily on the following combinatorial result, a generalization of Vandermonde's formula. 

\begin{thm} \label{identity}
Let $m,n,w,x,y,z \in \mathbb Z$ with $w+x = y+z$ and $m,n \geq 0$.  Then
\begin{multline} \label{eq:combinatorialidentity}
{w+m \choose w}
{y+m \choose x}
{w+n \choose y}
{z+n \choose z} 
\\
=
\sum_{
\substack{
0 \leq i \leq m \\
0 \leq j \leq n}}
{w+i +n \choose w+i+j}
{w+m+j \choose
i, \ j, \ m-i, \ x-i-j, \ z-x+j, \ y-x+i}.
\end{multline}
\end{thm}
We have thusfar not found this result in the literature, and it may stand alone as a worthwhile combinatorial identity, proved in Section~\ref{se:proofofidentity}.

\section{Background and Notation}\label{se:background}

\subsection{Flag varieties, Peterson varieties, and fixed points}\label{sse:background1}
Let $G=Gl(n;\IC)$, $B$ upper triangular invertible matrices, $B_-$ lower triangular invertible matrices, and 
$T$ the set of diagonal matrices in $G$. Recall $G/B$ is  naturally isomorphic to the set of complete flags
$$
Fl(n;\mathbb C) = \{V_\bullet := (V_1 \subseteq \cdots \subseteq V_{n-1} \subseteq \mathbb C^n)|
 \ V_i \text{ {\small is a subspace of} }\mathbb C^n,\dim_\mathbb C(V_i) = i\}.
$$ 
 The flag $V_\bullet $ corresponds to a a coset $gB$, where $g\in Gl(n,\IC)$ is any matrix whose first $k$ columns form a basis for $V_k$, for $k=1, \dots, n$. Note that right multiplication by an upper triangular matrix (in $B$) preserves the vector space spanned by the first $k$ columns, for all $k$.
The fixed points $(G/B)^T$ are isolated, and indexed by elements of the Weyl group, $W \cong S_n$. In particular,
$$
(G/B)^T = \{wB/B :\ w \in W\}.
$$

Following Tymoczko \cite{T1}, we describe Hessenberg varieties in $Fl(n;\IC)$ as a set of flags whose vector spaces satisfy linear conditions imposed by a principal nilpotent operator. The equivalence of this description with the original definition by Kostant is known to experts and proven in \cite{Goldin-Milhalcea-Singh}.
\begin{defi}
Let $h:\{1, \ldots, n\} \rightarrow \{1, \ldots, n\}$ be a function satisfying $i \leq h(i)$ for all  $i~ \in~ \{1, \ldots, n\}$ and let $M$ be any $n \times n$ matrix $M$.  The {\em Hessenberg variety} $H(h,M)$ corresponding to $h$ and $M$ is the collection of  flags $V_\bullet \in Fl(n;\IC)$ satisfying  $MV_i \subseteq V_{h(i)}$ for all $1 \leq i \leq n$.  
\end{defi}
The Peterson variety $Y$ is a specific Hessenberg variety, 
with $h$ given by:
\begin{equation}\label{phes}
h(i) =
\begin{cases}
i+1 & 1 \leq i \leq n-1 \\
n & i = n.
\end{cases}
\end{equation}
\begin{defi}\label{defi:Peterson}
The {\em Peterson variety} in $Fl(n;\IC)$
is the Hessenberg variety $Y = H(h, M)$ where $h$ is the function defined in Equation~\eqref{phes} and $M$ is a principal nilpotent operator.  Equivalently the Jordan canonical form for $M$ consists of one block and $M$ has eigenvalue 0.
\end{defi}

\begin{example} 
Let $n = 3$, $h(1) = 2$, $h(2) = 3$, $h(3) = 3$ and 
$
M = 
\begin{pmatrix}
0 & 1 & 0 \\
0 & 0 & 1 \\
0 & 0 & 0
\end{pmatrix}.
$
The Peterson variety in $Fl(\IC^3)$ consists of flags represented by matrices of the following forms:
\begin{equation} \label{n=2}
\begin{pmatrix}
a & b & 1 \\
b & 1 & 0 \\
1 & 0 & 0
\end{pmatrix},
\ \ \
\begin{pmatrix}
c & 1 & 0 \\
1 & 0 & 0 \\
0 & 0 & 1
\end{pmatrix},
\ \ \
\begin{pmatrix}
1 & 0 & 0 \\
0 & d & 1 \\
0 & 1 & 0
\end{pmatrix},
\ \ \
\begin{pmatrix}
1 & 0 & 0 \\
0 & 1 & 0 \\
0 & 0 & 1
\end{pmatrix}
\end{equation}
where $a,b,c,d  \in \mathbb C$.  We verify the condition that $MV_i \subseteq V_{h(i)}$ for the first matrix above.  We check that $MV_1 \subseteq V_2$ (clearly $MV_2 \subseteq V_3 = \mathbb C^3$):
$$
\begin{pmatrix}
0 & 1 & 0 \\
0 & 0 & 1 \\
0 & 0 & 0
\end{pmatrix}
\begin{pmatrix}
a \\
b \\
1
\end{pmatrix}
=
\begin{pmatrix}
b \\
1 \\
0
\end{pmatrix}
\in 
\text{span} \left \{
\begin{pmatrix}
a \\
b \\
1 
\end{pmatrix},
\begin{pmatrix}
b \\
1 \\
0
\end{pmatrix}
\right \}
=V_2.
$$
\end{example}

As $T$ consists of diagonal, unitary matrices, we write elements as $n$-tuples $(a_1,\dots a_n)$ listing the diagonal entires. The variety $Y$ is not $T$-stable, however it is stable under a subgroup isomorphic to 
$S^1$. 
Let
$$
S = 
\{(z^n, z^{n-1}, \dots, z^2, z):\  z \in \mathbb C^*, ||z||^2=1\}\subseteq T.
$$ 
We observe that $S$ preserves $Y$, as follows.  Let $e_i\in \IC^n$ be the vector with $1$ in the $i$th coordinate, and $0$ elsewhere. 
For any vector $v\in \IC^n$ given by $v = \sum_{i=1}^n a_i e_i,$ we have
$$Mv = \sum_{i=1}^{n-1} a_{i+1} e_i.$$ 
On the other hand, for each element $s$ of $S$ given by a diagonal matrix with entries $(z^n, z^{n-1}, \dots, z)$, we have $s\cdot v = \sum_{i=1}^n z^{n-i+1}a_i e_i.$
A quick calculation shows that $s\cdot Mv$ and  $M(s\cdot v)$ span the same line:
\begin{align*}
s\cdot Mv &= \sum_{i=1}^{n-1} z^{n-i+1}a_{i+1} e_i = z\sum_{i=1}^{n-1} z^{n-i}a_{i+1} e_i = zM(s\cdot v).
\end{align*}
It follows that $M(s\cdot V_k)$ is in the span of $s\cdot MV_k$. If $V_\bullet\in Y$, then $MV_k\subseteq V_{k+1}$ implies $M(s\cdot V_k)\subseteq s\cdot MV_k \subseteq s\cdot V_{k+1}$, and hence $s\cdot V_\bullet\in Y$. 

As $S$ is a regular one-parameter subgroup of $T$, the $S$-fixed points of $G/B$ are the same as the $T$-fixed points.  It follows that the fixed point set $Y^{S}$ may be described as the intersection $Y^{S} = Y\cap (G/B)^T.$

Explicitly, $Y^{S}$ consists of flags represented by block diagonal matrices where the diagonal blocks are anti-diagonal with $1$'s on the anti-diagonal:
$$
\begin{pmatrix}
0      & \cdots                & 1      &        &                       &        &        &        &                       &        \\
\vdots & \reflectbox{$\ddots$} & \vdots &        &                       &        &        &        &                       &        \\
1      & \cdots                & 0      &        &                       &        &        &        &                       &        \\
       &                       &        &        &                       &        & \ddots &        &                       &        \\
       &                       &        &        &                       &        &        & 0      & \cdots                & 1      \\
       &                       &        &        &                       &        &        & \vdots & \reflectbox{$\ddots$} & \vdots \\
       &                       &        &        &                       &        &        & 1      & \cdots                & 0     
\end{pmatrix}.
$$
For example if $n = 2$ then $Y^{S}$ consists of flags represented by matrices \eqref{n=2} in   the previous example with $a = b = c =d=0$.

Each simple root $\alpha_i$ corresponds to a simple reflection $s_i:=s_{\alpha_i}$ that interchanges $i$ and $i+1$.  
 Recall  an element $w\in S_n$ can be written as a product of simple reflections $s_1,\dots, s_{n-1}$, corresponding to the simple roots $\alpha_1,\dots, \alpha_{n-1},$ respectively. When $w = s_{i_1}s_{i_2}\cdots s_{i_{\ell(w)}}$ is written as a product with as few simple reflections as possible, $\ell(w)$ is called the {\em length} of $w$.  The expression
 $s_{i_1}s_{i_2}\cdots s_{i_{\ell(w)}}$ is called a {\em reduced word decomposition} for $w$. To distinguish the product (resulting in $w$) from a sequence of $\ell(w)$ simple reflections in a reduced word decomposition, we refer to the index sequence  $(i_1,i_2,\dots, i_{\ell(w)})$ as a {\em reduced word sequence} for $w$. Recall the Bruhat order for $u,v \in S_n$: we say $u\leq v$ if there exists a substring of a reduced word for $v$ whose corresponding product of reflections is $u$. There exists a unique element $w_0$ in $S_n$ with maximal length, and it satisfies $w\leq w_0$ for all $w\in S_n$.

Elements of $Y^{S}$ are represented by a specific set of permutations:
 \begin{equation}\label{Yfixedpoint}
 Y^{S} = \{w_A\in S_n: A\subseteq \{1, \ldots, n-1\} \},
 \end{equation}
 where the permutation $w_A$ associated to a subset $A$ is given as follows. 
 Let $A = A_1 \cup A_2 \cup \cdots \cup A_k$ where each $A_i$ is a maximal consecutive subset of $A$. For each $i$, denote by $w_{A_i}$ the long word of the subgroup $H_i$ of $S_n$ generated by reflections $s_j$ for $j\in A_i$, noting that $H_i\cong S_{|A_i|+1}$ is itself a permutation group. 
 Then 
 $$w_A = w_{A_1}w_{A_2}\cdots w_{A_k}$$
  is the long word of the subgroup $H_1 \times H_2 \times \cdots \times H_k\subseteq S_n$.  A matrix representing a $w_AB\in Y^{S}$ has anti-diagonal blocks of size $|A_i|+1$.

\subsection{The equivariant cohomology ring of $G/B$ and Schubert classes}\label{sse:background2}

Define $B$-invariant Schubert varieties $X^w := \overline{BwB}/B$ in $G/B$, and let $[X^w]$ denote the corresponding $T$-equivariant homology class, following \cite{Brion00}. 
We use Poincar\'e duality between equivariant homology and equivariant cohomology to define a dual basis $\{\sigma_w: \ w\in W\}$ of $H_T^*(G/B)$ to the equivariant homology basis $\{[X^w]:\ w\in W\}$.  These bases satisfy the property that $\langle \sigma_w, [X^v]\rangle = \delta_{wv}$, where $\langle\ , \ \rangle$ denotes the equivariant cap product, followed by the pushforward to a point. 

Alternatively, $\sigma_w$ is Poincar\'e dual to the equivariant homology class of the opposite Schubert variety $X_w:=\overline{B_-wB}/B$, which has finite codimension in the mixing space for $G/B$.

The inclusion $(G/B)^T \hookrightarrow G/B$ induces a map on cohomology 
\begin{equation} \label{rtfp}
H_T^*(G/B)\rightarrow H_T^*((G/B)^T)= \bigoplus_{w\in W} H_T^*(wB/B) = \bigoplus_{w\in W} 
\IC[x_1,\dots, x_{n}]
\end{equation}
that is known to be injective  \cite{CS}, \cite{GKM}.

Suppose $W=(i_1,\dots, i_\ell)$ is a reduced word sequence for  $w\in W$. 
If $U = (i_{j_1}, \dots, i_{j_d})$ with $\{j_1,\dots, j_d\}\subset \{1, \dots, \ell\}$ and $j_1< \dots < j_d$, we write  $U\subseteq W$. It is possible that $U\subseteq W$ in multiple ways, if $W$ has repeated indices. If $U$ is also is a reduced word sequence for $u = s_{i_{j_1}}\cdots s_{i_{j_d}}$, then clearly $u\leq w$; we say that  $U$ is a {\em reduced word for $u$  occurring as} a subword of $W$.

The image of Schubert class $\sigma _u$ under the map in Equation~\eqref{rtfp} may be computed using the AJS-Billey formula \cite{sbilley}, \cite{AJS}:

\begin{thm}  [\cite{AJS}, \cite{sbilley}, AJS-Billey Restriction Formula] \label{billey}
Given a fixed reduced word sequence  $V = (i_1, i_2,\dots i_{\ell(v)})$ for $v$, define 
$$r(k, V):= s_{i_1} \dots s_{i_{k-1}}(\alpha_{i_k}).$$
For   $U = (i_{j_1},{i_{j_2}},\cdots ,{i_{j_{\ell(u)}}})\subseteq V$, we write
\begin{align*}
\prod_{k\in U} r(k, V) := r(j_1, V)r(j_2, V)\cdots r(j_{\ell(u)}, V).
\end{align*}
Then for any $u,v\in S_n$, 
$$
 \sigma _u |_v = \sum _{U \subseteq V} \prod_{k\in U} r(k, V),
 $$
 where the sum is over reduced words $U$
 occurring as subwords of $V$. 
\end{thm}
An immediate corollary is that $\sigma _u |_v = 0$ unless $u\leq v$.

\subsection{The equivariant cohomology of the Peterson $Y$ and Peterson Schubert classes}\label{sec:Petersonclasses}

The inclusion $S\hookrightarrow T$ given by $z\mapsto (z^n, z^{n-1}, z^{n-2},\dots, z)$ for $z$ a complex number with $|z|=1$, induces a map on Lie algebras, $\mathfrak{s}\rightarrow \mathfrak{t}$ given by 
$$1\mapsto (n, n-1, n-2,\dots, 2, 1).$$  
Using the dual coordinate basis $\{x_j\}$ of $\mathfrak{t}^*$ introduced above, 
the dual map $\mathfrak{t}^*\rightarrow \mathfrak{s}^*$ induced by the inclusion is given by
$x_j\mapsto (n-j+1)t$ for $j=1,\dots, n$, where $t \in \mathfrak s^*$ is the dual coordinate to $1\in \mathfrak s$. 
The inclusion $S\hookrightarrow T$ 

thus induces  a map $H_T^*\rightarrow H_{S}^*$ in which 
$$\alpha_i\mapsto t$$
 for 
 $i= 1,2, \dots, n-1$. This observation justifies the decision to  call 
$b\in H_{S}^*$ {\em positive} if it is a polynomial in $t$ with positive coefficients.

The map on equivariant cohomology in turn induces a map of modules for any $T$-space $X$, which we also denote by $\pi$:
$$
\begin{CD}
H_T^*(X) @>\pi>> H_{S}^*(X).
\end{CD}
$$
When $X=G/B$, this is a surjective map of free modules. 
The $S$-equivariant inclusion $\iota:Y\hookrightarrow G/B$ of the Peterson variety
induces a surjective map:
$$
\begin{CD}
H_{S}^*(G/B) @>\iota^*>> H_{S}^*(Y),
\end{CD}
$$
and these maps naturally commute with the restrictions to fixed points. We thus obtain a commutative diagram:
$$
\begin{CD}
H_T^*(G/B) @>\pi>> H_{S}^*(G/B) @>\iota^*>> H_{S}^*(Y) \\
@VVV @VVV @VVV\\
H_T^*((G/B)^T) @>>>  H_{S}^*((G/B)^{S}) @>\iota^*_{fps} >> H_{S}^*(Y^{S})\\
@| @| @| \\
\bigoplus_{w\in W} H_T^* @>\bigoplus_{w\in W} \pi >>  \bigoplus_{w\in W} H_{S}^* @>>> \bigoplus_{w_A \in S_n} H_{S}^*\\
\end{CD}
$$
where  $\iota ^*$ is the map induced by the inclusion $Y \hookrightarrow G/B$, and $\iota^*_{fps}$ is the induced map from the inclusion of fixed point sets on $Y$ to those on $G/B$. The kernel of $\iota ^*_{fps}$ consists of all copies of $H_{S}^*(wB/B)$ with $wB/B$ not in $Y$, i.e. $w
\neq w_A$ for any $A \subseteq \{1, \ldots, n-1\}$. 

 All vertical maps of the commutative diagram are obtained from the inclusion of fixed point sets. As discussed, the first two vertical maps are injective. In \cite{HT}, the authors prove that the third vertical map is injective, and that $H_S^*(Y)$ is a free module over the equivariant cohomology of a point.

\begin{thm}[\cite{HT}, Thoerem 3.2]\label{th:injection}
Let $S$ act on the Peterson variety $Y$ as described above. Then 
$H_{S}^*(Y)$ is a free module over $H_{S}^*$, and in particular,
$$
H_{S}^*(Y) \simeq H^*(Y) \otimes _{\mathbb C} H_{S}^*.
$$
In addition, the inclusion $Y^{S}\hookrightarrow Y$ induces an {\em injection}
$$
H_{S}^*(Y) \longrightarrow  H_{S}^*(Y^{S}).
$$
\end{thm}

The authors also discovered a basis of $H_{S}^*(Y)$ by mapping a subset of Schubert classes across the vertical arrows of the commuting diagram.  

For any subset $A\subseteq \{1,\dots, n-1\}$, define the {\em Peterson Schubert class} corresponding to $A \subseteq \{1, \ldots, n-1\}$ by
$$
p_A :=\iota^* \circ \pi (\sigma _{v_A}) \in H_{S}^*(Y),
$$ 
where $v_A =s_{a_1} s_{a_2} \cdots s_{a_k}$ with $a_i \in A$ and $a_i <a_j$ whenever $i <j$, and $\sigma _{v_A}\in H_T^*(G/B)$ is the corresponding Schubert class. The degree of $p_A$ is $2\ell(v_A) =2 |A|$.

\begin{thm}[\cite{HT}, Theorem 4.12]
The collection $\{p_A\}_{A \subseteq \{1, \ldots, n-1\}}$ form an $H_{S}^*$-module basis for $H_{S}^*(Y)$. We call this basis the {\em Peterson Schubert basis} of $H_S^*(Y)$. 
\end{thm}

\subsection{Peterson Schubert classes: basic properties}\label{sse:basicproperties}

Here we collect together a number of properties of Peterson Schubert classes, their products,  and their restrictions.

For $A \subseteq \{1, \ldots, n-1\}$   with $j \in A$, $\mathcal T_A(j)$ is the smallest integer in the maximal consecutive subset of $A$ containing $j$, and similarly, $\mathcal H_A(j) $ is the largest integer of the same set. 
Write $A= A_1\cup\cdots\cup A_k$ as a union of maximally consecutive sets. Consider the reduced word  sequence for the longest word $w_{A_i}$ given by 
\begin{align}\label{eq:redwordchoice}
W_{A_i} = (&\mathcal T_A(j),\mathcal T_A(j)+1, \dots, \mathcal H_A(j), \mathcal T_A(j), \mathcal T_A(j)+1,\dots \mathcal H_A(j)-1, \nonumber \\
& \dots, \mathcal T_A(j), \mathcal T_A(j)+1, \mathcal T_A(j)).
\end{align}
Observe that $W_{A_i}$ is independent of $j\in A_i$ since $A_i$ is consecutive. 
One reduced word seequence $W_A$ for $w_A$ is given by the concatenation of sequences $W_{A_i}$ for $i=1, \dots, k$, i.e. $W_A = W_{A_1}W_{A_2}\cdots W_{A_k}$.

 The following restriction formula is a tiny generalization of a formula proved in \cite[Proposition 5.9]{HT}.

\begin{lemm}\label{billeyp1}
Let $\sigma_u\in H_T^*(G/B)$ be a Schubert class and let $w_A$ be the $S$-fixed point of the Peterson variety $Y$ associated to $A\subseteq \{1, \ldots, n-1\}$.  Let $A= A_1\cup\cdots\cup A_k$ be written as a union of maximally consecutive sets, and let $W_A$ be the reduced word sequence for $w_A$ given by the concatenation $W_{A_1}W_{A_2}\cdots W_{A_k}$ of sequences $W_{A_i}$ given in Equation~\eqref{eq:redwordchoice} for $i=1, \dots, k$. 
 Then
\begin{equation}\label{eq:restrictproj}
\iota^* \circ \pi(\sigma_u) |_{w_A} =
\sum _{U} n_{W_A}(U) \left(\prod _{j \in U} (j - \mathcal T_A(j) +1)\right)t^{\ell(u)}
\end{equation}
where the sum is over distinct reduced words $U$ of $u$,  $n_{W_A}(U)$ is the number times the word  $U$ occurs as a subword of $W_A$. 
\end{lemm}

Since the Peterson Schubert class $p_A = \iota^*\circ \pi (\sigma_{v_A})$,   Lemma~\ref{billeyp1} implies the following Corollary.

\begin{coro}[\cite{HT}, Theorem 4.12]\label{co:HTvanishing}
 $p_A|_{w_C} =0$ unless $A\subseteq C$.
\end{coro}
Observe that in the poset of subsets ordered by inclusion, $C=A$ is the minimal subset for which $p_A|_{w_C}$ may not vanish. See Corollary \ref{cor:selfrestrict}.
As a consequence, the structure constants also satisfy support conditions:
\begin{lemm}\label{le:vanishingcoefficientsupport}
Let $A, B, C\subseteq \{1,2, \dots, n-1\}$. Then $b_{A,B}^C \neq 0$ implies $A\cup B\subseteq C$ and $|C|\leq |A|+|B|$.
\end{lemm}
\begin{proof} Assume $A\cup B\not\subseteq C$, then either $A\not\subseteq C$ or $B\not\subseteq C$, 
so the product $p_Ap_B|_{w_C}$ vanishes by Corollary~\ref{co:HTvanishing}. Similarly, $p_D|_{w_C}=0$ unless $D\subseteq C$. Thus
\begin{equation}\label{eq:sumD}
p_Ap_B|_{w_C} =\sum_{D\subseteq C} b_{A,B}^Dp_D|_{w_C}=0.
\end{equation}
Note that $D\subseteq C$ implies $A\cup B\not\subseteq C$, else $A\cup B\subseteq C$.
If $|C|=0$, the sum is over a single term $C=\emptyset$, so $b_{A,B}^C p_C|_{w_C}=0$. However $p_C|_{w_C}\neq 0$ by Lemma~\ref{billeyp1}, so $b_{A,B}^C=0$. Make the inductive assumption that $A\cup B\not \subseteq C$ implies  $b_{A,B}^C =0$ for  $|C|\leq k$. Then for $|C|=k+1$, Equation~\eqref{eq:sumD} may be written
$$
p_Ap_B|_{w_C} =\sum_{D\subsetneq C} b_{A,B}^Dp_D|_{w_C}+ b_{A,B}^C p_C|_{w_C}=0.
$$
If $D$ is a proper subset of $C$, if $|D|\leq k$, and by the inductive assumption, $b_{A,B}^D=0$. Thus as before, we conclude $b_{A,B}^C p_C|_{w_C}=0$ and, since $p_C|_{w_C}\neq 0$ that $b_{A,B}^C=0$. Since $\deg(p_Ap_B)=|A|+|B|$ (as a polynomial), each summand $b_{A,B}^Cp_C$ in the product $p_A p_B$ has degree $|A|+|B|$, and therefore $b_{A,B}^C\neq 0$ implies that $|C|=\deg(p_C) \leq |A|+|B|$.
  \end{proof}

Lemma~\ref{billeyp1} also implies that the restrictions of Peterson Schubert classes remain constant when nonconsecutive elements are added to a fixed point.

\begin{coro} \label{nconsec}
Let $A \subseteq C^0$ with $C^0$ consecutive, and let $C\supset C^0$ be any set so that $C\setminus C^0$ is not consecutive with $C^0$.  Then
$$
p_A |_{w_{C^0}} = p_A |_{w_{C}}.
$$
\end{coro}
\begin{proof}
Let $A = \{a_1,\dots, a_k\}$ with $a_i< a_j$ for $i<j$. There is only one reduced word decomposition $v_A = s_{a_1}s_{a_1}\dots  s_{a_k}$ and thus one reduced word sequence $V_A = (a_1, \dots, a_k)$.
Neither $\mathcal T_{C^0} - 1$ nor $\mathcal H_{C^0}+1$ are in $C$,  so we may choose $W_{C} = W_{C^0} W_{C\setminus C^0}$ for some choice $W_{C\setminus C^0}$. Lemma~\ref{billeyp1} therefore implies
\begin{align*}
p_A |_{w_{C^0}} &= n_{W_{C^0}}(V_A) \left(\prod_{j\in V_A} (j-\mathcal T_{C^0}(j)+1)\right)t^{|A|},\mbox{ and}\\
 p_A |_{w_{C} }&= n_{W_{C} }(V_A)\left( \prod_{j\in V_A}  (j-\mathcal T_{C} (j)+1)\right)t^{|A|}.
\end{align*}
 As $A\subseteq C^0$ and $W_{C} = W_{C^0} W_{C\setminus C^0}$,
$n_{W_{C}}(V_A)=n_{W_{C^0}}(V_A)$.  Note that the product over the entries $j$ of $V_A$ consists of a single factor for each $j\in A$. Furthermore, $j\in A$ implies $\mathcal T_{C}(j)=\mathcal T_{C^0}(j)$ since $C$ does not contain $\mathcal T_{C^0}-1$. Thus the products have identical factors.
  \end{proof}

\begin{lemm}\label{le:vanishingb}
Suppose $A\cup B\subseteq C^0$ (not necessarily consecutive) and $C\supset C^0$ is any set so that $C\setminus C^0$ is nonempty and not consecutive with $C^0$.  Then $b_{A,B}^{C}=0$.
\end{lemm}
\begin{proof}
By Corollary~\ref{nconsec},
$p_{A}p_{B}|_{w_C}= p_{A}p_{B}|_{w_{C^0}}.$
Since the restrictions are the same, 
$$
 \sum_{D\subseteq C} b_{A,B}^{D} p_{D}|_{w_C}=\sum_{D\subseteq C^0} b_{A,B}^{D} p_{D}|_{w_{C^0}}
$$
and in particular also by Corollary~\ref{nconsec}, 
\begin{equation}\label{eq:sumsto0}
 \sum_{D: D\subsetneq C^0, D \subseteq C } b_{A,B}^{D} p_{D}|_{w_C} =0.
\end{equation}

We proceed inductively on $|C'|$. If $C'=\{m\}$ consists of one element, the sum is over one set $D=C$, so $b_{A,B}^{C} p_{C}|_{w_C} =0$. Since $p_c|_{w_C}\neq 0$, we conclude $b_{A,B}^{C} =0$. More generally, the sum \eqref{eq:sumsto0} is
$$
 \sum_{D: C^0\subsetneq D \subsetneq C} b_{A,B}^{D} p_{D}|_{w_C} + b_{A,B}^{C} p_{C}|_{w_C} =0
 $$
 where the first sum is 0 by the inductive assumption. Thus  $b_{A,B}^{C} =0.$
  \end{proof}

Lemma~\ref{billeyp1} also implies an easy formula for the restriction of any Peterson Schubert class $p_A$ to its minimal fixed point $w_A$. 
\begin{coro}\label{cor:selfrestrict} Let $A$ be consecutive. Then 
$$
p_A |_{w_A} =  |A|!\ t^{|A|}.
$$
\end{coro} 
\begin{proof}
We calculate directly using the Peterson Schubert restriction formula.
$$p_A |_{w_{A} }= \iota^*\pi(\sigma_{v_A}|_{w_A })=n_{W_A}(V_A) \left(\prod_{j\in A} (j-\mathcal T_A+1)\right)t^{|A|}.
$$
Then $V_A$ occurs in $W_A$ exactly one time, so the restriction is
$$
 \prod_{j\in A} (j-\mathcal T_A+1)t = |A|!\ t^{|A|}.$$

\vspace{-.4in}
 \end{proof}

A fundamental observation is that $p_{A\cup B} = p_A p_B$ when $A$ and $B$ are disjoint strings of consecutive integers separated by at least one number. 

\begin{lemm} [\cite {HT}, Lemma 6.7]\label{breakup}
Let $A \subseteq \{1, \ldots, n-1\}$ and suppose 
$$A = A_1 \cup A_2 \cup \ldots \cup A_k$$
 where each $A_i$ is a nonempty maximal consecutive string of integers and $A_i \neq A_j$ for $i\neq j$.  Then
$$
p_A = \prod _{1 \leq i \leq k} p_{A_i}.
$$
\end{lemm}

\section{Proof of main theorems and lemmas}\label{se:proofs}
Here we prove Theorems \ref{general}, \ref{thm:AuBconsec}, \ref{thm:AuBnotconsec}, \ref{thm:disconnectedgeneral}, and \ref{thm:positive}.  There are two substantial cases required to prove Theorem~\ref{general}, recalling that $A$ and $B$ are consecutive by hypothesis.
In the first case,
either $A\cap B$ is nontrival but neither set contains the other, or the two sets are consecutive to each other. In the second case, one set is contained in the other.

\begin{defi} 
Let $A$ and $B$ be consecutive sequences of $\{1,2,\dots, n-1\}$. We say that $A$ and $B$ are {\em intertwined} if $\mathcal T_A \leq \mathcal T_B \leq \mathcal H_A \leq \mathcal H_B$ or $\mathcal T_B\leq \mathcal T_A \leq \mathcal H_B \leq \mathcal H_A$.
\end{defi}

\begin{lemm}\label{lem:restrictAuB} Suppose $A, B$  and $A\cup B$ are consecutive. Then
$$
p_A |_{w_{A\cup B}} = {\mathcal H_{A\cup B} -\mathcal T_A +1 \choose |A|}\frac{(\mathcal H_A - \mathcal T_{A\cup B} +1)!}{(\mathcal T_A-\mathcal T_{A\cup B})!} t^{|A|} .
$$ 
In particular, if $A$ and $B$ are intertwined or if $A$ and $B$ are consecutive to each other and nonintersecting,
$$
p_A |_{w_{A\cup B}} =  \frac{|A\cup B|!}{|B\setminus A|!} t^{|A|}.
$$
\end{lemm}

\begin{proof}
According to Lemma \ref{billeyp1},
\begin{equation}\label{restrict}
p_A |_{w_{A\cup B} }= 
n_{W_{A\cup B}}(V_A)\left( \prod_{j\in V_A} (j-\mathcal T_{A\cup B}(j)+1)\right)t^{|A|}.
\end{equation}
We claim that $n_{W_{A\cup B}}(V_A) =  {\mathcal H_{A\cup B} -\mathcal T_A +1 \choose |A|}$, where  $V_A= (\mathcal T_A,\mathcal T_A +1,\cdots ,\mathcal H_A)$.
Choose the reduced decomposition $W_{A\cup B}$ of $w_{A\cup B}$ given by the sequence (read from left to right and top to bottom) in Figure~\ref{fig:reducedwords} (left panel), ignoring the grid and path within.

Each increasing consecutive string of $W_{A \cup B}$ is written on its own line, all left aligned. 
 All rows finishing in numbers $\mathcal H_A$ or larger contain the string $\mathcal T_A \mathcal T_A +1 \cdots \mathcal H_A$.  To count the number of occurrences of $V_A$ in this product, we first
 draw a grid around all of these strings except for the one appearing in the first row. The grid has $\mathcal H_{A \cup B} - \mathcal H_A$ rows and $\mathcal H_A - \mathcal T_A +1$ columns.

For example, suppose $A = \{2, 3\}$ and $A\cup B = \{1, \dots, 6\}$. Let $w_{A\cup B}$ be the longest word for the permutations group generated by $\{s_i: i\in A\cup B\}$. Then
$$
W_{A\cup B}= (1, 2, 3, 4, 5, 6, 1, 2, 3, 4, 5, 1, 2, 3, 4, 1, 2, 3, 1, 2, 1),
$$ and $V_A = (2, 3)$. We have the grid containing the $2$ and $3$ in the second, third and fourth rows of $W_{A\cup B}$, pictured in Figure~\ref{fig:reducedwords}, (right panel).

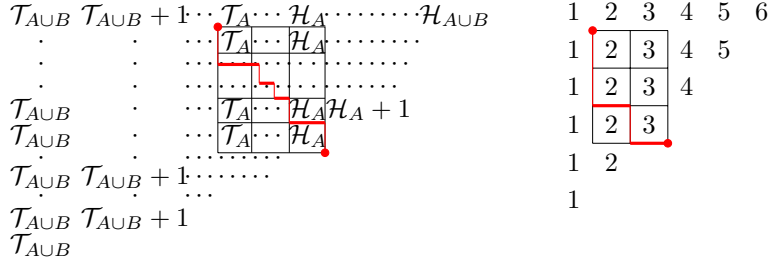
\begin{figure}\label{fig:reducedwords}\caption{Finding reduced words $V_A$ occurring in $W_A$ (left panel) and an example (right panel)}
\begin{tikzpicture}[scale=0.50]
\node at (-2.5, 3.5) {$\mathcal T_{A\cup B}$};
\node at (0, 3.5) {$\mathcal T_{A\cup B}+1$};
\node at (1.4, 3.5) {$\cdot$};
\node at (1.7, 3.5) {$\cdot$};
\node at (2, 3.5) {$\cdot$};
\node at (2.7, 3.5) {$\mathcal T_A$};
\node at (3.2, 3.5) {$\cdot$};
\node at (3.5, 3.5) {$\cdot$};
\node at (3.8, 3.5) {$\cdot$};
\node at (4.6, 3.5) {$\mathcal H_A$};
\node at (5.1, 3.5) {$\cdot$};
\node at (5.4, 3.5) {$\cdot$};
\node at (5.7, 3.5) {$\cdot$};
\node at (6, 3.5) {$\cdot$};
\node at (6.3, 3.5) {$\cdot$};
\node at (6.6, 3.5) {$\cdot$};
\node at (6.9, 3.5) {$\cdot$};
\node at (7.2, 3.5) {$\cdot$};
\node at (7.5, 3.5) {$\cdot$};
\node at (8.5, 3.5) {$\mathcal H_{A\cup B}$};
\node at (-2.5, 2.8) {$\cdot$};
\node at (0, 2.8) {$\cdot$};
\node at (1.4, 2.8) {$\cdot$};
\node at (1.7, 2.8) {$\cdot$};
\node at (2, 2.8) {$\cdot$};
\node at (2.7, 2.8) {$\mathcal T_A$};
\node at (3.2, 2.8) {$\cdot$};
\node at (3.5, 2.8) {$\cdot$};
\node at (3.8, 2.8) {$\cdot$};
\node at (4.6, 2.8) {$\mathcal H_A$};
\node at (5.1, 2.8) {$\cdot$};
\node at (5.4, 2.8) {$\cdot$};
\node at (5.7, 2.8) {$\cdot$};
\node at (6, 2.8) {$\cdot$};
\node at (6.3, 2.8) {$\cdot$};
\node at (6.6, 2.8) {$\cdot$};
\node at (6.9, 2.8) {$\cdot$};
\node at (7.2, 2.8) {$\cdot$};
\node at (7.5, 2.8) {$\cdot$};
\node at (-2.5, 2.2) {$\cdot$};
\node at (0, 2.2) {$\cdot$};
\node at (1.4, 2.2) {$\cdot$};
\node at (1.7, 2.2) {$\cdot$};
\node at (2, 2.2) {$\cdot$};
\node at (2.3, 2.2) {$\cdot$};
\node at (2.6, 2.2) {$\cdot$};
\node at (2.9, 2.2) {$\cdot$};
\node at (3.2, 2.2) {$\cdot$};
\node at (3.5, 2.2) {$\cdot$};
\node at (3.8, 2.2) {$\cdot$};
\node at (4.1, 2.2) {$\cdot$};
\node at (4.4, 2.2) {$\cdot$};
\node at (4.75, 2.2) {$\cdot$};
\node at (5.1, 2.2) {$\cdot$};
\node at (5.4, 2.2) {$\cdot$};
\node at (5.7, 2.2) {$\cdot$};
\node at (6, 2.2) {$\cdot$};
\node at (6.3, 2.2) {$\cdot$};
\node at (6.6, 2.2) {$\cdot$};
\node at (6.9, 2.2) {$\cdot$};
\node at (7.2, 2.2) {$\cdot$};
\node at (-2.5, 1.6) {$\cdot$};
\node at (0, 1.6) {$\cdot$};
\node at (1.4, 1.6) {$\cdot$};
\node at (1.7, 1.6) {$\cdot$};
\node at (2, 1.6) {$\cdot$};
\node at (2.3, 1.6) {$\cdot$};
\node at (2.6, 1.6) {$\cdot$};
\node at (2.9, 1.6) {$\cdot$};
\node at (3.2, 1.6) {$\cdot$};
\node at (3.5, 1.6) {$\cdot$};
\node at (3.8, 1.6) {$\cdot$};
\node at (4.1, 1.6) {$\cdot$};
\node at (4.4, 1.6) {$\cdot$};
\node at (4.75, 1.6) {$\cdot$};
\node at (5.1, 1.6) {$\cdot$};
\node at (5.4, 1.6) {$\cdot$};
\node at (5.7, 1.6) {$\cdot$};
\node at (6, 1.6) {$\cdot$};
\node at (6.3, 1.6) {$\cdot$};
\node at (6.6, 1.6) {$\cdot$};
\node at (6.9, 1.6) {$\cdot$};
\node at (-2.5, 1) {$\mathcal T_{A\cup B}$};
\node at (0, 1) {$\cdot$};
\node at (1.4, 1) {$\cdot$};
\node at (1.7, 1) {$\cdot$};
\node at (2, 1) {$\cdot$};
\node at (2.7, 1) {$\mathcal T_A$};
\node at (3.2, 1) {$\cdot$};
\node at (3.5, 1) {$\cdot$};
\node at (3.8, 1) {$\cdot$};
\node at (4.6, 1) {$\mathcal H_A$};
\node at (6.2, 1) {$\mathcal H_A+1$};
\node at (-2.5, .3) {$\mathcal T_{A\cup B}$};
\node at (0, .3) {$\cdot$};
\node at (1.4, .3) {$\cdot$};
\node at (1.7, .3) {$\cdot$};
\node at (2, .3) {$\cdot$};
\node at (2.7, .3) {$\mathcal T_A$};
\node at (3.2, .3) {$\cdot$};
\node at (3.5, .3) {$\cdot$};
\node at (3.8, .3) {$\cdot$};
\node at (4.6, .3) {$\mathcal H_A$};
\node at (-2.5, -.3) {$\cdot$};
\node at (0, -.3) {$\cdot$};
\node at (1.4, -.3) {$\cdot$};
\node at (1.7, -.3) {$\cdot$};
\node at (2,-.3) {$\cdot$};
\node at (2.3, -.3) {$\cdot$};
\node at (2.6, -.3) {$\cdot$};
\node at (2.9,-.3) {$\cdot$};
\node at (3.2, -.3) {$\cdot$};
\node at (3.5, -.3) {$\cdot$};
\node at (3.8, -.3) {$\cdot$};
\node at (-2.5, -.8) {$\mathcal T_{A\cup B}$};
\node at (0, -.8) {$\mathcal T_{A\cup B}+1$};
\node at (1.4, -.8) {$\cdot$};
\node at (1.7, -.8) {$\cdot$};
\node at (2, -.8) {$\cdot$};
\node at (2.3, -.8) {$\cdot$};
\node at (2.6, -.8) {$\cdot$};
\node at (2.9, -.8) {$\cdot$};
\node at (3.2, -.8) {$\cdot$};
\node at (3.5, -.8) {$\cdot$};
\node at (-2.5, -1.3) {$\cdot$};
\node at (0, -1.3) {$\cdot$};
\node at (1.4, -1.3) {$\cdot$};
\node at (1.7, -1.3) {$\cdot$};
\node at (2,-1.3) {$\cdot$};
\node at (-2.5, -1.9) {$\mathcal T_{A\cup B}$};
\node  at (0, -1.9) {$\mathcal T_{A\cup B}+1$};
\node at (-2.5,-2.6) {$\mathcal T_{A\cup B}$};

\draw[very thin] (2.2,-.15)--(5.05, -.15)--(5.05, 3.2)--(2.2,3.2)--(2.2,-.15);
\draw[very thin] (2.2,.65)--(5.05, .65);
\draw[very thin] (2.2,1.3)--(5.05, 1.3);
\draw[very thin] (2.2,2.5)--(5.05, 2.5);
\draw[very thin] (3.1,-.15)--(3.1, 3.2);
\draw[very thin] (4.1,-.15)--(4.1, 3.2);
\draw[color=red,fill=red] (2.2,3.2) circle (.1cm);
\draw[color=red,fill=red] (5.05,-.15) circle (.1cm);
\draw[color=red] (2.2,3.2)--(2.2,2.2); 
\draw[color=red,very thick] (2.2,2.2)--(3.3, 2.2); 
\draw[color=red] (3.3,2.2)--(3.3, 1.7); 
\draw[color=red, very thick] (3.3,1.7)--(3.7, 1.7);
\draw[color=red] (3.7, 1.7)--(3.7, 1.3); 
\draw[color=red, very thick] (3.7, 1.3)--(4.1,1.3); 
\draw[color=red] (4.1,1.3)--(4.1,.65);
\draw[color=red,very thick] (4.1,.65)--(5.05,.65);
\draw[color=red] (5.05,.65)--(5.05,-.15);
\end{tikzpicture}
\qquad
\begin{tikzpicture}[scale=0.50]
\node at (-2.5, -2) {};
\node at (-2.5, -.5) {$1$};
\node at (-2.5, 0.5) {$1$};
\node at (-2.5, 1.5) {$1$};
\node at (-2.5, 2.5) {$1$};
\node at (-2.5, 3.5) {$1$};
\node at (-2.5, 4.5) {$1$};
\node at (-1.5, .5) {$2$};
\node at (-1.5, 1.5) {$2$};
\node at (-1.5, 2.5) {$2$};
\node at (-1.5, 3.5) {$2$};
\node at (-1.5, 4.5) {$2$};
\node at (-.5, 1.5) {$3$};
\node at (-.5, 2.5) {$3$};
\node at (-.5, 3.5) {$3$};
\node at (-.5, 4.5) {$3$};
\node at (0.5, 2.5) {$4$};
\node at (0.5, 3.5) {$4$};
\node at (0.5, 4.5) {$4$};
\node at (1.5, 3.5) {$5$};
\node at (1.5, 4.5) {$5$};
\node at (2.5, 4.5) {$6$};
\draw[step=1cm,black,very thin] (-2,1) grid (0,4);
\draw[color=red,fill=red] (-2,4) circle (.1cm);
\draw[color=red,fill=red] (0,1) circle (.1cm);
\draw[color=red] (-2, 4) -- (-2, 2);
\draw[color = red, very thick] (-2,2) -- (-1, 2);
\draw[color=red] (-1, 2)--(-1, 1);
\draw[color = red, very thick] (-1, 1)--(0,1); 
\end{tikzpicture}
\end{figure}

There is a one-to-one correspondence between paths from the top left corner to the bottom right corner of this grid (moving only right and down) and occurrences of $V_A$ inside of $W_A$.  Each instance of $V_A$ inside of $W_{A \cup B}$ is  ``underlined'' by the horizontal components of a path,  as indicated with the red path in Figure~\ref{fig:reducedwords} (left panel).  For example, $V_A$ is given by the subset of $W_{A \cup B}$ underlined by the path in Figure~\ref{fig:reducedwords} (right panel), it selects the subset indicated by boxed elements:
$$(1, 2, 3, 4, 5, 6, 1, 2, 3, 4, 5, 1, \boxed{2}, 3, 4, 1, 2, \boxed{3}, 1, 2, 1).
$$

The dimensions of the grid are $(\mathcal H_{A \cup B} - \mathcal H_A) \times (\mathcal H_A - \mathcal T_A +1)$ and hence the number of reduced words for $v_A$ inside of $W_{A \cup B}$ is the count of such paths, known to be the number of ``right" (or ``down") moves among the total moves given by the sum of the row and column lengths. Therefore,
$$
n_{W_{A\cup B}}(V_A)={\mathcal H_{A \cup B} - \mathcal T_{A} + 1 \choose \mathcal H_A - \mathcal T_A +1} = {\mathcal H_{A \cup B} - \mathcal T_{A} + 1 \choose |A|}.
$$

We turn our attention to the factor $\left(\prod_{j\in V_A} (j-\mathcal T_{A\cup B}(j)+1)\right)t^{|A|}$ in Equation~\ref{restrict}. Since $A\cup B$ is consecutive and the product is over $|A|$ elements with the highest $j$ occurring at $j=\mathcal H_A$, but only descending $|A|$ terms:

\begin{align*}
\left(\prod_{j\in V_A} (j-\mathcal T_{A\cup B}+1)\right)t^{|A|} &=\frac{(\mathcal H_A-\mathcal T_{A\cup B} +1)!}{(\mathcal H_A-\mathcal T_{A\cup B} +1-|A|)!} t^{|A|} \\
 &= \frac{(\mathcal H_A-\mathcal T_{A\cup B} +1)!}{(\mathcal T_A-\mathcal T_{A\cup B})!} t^{|A|}.
 \end{align*}
We put the two terms together to get the formula. 

If  $A$ and $B$ are intertwined, then if
 $\mathcal T_A=\mathcal T_{A\cup B}$ and $\mathcal H_{B}= \mathcal H_{A\cup B}$, 
 $$
{\mathcal H_{A\cup B} -\mathcal T_A +1 \choose |A|}={|A\cup B|\choose |A|}, 
\qquad
\frac{(\mathcal H_A -\mathcal T_{A\cup B} +1)!}{(\mathcal T_A-\mathcal T_{A\cup B})!}=|A|!
$$
so the product is $\frac{|A\cup B|!}{(|A\cup B|-|A|)!}= \frac{|A\cup B|!}{|B\setminus A|!}$.
If $\mathcal T_B=\mathcal T_{A\cup B}$ and $\mathcal H_{A}= \mathcal H_{A\cup B}$, 
$$
{\mathcal H_{A\cup B} -\mathcal T_A +1 \choose |A|}={|A|\choose |A|} =1,\qquad
\frac{(\mathcal H_A -\mathcal T_{A\cup B} +1)!}{(\mathcal T_A-\mathcal T_{A\cup B})!}=\frac{|A\cup B|!}{|B\setminus A|!},
$$
resulting in the same product.  
  \end{proof}

The following Lemma serves as the base case for an inductive argument in the proof of Theorem~\ref{general}.

\begin{lemm}\label{le:cor:intertwined} Suppose $A, B$  are consecutive. When $A$ and $B$ are intertwined, or when $A$ and $B$ are consecutive to each other and nonintersecting,
$$
b_{A,B}^{A\cup B}= \frac{|A\cup B|!}{|B\setminus A|!|A\setminus B|!}t^{|A\cap B|}.
$$
\end{lemm}
\begin{proof} 
Restrict the product
$$
(p_A p_B)|_{w_{A\cup B}} = \sum_{C: A\cup B\subseteq C}  b_{A,B}^C p_C|_{w_{A\cup B}} = b_{A,B}^{A\cup B} p_{A\cup B}|_{w_{A\cup B}},
$$
since $p_C|_{w_{A\cup B}}=0$ unless $C\subseteq A\cup B$.  
 By Lemma~\ref{lem:restrictAuB},
\begin{align*}
p_A|_{w_{A\cup B}} p_B|_{w_{A\cup B}} & = \frac{|A\cup B|!}{|B\setminus A|!} \frac{|A\cup B|!}{|A\setminus B|!} t^{|A|+|B|}.
\end{align*}
By Corollary~\ref{cor:selfrestrict}, $p_{A\cup B}|_{w_{A\cup B}} = |A\cup B|!t^{|A\cup B|}$. We then solve:
$$
b_{A,B}^{A\cup B}= \frac{1}{|A\cup B|! t^{|A\cup B|}}  \frac{|A\cup B|!}{|B\setminus A|!} \frac{|A\cup B|!}{|A\setminus B|!} t^{|A|+|B|} = \frac{|A\cup B|!}{|B\setminus A|!|A\setminus B|!}t^{|A\cap B|}.
$$

\vspace{-.2in}
  \end{proof}

When $B\subseteq A$, the structure constant $b_{A,B}^{C}$ can be recast in terms of another structure constant with intertwined sets.
\begin{lemm}\label{le:bconversion} Suppose $A, B$ are consecutive and $C$ any set with $B\subseteq A\subseteq C$. Then
$$
|A|!\ |B|!\  b_{A,B}^{C} = |A'|!\ |B'|!\  b_{A', B'}^{C},
$$
where $A' = \{a\in A:\ a\leq \mathcal H_B\}$ and $B' =\{b\in A:\ b\geq \mathcal T_B\}$.
\end{lemm}
\begin{proof}
We show that 
\begin{equation}\label{eq:prodintersectunion}
|A|!\ |B|!\ p_A p_B =  |A'|!\ |B'|!\ \ p_{A'} p_{B'},
\end{equation}
 which implies that the coefficients have the desired relationship since $\{p_C\}$ forms a basis of $H_S^*(Y)$.

By Lemma~\ref{lem:restrictAuB} if $C$ is consecutive, and Lemma~\ref{nconsec} otherwise, 
\begin{align*}
p_A|_{w_{C}} &=  {\mathcal H_C - \mathcal T_A +1\choose |A|}\frac{(\mathcal H_A - \mathcal T_C +1)!}{(\mathcal T_A - \mathcal T_C)!} t^{|A|}  \\
&= \frac{(\mathcal H_C - \mathcal T_A +1)!}{|A|!\ (H_C-H_A)!}\frac{(\mathcal H_A - \mathcal T_C +1)!}{(\mathcal T_A - \mathcal T_C)!} t^{|A|}
\end{align*}
with a similar formula for $p_B|_{w_C}$. 
Using the relationships 
\begin{equation}\label{eq:intersectunion}
\mathcal T_{A'}=\mathcal T_A, \quad \mathcal H_{A'}=\mathcal H_B,\quad \mathcal T_{B'}=\mathcal T_B, \quad \mathcal H_{B'}=\mathcal H_A
\end{equation}
and simplifying as above,
$$
p_{A'}|_{w_{C}} =  \frac{(\mathcal H_C - \mathcal T_A +1)!}{|A'|!\ (\mathcal H_C-\mathcal H_B)!}\frac{(\mathcal H_B - \mathcal T_C +1)!}{(\mathcal T_A - \mathcal T_C)!} t^{|A'|} 
$$
with a similar formula for $p_{B'}|_{w_C}$.
Since $|A|+|B|=|A'|+|B'|$, we conclude
\begin{equation}\label{eq:unionintersectrestrict}
|A|!\ |B|!\ p_A|_{w_C} p_B|_{w_C} =  |A'|!\ |B'|!\ \ p_{A'}|_{w_C}  p_{B'}|_{w_C} 
\end{equation}
for all sets $C$ containing $A\supseteq B$. By 
Theorem~\ref{th:injection}, the equality at every fixed point implies Equation \eqref{eq:prodintersectunion} holds.~  \end{proof}

Finally, we state the crucial lemma for the proof of Theorem~\ref{general}. 

\begin{defi}
Let $A, B,$ and $A \cup B$ be consecutive.  Let $D= A\cup B$ and define 
$$
_i D _j := \{\mathcal T_{A \cup B} - i, \mathcal T_{A \cup B} - i+1, \ldots, \mathcal H_{A \cup B} + j-1, \mathcal H_{A \cup B}+j\}.
$$
\end{defi}

\begin{lemm}\label{le:TATBHAHB} Let $A, B$ and ${_m D _n}$ be consecutive for $m=0,1,\dots, \ell$, $n=0, 1\dots, r$, with $D={_0D_0}=A\cup B$, and $|A\cap B|=\ell+r$. If $A$ and $B$ are intertwined or if $A$ and $B$ are consecutive to each other and disjoint, 
$$
b_{A,B}^{_mD _n} =
\frac{
|A\cup B|!
|A\cap B|!
}
{(|A\cap B|-m-n)!\
m!\ 
n!\ 
(|A\setminus B| +m)!\ 
(|B\setminus A| +n)!
}
t^{|A\cap B|-m-n}. 
$$
\end{lemm}

\begin{proof}
When $m=n=0$, this formula is the statement of Lemma~\ref{le:cor:intertwined}.

We prove this by induction  on $m+n$. For ease of notation, let $K$ denote $_mD_n$.
Restrict $p_A p_B = \sum_C b_{A,B}^C p_C$  to $w_K$:
\begin{equation}\label{eq:b_AB^D}
b_{A,B}^{K} p_{K}|_{w_K}   = p_A|_{w_K} p_B|_{w_{K}} -\sum_{\substack{0\leq i\leq m, 0\leq j \leq n\\ i+j< m+n}} b_{A,B}^{_iD_j} p_{_iD_j}  |_{w_K}.
\end{equation}
For all $0 \leq i \leq m$, $0 \leq j \leq n$ and $i+j <m +n$ assume 
$$
b_{A,B}^{_i D _j} =
\frac{
|A\cup B|!
|A\cap B|!
}
{(|A\cap B|-i-j)!\
i!\ 
j!\ 
(|A\setminus B| +i)!\ 
(|B\setminus A| +j)!
}t^{|A\cap B|-i-j}.
$$

Assume without loss of generality that $\mathcal T_A\leq \mathcal T_B$. Then if $A$ or $B$ are intertwined or disjoint and consecutive to each other,
 $|A\cap B| = \mathcal H_A-\mathcal T_B+1$ and $|A\cup B| =  \mathcal H_B-\mathcal T_A+1$. Using Lemma~\ref{lem:restrictAuB} for each restriction, 
$$
p_A|_{w_{K}} = {\mathcal H_K -\mathcal T_A+1\choose|A|}\frac{(\mathcal H_A - \mathcal T_K +1)!}{(\mathcal T_A-\mathcal T_K)!}t^{|A|}={\mathcal |A\cup B| + n \choose |A|}\frac{(|A| +m)!}{m !}t^{|A|}.
$$
By the inductive assumption, Equation~\eqref{eq:b_AB^D} becomes
\begin{align*} 
b_{A,B}^{K} & p_{K} |_{w_{K}}  = 
 {|A\cup B| + n \choose |A|}\frac{(|A| +m)!}{m !}
{ |B| + n \choose |B|}\frac{(|A \cup B| + m)!}{(|A \setminus B| + m) !} 
 t^{|A|+|B|}\\
 & -\sum_{\substack{
0 \leq i \leq m \\
0 \leq j \leq n \\
i+j < m +n}}\bigg[\frac{|A\cup B|! |A\cap B|!}{i! j! (|A\setminus B| +i)!(|B\setminus A| +j)!
 (|A\cap B| -i -j)!}
t^{|A\cap B| -i -j}\\
&\phantom{-\sum_{i+j <m +n}}
\cdot {|A\cup B| +i +n \choose \mathcal |A\cup B| +i +j}
\frac{
(|A\cup B| + m +j)!
}
{
(m - i)! 
} t^{|A\cup B|+i+j}\bigg]. 
\end{align*}

The  coefficient of $ t^{|A|+|B|}$ of the first term of the sum on the right hand side of this equation is:
\begin{align*}
& {|A\cup B|+ n \choose |A|}\frac{(|A| +m)!}{m !}
{ |B| + n \choose |B|}\frac{(|A \cup B| + m)!}{(|A \setminus B| + m)!} 
\\
 & = {|A\cup B|! \ |A\cap B|!} {|A\cup B| +m \choose \mathcal |A\cup B|}
{|A|+ m \choose \mathcal |A\cap B|} 
{|B| +n \choose |B|}
{|A\cup B| + n \choose |A|}.
\end{align*}
Divide both sides of the equation by $ {|A\cup B|! \ |A\cap B|!}$, so the the right hand side becomes $t^{|A|+|B|}$ times the coefficient
\begin{align*}
&{|A\cup B| +m \choose \mathcal |A\cup B|}
{|A|+ m \choose \mathcal |A\cap B|} 
{|B| +n \choose |B|}
{|A\cup B| + n \choose |A|}-\\ 
&\sum_{\substack{
0 \leq i \leq m \\
0 \leq j \leq n \\
i+j < m +n}}\!\!\!
{{|A\cup B| +i +n} \choose {\mathcal |A\cup B| +i +j}} 
{{|A\cup B| +m + j} \choose {i, j, |A\setminus B| +i, |B\setminus A|+j, |A\cap B| - i - j,  m - i}}.
\end{align*}
Leting $x=|A\cap B|$, $w=|A\cup B|$, $y = |A|$ and $z=|B|$, this expression is
\begin{align*}
&{w +m \choose w}
{y+ m \choose x}
{z +n \choose z}
{w+ n \choose y}-\\ 
&\sum_{\substack{
0 \leq i \leq m \\
0 \leq j \leq n \\
i+j < m +n}}
{w +i +n \choose w +i +j}
{w +m + j \choose i, \ j, \ (y-x)+i, \ (z-x)+j, \ x- i - j, \ m - i},
\end{align*}
which we recognize as the term with $i=m,$ and $j=n$ of the sum on the right hand side of Theorem~\ref{identity}. Using these variables and substituting the right hand side of Theorem~\ref{identity}, we obtain
\begin{equation}\label{simp}
\frac{1}{x!}\frac{1}{w!}b_{A,B}^{K}p_K|_{w_K} = {w+m+n \choose m, n, x-m-n, z-x+n, y-x+m} t^{|A|+|B|}.
\end{equation}
On the other hand, 
\begin{align*}
p_K |_{w_K} = |K|! t^{|K|}=(|A\cup B|+m + n)! t^{|A\cup B|+m + n},
\end{align*} 
so that Equation~\eqref{simp} is
\begin{align*}
\frac{(w+m + n)!} {w! \ x!}& t^{w+m + n} \ b_{A,B}^{K}
= {w+m+n \choose m, \ n, \ x-m-n, \ z-x+n, \ y-x+m} t^{y+z}.
\end{align*}
Finally, we solve for $ b_{A,B}^{K}$ 
and substitute back for $x, y, w, z$ to obtain
$$
b_{A,B}^{K} = \frac{
|A\cup B|!\ 
|A\cap B|!
}
{
m !\  n!\ 
(|A\setminus B|+m)!
(|B\setminus A| +n)!
(|A\cap B| - m - n)!
}
t^{|A\cap B|- m - n}.
\quad 
$$
 \end{proof}

\begin{proof}[of Theorem \ref{general}.]
Assume $A$, $B$, and $C$ are consecutive and that $A \cup B \subseteq C$ with $|C| \leq |A| + |B|$.  Without loss of generality, assume also that $\mathcal T_A\leq \mathcal T_B$.

If $A$ and $B$ are disjoint, then $C$ consecutive and $|C|\leq |A|+|B|$ forces $C= A\cup B$ and thus $A$ and $B$ are adjacent. 
If either $A$ and $B$ are intertwined, or if $A$ and $B$ are adjacent and disjoint, 
\begin{align*}
|A \cup B| &= \mathcal H_B - \mathcal T_A +1 & | A \cap B| &= \mathcal H_A - \mathcal T_B +1 \\
|A \setminus B| &= \mathcal T_B - \mathcal T_A & |B \setminus A| &= \mathcal H_B - \mathcal H_A. \\
\end{align*}
As $C$ is consecutive, $C={_mD_n}$  where $m   = \mathcal T_A - \mathcal T_C$ and $n  = \mathcal H_C -  \mathcal H_B$.
It follows that 
$|A \setminus B|+m =  \mathcal  T_B- \mathcal T_C$
and
$ |B \setminus A| +n= \mathcal H_C-\mathcal H_A$.
Then by Lemma~\ref{le:TATBHAHB} with $d:=  |A| + |B| -|C| =  |A \cap B| - m - n$,
$$
b_{A, B}^C = 
\frac{(\mathcal H_A - \mathcal T_B +1)! 
(\mathcal H _B - \mathcal T_A +1)!}
{d !
(\mathcal T_A - \mathcal T_C)!
(\mathcal H_C - \mathcal H_B)!
(\mathcal T_B - \mathcal T_C)!
(\mathcal H_C -\mathcal H_A)!}
t^d.
$$

To prove the case when $B\subseteq A$ we construct two intertwined sets from $A$ and $B$ and apply Lemma~\ref{le:bconversion}.  Let 
$$A' := \{a\in A:\ a\leq \mathcal H_B\} \text{  and  }B' := \{b\in A:\ b\geq \mathcal T_B\}.$$
Then $A'$ and $B'$ are intertwined and also satisfy the relationships in \eqref{eq:intersectunion} with 
  $$d = |A'| + |B'| - |C| = |A| + |B| -|C|.$$
Furthermore, 
$\mathcal T_{A'} = \mathcal T_{A}$, $\mathcal T_{B'}=\mathcal T_B$, $\mathcal H_{B'}=\mathcal H_A$, and $ \mathcal H_{A'}=\mathcal H_B$.
Thus by the formula above for the intertwined case,
\begin{align*}
b_{A', B'}^C &= 
\frac{(\mathcal H_{A'} - \mathcal T_{B'} +1)! 
(\mathcal H _{B'} - \mathcal T_{A'} +1)!}
{d !
(\mathcal T_{A'} - \mathcal T_C)!
(\mathcal H_C - \mathcal H_{B'})!
(\mathcal T_{B'} - \mathcal T_C)!
(\mathcal H_C -\mathcal H_{A'})!}
t^d \\
&=
\frac{(\mathcal H_B - \mathcal T_B +1)! 
(\mathcal H _A - \mathcal T_A +1)!}
{d !
(\mathcal T_A - \mathcal T_C)!
(\mathcal H_C - \mathcal H_A)!
(\mathcal T_B - \mathcal T_C)!
(\mathcal H_C -\mathcal H_B)!}
t^d.
\end{align*}
Applying Lemma~\ref{le:bconversion}.
\begin{align*}
b_{A,B}^C =& \frac{|A'|!\ |B'|!}{|A|!\ |B|!} b_{A', B'}^C\\
=& \frac{(\mathcal H_B - \mathcal T_A+1)!(\mathcal H_A - \mathcal T_B +1)!}{(\mathcal H_A - \mathcal T_A +1)!(\mathcal H_B - \mathcal T_B +1)!} \\
&\cdot
\frac{(\mathcal H_B - \mathcal T_B +1)!(\mathcal H_A - \mathcal T_A +1)!}{d!(\mathcal T_A - \mathcal T_C)!(\mathcal H_C - \mathcal H_A)!(\mathcal T_B - \mathcal T_C)!(\mathcal H_C - \mathcal H_B)!}t^d \\
= &
\frac{(\mathcal H_A - \mathcal T_B +1)! 
(\mathcal H _B - \mathcal T_A +1)!}
{d!(\mathcal T_A - \mathcal T_C)!(\mathcal H_C - \mathcal H_A)!(\mathcal T_B - \mathcal T_C)!(\mathcal H_C - \mathcal H_B)!}
t^d.
\end{align*}
To make the formula obviously integral, multiply by $\frac{d!}{d!}$ to obtain
$$
b_{A,B}^C =
d!
{\mathcal H_A - \mathcal T_B +1 \choose d, \ \mathcal T_A - \mathcal T_C, \ \mathcal H_C - \mathcal H_B}
{\mathcal H_B - \mathcal T_A +1 \choose d, \ \mathcal T_B - \mathcal T_C, \ \mathcal H_C - \mathcal H_A}
t^d.
$$

\vspace{-.25in}
  \end{proof}

\begin{proof}[of Theorem~\ref{thm:AuBconsec}]
Let $A= A_1\cup\cdots A_k$ and $B=B_1\cup\cdots \cup B_\ell$ be written as a union of disjoint maximal consecutive subsets. Now rename the sets $\{A_1,\dots, A_k, B_1, \dots, B_\ell\}$ by $E_1, \dots, E_v$ where $v=k+\ell$ so that $\mathcal T_{E_i}\leq \mathcal T_{E_{i+1}}$ for all $i$.  
By assumption, $A\cup B= \cup_j E_j$ is consecutive.  
Since each $E_i$ is consecutive, the reordering implies $E_i\cup E_{i+1}$ is consecutive. Then by Lemma~\ref{breakup} and expanding the product,
\begin{align*}
p_Ap_B=\prod_{j=1}^v p_{E_v} & =p_{E_1} p_{E_2}  \prod_{j=3}^v p_{E_v}= \sum_{C_2} b_{E_1, E_2}^{C_2} p_{C_2} \prod_{j=3}^v p_{E_v}  \\
& =  \sum_{(C_2, C_3, C_4,\dots, C_{v-1}, C)} b_{E_1, E_2}^{C_2} b_{C_2, E_3}^{C_3}\cdots b_{C_{v-1}, E_v}^{C} p_C.
\end{align*} 
By Lemma~\ref{le:vanishingcoefficientsupport}, $b_{E_1, E_2}^{C_2} \neq 0$ 
implies $E_1\cup E_2\subseteq C_2$.
If $C_2$ weren't consecutive, there exists a maximal consecutive subset $C^0\subset C_2$ with $E_1\cup E_2\subseteq C^0$, since $E_1\cup E_2$ is consecutive. Thus  $b_{E_1,E_2}^{C_2}=0$ by Lemma~\ref{le:vanishingb}, contrary to assumption.  Thus $C_2$ is consecutive. 

Similarly, as $C_2$ is consecutive and the tails of $E_i$ are increasing with $\cup_j E_j$ consecutive, $C_2\cup E_3$ is consecutive. Thus $b_{C_2, E_3}^{C_3}\neq 0$ implies
$C_3$ is consecutive and $C_2\cup E_3\subseteq C_3$. Inductively it follows that  the sum may be taken over sequences in which all $C_i$ are consecutive, and that each coefficient $b_{E_1,E_2}^{C_2}$ and $b_{C_i, E_{i+1}}^{C_{i+1}}$ may be calculated by Theorem~\ref{general} as the corresponding sets are consecutive. 
Therefore,
\begin{equation}\label{eq:nonvanishingAuB,Cconsec}
b_{A,B}^C = \sum_{(C_2, C_3, C_4,\dots, C_{v-1})\atop {C_i\tiny{\text{consecutive}} }} b_{E_1, E_2}^{C_2} b_{C_2, E_3}^{C_3}\cdots b_{C_{v-1}, E_v}^{C} 
\end{equation}
 is the coefficient of $p_C$, as stated in Theorem~\ref{thm:AuBconsec}.   \end{proof}

Furthermore, all factors of any term in the sum \eqref{eq:nonvanishingAuB,Cconsec} are nonnegative by Theorem~\ref{general}.  Corollary~\ref{co:nonvanishingAuB,Cconsec} claims that, if a consecutive set $C$ contains $A\cup B$ and $|C|\leq |A|+|B|$, the sum is actually positive. 

\begin{coro}\label{co:nonvanishingAuB,Cconsec}
If $A\cup B$ and $C$ are consecutive, $A\cup B\subseteq C$ and $|C|\leq |A|+|B|$, then $b_{A,B}^C\neq 0$.
\end{coro}

\begin{proof}[of Corollary \ref{co:nonvanishingAuB,Cconsec}]
We need only find a single sequence 
$$
(C_2, C_3, C_4,\dots, C_{v-1})
$$ for which the corresponding summand in \eqref{eq:nonvanishingAuB,Cconsec} is nonzero. As in the proof of Theorem~\ref{thm:AuBconsec}, let $E_1,\dots, E_v$ be a reordering of the maximally  consecutive subsets of $A$ and of $B$, as for the proof of Theorem~\ref{thm:AuBconsec}. Note that
\begin{equation}\label{eq:A+B}
|A|+|B| = \sum_{i=1}^v |E_j|.
\end{equation}
Since $A\cup B$ is consecutive, $E_{j-1}\cup E_{j}$ is consecutive for each $j = 2, \dots, v$. 
Let $C_1=E_1$.  We find a set $C_j$ for $j = 2, \dots, v-1$ inductively.  Choose $C_j\subset C$ of maximal size such that
\begin{enumerate}
\item[(1)] $C_j$ is consecutive
\item[(2)]  $C_{j-1}\cup E_j \subset C_j$, and 
\item[(3)]  $|C_j|\leq |C_{j-1}|+|E_j|$.
\end{enumerate}
If $|C| > |C_{j-1}|+|E_j|$, there exists $C_j$ satisfying (1)-(2) with $|C_j|= |C_{j-1}|+|E_j|$, the maximal allowable size of property (3).  If  $|C|\leq |C_{k}|+|E_k|$, for some $k$, set $C_j= C$ for all $j\geq k+1$ and note that it necessarily satisfies conditions (1)-(3). The sets $C_{j-1}, E_j,$ and  $C_j$ are consecutive, and satisfy the degree condition of Theorem~\ref{general}, ensuring $b_{C_{j-1}, E_j}^{C_j} \neq 0$.  

We have only to show that the last term in the product is nonzero, i.e. $b_{C_{v-1}, E_v}^C\neq 0$. If $C_{v-1}=C$, then the sets $C_{v-1}, E_v$ and $C$ satisfy the conditions of Theorem~\ref{general} so the statement holds. If $C\neq C_{v-1}$, then $|C_j|=|C_{j-1}|+|E_j|$ for all $j=2, 3, \dots, v-1$. Then by Equation~\ref{eq:A+B},
$$
|C_{v-1}|=\sum_{j=1}^{v-1} |E_j| = |A|+|B|-|E_v|.
$$
Then 
$$
|C|\leq |A|+|B| = |C_{v-1}| + |E_v|,
 $$
 which is the degree requirement of Theorem~\ref{general}. Since $C_{v-1}, E_v$ and $C$ are also consecutive, Theorem~\ref{general} implies $b_{C_{v-1}, E_v}^C\neq 0$.
   \end{proof}

\begin{proof}[of Theorem~\ref{thm:AuBnotconsec}.]
Let $A\cup B= D_1\cup \dots \cup D_u$ be a union of maximal consecutive components of $A\cup B$. Note that each $A_j$ and each $B_j$ occurs in exactly one $D_i$. Thus
\begin{align*}
p_A p_B &= p_{A_1}\dots p_{A_s} p_{B_1}\dots p_{B_t}\\
& = \prod_{i=1}^u p_{A^i} p_{B^i},\quad \mbox{where }A^i =A \cap D_i, B^i = B\cap D_i \\
& = \prod_{i=1}^u \sum_{E} b_{A^i,B^i}^{E} p_{E}\\
& = \sum_{E_1,\dots, E_u} \prod_{i=1}^u b_{A^i, B^i}^{E_i} p_{E_i}\\
& =  \sum_{E_1,\dots, E_u} \left(\prod_{i=1}^u b_{A^i, B^i}^{E_i}\right)\left(\prod_{i=1}^u  p_{E_i}\right)
\end{align*}
where the sum is over sequences of consecutive $E_i$ by Lemma~\ref{breakup}, each containing $D_i=A^i\cup B^i$ by Lemma~\ref{le:vanishingcoefficientsupport}. 

Therefore, the coefficient of $p_C$ in this product is
$$
b_{A,B}^C =  \sum_{E_1,\dots, E_u} \left(\prod_{i=1}^u b_{A^i, B^i}^{E_i}\right)b_{E_1,\dots, E_u}^C,
$$
as stated by Theorem~\ref{thm:AuBnotconsec}. Each factor $b_{A^i, B^i}^{E_i}$ 
is calculated by Theorem~\ref{thm:AuBconsec} since $A^i\cup B^i$ and $E_i$ are consecutive.

We now take to calculating $\prod_{i=1}^u  p_{E_i}$  to find the coefficient $b_{E_1,\dots, E_u}^C$ of $p_C$, noting that $E_i$ is consecutive for each $i$.

If $\cup_i E_i$ is consecutive, then as before we order $E_1,\dots, E_u$ so that their tails are increasing. 
Then $E_1\cup E_2$ must be consecutive, and so we apply Theorem~\ref{thm:AuBconsec} to find
$$
p_{E_1}p_{E_2} = \sum_{C \tiny \mbox{ consecutive}\atop C\supset E_1\cup E_2} b_{E_1,E_2}^C p_C
$$
with $b_{E_1,E_2}^C$ determined by the formula in Theorem~\ref{general}. Since each $C$ contains $E_1$ and $E_2$, the union $C\cup E_3$ is consecutive for all $C$.  Therefore 
$$
p_{E_1}p_{E_2}p_{E_3} = \sum_{C \tiny \mbox{ consecutive}\atop C\supset E_1\cup E_2} b_{E_1,E_2}^C p_Cp_{E_3} = \sum_{(C_1, C_2 )\tiny \mbox{ both consecutive}\atop C_1\supset E_1\cup E_2, C_2\supset C_1\cup E_3} b_{E_1,E_2}^{C_1} b_{C_1, E_3}^{C_2} p_{C_2}.
$$
Continuing inductively, we arrive at the equation 
$$
\prod_{i=1}^u  p_{E_i} =  \sum_{(C_1, C_2, \dots, C_u)} b_{E_1,E_2}^{C_1} b_{C_1, E_3}^{C_2} \dots b_{C_{u-1}, E_u}^{C_u}p_{C_u}
$$
where the sum is over consecutive $C_s$ with $C_s \supset C_{s-1} \cup E_{s+1}$. We thus conclude 
$$
b_{E_1,\dots, E_u}^C =  \sum_{(C_1, C_2, \dots, C_{u-1})} b_{E_1,E_2}^{C_1} b_{C_1, E_3}^{C_2} \dots b_{C_{u-1}, E_u}^{C},
$$
where $C\supseteq \cup_i E_i \supseteq A\cup B$.

Now suppose $\cup_i E_i$ is not consecutive.  If none of the $E_i$ are adjacent or overlapping, then $\prod_{i=1}^u  p_{E_i} = p_{\cup_i E_i}$ has no $p_C$ term, as $C$ is consecutive.  Otherwise, there exist two sets $E_{j_1}$ and $E_{k_1}$ whose union is consecutive. Then
$$
\prod\limits_{i=1}^u p_{E_i} = p_{E_{j_1}} p_{E_{k_1}} \prod\limits_{i\neq  j_1, k_1} p_{E_i}
 = \sum_{F_1 \supset E_{j_1} \cup E_{k_1},\atop \tiny\mbox{consecutive}} b_{ E_{j_1},  E_{k_1}}^{F_1} p_{F_1}   \prod\limits_{i\neq j_1, k_1} p_{E_i}.
 $$
 For each such $F_1$, expand the product  $p_{F_1}   \prod\limits_{i\neq j_1, k_1} p_{E_i}$ with one fewer factor. Relabel the sets $F_1, E_1, \dots, \widehat{E}_{j_1}, \widehat{E}_{k_1},  \dots E_u$, and continue inductively. At each step, if the union of the sets is not consecutive, and if no two sets are adjacent, the coefficient of $p_C$ vanishes. If there are any two sets whose union is consecutive, we may expand their product using Theorem~\ref{thm:AuBconsec}. 
 
Explicitly,  for each $F_1$, we  
relabel the sets
$F_1, E_1, \dots, \widehat{E}_{j_1}, \widehat{E}_{k_1},  \dots E_u$
by $E^{(2)}_1, \dots, E^{(2)}_{u-1}$. 
Choose $j_2, k_2$ such that $E^{(2)}_{j_2}\cup E^{(2)}_{k_2}$ is consecutive. Then
\begin{align*}
\prod\limits_{i=1}^u p_{E_i} &= \sum_{F_1} b_{ E_{j_1}, E_{k_1}}^{F_1}\prod_{i=2}^{u}p_{E^{(2)}_{j}} =  \sum_{F_1} b_{ E_{j_1}, E_{k_1}}^{F_1} \left(p_{E^{(2)}_{j_2}} p_{E^{(2)}_{k_2}}\right)
\prod_{i\neq j_2, k_2}^{u}p_{E^{(2)}_{j}} \\
&= \sum_{F_1} b_{ E_{k_1},  E_{j_1}}^{F_1} \left(\sum_{F_2} b_{E^{(2)}_{j_2},E^{(2)}_{k_2}}^{F_2} p_{F_2} \right)\prod\limits_{i\neq j_2, k_2} p_{E^{(2)}_{j}}\\
&= \sum_{F_1, F_2} b_{ E_{j_1},  E_{k_2}}^{F_1} b_{E^{(2)}_{j_2},E^{(2)}_{k_2}}^{F_2} p_{F_2} \prod\limits_{i\neq j_2, k_2} p_{E^{(2)}_{j}},
\end{align*}
where the sum is over consecutive $F_1$ and $F_2$ such that $E_{j_1}\cup E_{k_1}\subseteq F_1$ and $E^{(2)}_{j_2}\cup E^{(2)}_{k_2}\subseteq F_2$. Note that the choice of $F_2$ over which we sum, and indeed the sets $E^{(2)}_j$ depend on each $F_1$.  We continue inductively.   For each sequence $F_1, \dots , F_s$ with $s<u$, there exist two sets $E^{(s)}_{j_s}, E^{(s)}_{k_s}$ among $F_s, E^{(s)}_1,\dots, E^{(s)}_{u-s+1}$ whose union is consecutive.  
Label the sets
 $F_s, E^{(s)}_1,\dots, \widehat{E}^{(s)}_{j_s}, \widehat{E}^{(s)}_{k_s}, \dots ,E^{(s)}_{u-s+1} $ 
 by $E^{(s+1)}_1,\dots , E^{(s+1)}_{u-s}$ for $s=1,\dots, u-2$, so that there is one set $E_1^{(u-1)}$ when the super index is $u-1$. 
 We have found:  
\begin{align*}
\prod\limits_{i=1}^u  p_{E_i} &=\!\! \sum_{(F_1, F_2, \dots, F_s)} b_{ E_{j_1},  E_{k_1}}^{F_1} b_{E^{(2)}_{j_2},E^{(2)}_{k_2}}^{F_2} \dots b_{E^{(s)}_{j_s},E^{(s)}_{k_s}}^{F_s} 
p_{F^{\phantom{(s)}}_s} \prod\limits_{i\neq j_s, k_s} p_{E^{(s)}_{j}}\\
&=\!\! \!\!\!\! \!\!  \sum_{(F_1, F_2, \dots, F_{u-2})} \!\! \!\! b_{ E_{j_1},  E_{k_1}}^{F_1} b_{E^{(2)}_{j_2},E^{(2)}_{k_2}}^{F_2} \dots b_{E^{(u-2)}_{j_{u-2}},E^{(u-2)}_{k_{u-2}}}^{F_{u-2}} 
p_{F^{\phantom{(u)}}_{u-2}}\!\!  \prod_{i\neq j_{u-2}, k_{u-2}}
p_{E^{(u-2)}_{i}},
\end{align*}
which, by relabeling $F_{u-2}$ and the single $E_i^{(u-2)}$ in the product,
\begin{align*}
 & =\sum_{(F_1, F_2, \dots, F_{u-2})} b_{ E_{j_1},  E_{k_1}}^{F_1} b_{E^{(2)}_{j_2},E^{(2)}_{k_2}}^{F_2} \dots b_{E^{(u-2)}_{j_{u-2}},E^{(u-2)}_{k_{u-2}}}^{F_{u-2}} 
p_{E^{(u-1)}_{1}} p_{E^{(u-1)}_{2}}\\
& = \sum_{(F_1, F_2, \dots, F_{u-2})} b_{ E_{j_1},  E_{k_1}}^{F_1} b_{E^{(2)}_{j_2},E^{(2)}_{k_2}}^{F_2} \dots b_{E^{(u-2)}_{j_{u-2}},E^{(u-2)}_{k_{u-2}}}^{F_{u-2}} 
\left( \sum_C b^C_{E^{(u-1)}_{1}, E^{(u-1)}_{2}} p_C\right),
\end{align*}
and thus
\begin{align*}
b_{E_1,\dots, E_u}^C = \sum_{(F_1, F_2, \dots, F_{u-2})} b_{ E_{j_1},  E_{k_1}}^{F_1} b_{E^{(2)}_{j_2},E^{(2)}_{k_2}}^{F_2} \dots b_{E^{(u-2)}_{j_{u-2}},E^{(u-2)}_{k_{u-2}}}^{F_{u-2}} b^C_{E^{(u-1)}_{1}, E^{(u-1)}_{2}}.
\end{align*}
Finally, to obtain the statement of Equation~\eqref{eq:productofmanyEs} in Theorem~\ref{thm:AuBnotconsec} we note that $j_{u-1}$ and $k_{u-1}$ must be the two indices $1,2$ as the union of the two sets $E^{(u-1)}_{1}$ and $E^{(u-1)}_{2}$ are necessarily consecutive.
  \end{proof}

\begin{proof}[of Theorem~\ref{thm:disconnectedgeneral}.]
We want to show that $b_{A,B}^C =\prod_k b_{A\cap C_k, B\cap C_k}^{C_k}$ where $C=C_1\cup\dots\cup C_m$ is a union of (nonempty) maximal consecutive subsets of $C$.
As $C_1,\dots, C_m$ are maximal consecutive subsets,
$A = \cup_k (A\cap C_k)$ is nonconsecutive (though for an individual $k$, $A\cap C_k$ may be consecutive).  Similarly $B = \cup_k (B\cap C_k)$ is nonconsecutive. 
By Lemma~\ref{breakup},
$$
p_A = \prod_k p_{A\cap C_k}\qquad \mbox{and}\qquad p_B = \prod_k p_{B\cap C_k},
$$
which implies
$$
p_A p_B =  \prod_k p_{A\cap C_k} p_{B\cap C_k} =   \prod_k \sum_E b_{A\cap C_k, B\cap C_k}^E p_E.
$$
Note that $b_{A\cap C_k, B\cap C_k}^E =0$ unless $E$ contains $(A\cap C_k) \cup (B\cap C_k)$ by Lemma~\ref{le:vanishingcoefficientsupport}. 
We first argue that the only terms $b_{A\cap C_k, B\cap C_k}^E\neq 0$ that contribute to the coefficient $p_C$ are those with $E\subseteq C_k$. 

Clearly, if $E$ contains elements not in $C$, the corresponding terms $p_E$ do not contribute to the coefficient $p_C$, since for any $F$,  $b_{E, F}^C\neq 0$ implies $C$ contains $E$. Thus we may suppose $E=E^0\cup E'$, where $E'$ is not consecutive with, nor intersects, $C_k$, and $E^0\subseteq C_k$. Then by Lemma~\ref{le:vanishingb}, $b_{A\cap C_k, B\cap C_k}^{E} =0.$

It follows that the coefficient of $p_C$ in $p_Ap_B$ is the coefficient of $p_C$ in 
$$
\prod_k \sum_{E_k\subseteq C_k} b_{A\cap C_k, B\cap C_k}^{E_k} p_{E_k}.
$$

On the other hand, if $E_k \neq C_k$, then $ \prod_k p_{E_k}=p_{\cup_k E_k} \neq p_C$, where the first equality follows because $\cup_k E_k$ is a non-consecutive union (Lemma~\ref{breakup}). Therefore 
$$\prod_k b_{A\cap C_k, B\cap C_k}^{C_k} p _{C_k} = \left(\prod_k b_{A\cap C_k, B\cap C_k}^{C_k}\right)p_C,
$$ as $p_C = p_{C_1}p_{C_2}\dots p_{C_m}$ (Lemma~\ref{breakup} again).
  \end{proof}

A slight generalization shows that  the non-vanishing of the structure constant holds also when $A$ and $B$ are not consecutive. To prove the general case, we need the following lemma.

\begin{lemm}\label{le:consecutiveCvanishing}
Let $A$ and $B$ be arbitrary subsets of $\{1,\dots, n-1\}$, and $C$ consecutive. Then $b_{A,B}^C\neq 0$ if and only if $C$ contains $A\cup B$ and $|C|\leq |A|+|B|$.
\end{lemm}

 \begin{proof}[of Lemma~\ref{le:consecutiveCvanishing}] 
If  $b_{A,B}^C\neq 0$, then $A\cup B\subseteq C$ and $|C|\leq |A|+|B|$ by Lemma~\ref{le:vanishingcoefficientsupport}.
 
To prove the converse, let $A \cup B = D_1 \cup \cdots \cup D_u$ where each $D_i$ is a maximal consecutive subset of $A \cup B$ and let $A^i = D_i \cap A$ and $B^i = D_i \cap B$.  By Theorem~\ref{thm:AuBnotconsec}, we have the equality
$$
b_{A,B}^C = \sum_{(E_1,\dots, E_u): \ D_i\subseteq E_i,\atop{E_i\mbox{ \tiny{consecutive}}} }\left(\prod_{i=1}^u b_{A^i, B^i}^{E_i}\right)b_{E_1,\dots, E_u}^C
$$
where $b_{E_1,\dots, E_u}^C$ is the coefficient of $p_C$ in the product $\prod_{i=1}^u p_{E_i}$. We prove there exists a sequence of sets $(E_1,\dots, E_u)$ in the index set of the sum such that $b_{A^i, B^i}^{E_i}\neq 0$ for all $i$, and $b_{E_1,\dots, E_u}^C\neq 0$.  Indeed, consider any sequence $(E_1,\dots, E_u)$ with $E_i$ consecutive and containing $D_i$, with the additional properties that $E_i\subseteq C$ and $|E_i| = \min(|A^i|+|B^i|, |C|)$. Since $D_i=A^i\cup B^i$ is consecutive and $|E_i|\leq |A^i|+|B^i| $,  by Corollary \ref{co:nonvanishingAuB,Cconsec}, $b_{A^i, B^i}^{E_i}\neq 0$. It remains to show that  $b_{E_1,\dots, E_u}^C\neq 0$.

If $\cup_i E_i$ consecutive, then by Lemma~\ref{co:nonvanishingAuB,Cconsec} $b_{E_1,\dots, E_u}^F \neq 0$ for all consecutive $F$ such that $|F|\leq \sum_i |E_i|$ and $F$ contains $\cup_i E_i$.  Since 
$$|C|\leq |A|+|B| = \sum_i |A_i|+|B_i| = \sum_i |E_i|,$$ 
and $\cup_i E_i\subseteq C$,
 the coefficient  $b_{E_1,\dots, E_u}^C \neq 0$.

 If  $\cup_i E_i$  is not consecutive, then $|C|\leq |A|+|B|= \sum_i |E_i|$ and $C$ consecutive containing $\cup_i E_i$ implies there are at least two sets $E_{j_1}, E_{k_1}$ whose union is consecutive. Thus 
\begin{align*}
\prod\limits_{i=1}^u p_{E_i} &= p_{E_{j_1}} p_{E_{k_1}} \prod\limits_{i\neq  j_1, k_1} p_{E_i}
=\left( \sum_{F \supset E_{j_1} \cup E_{k_1},\atop \tiny\mbox{consecutive}} b_{ E_{j_1},  E_{k_1}}^{F} p_{F}  \right) \prod\limits_{i\neq j_1, k_1} p_{E_i}\\
\end{align*}
and the terms $b_{ E_{j_1},  E_{k_1}}^{F}$ are nonzero whenever $F$ satisfies the degree condition $|F| \leq |E_{j_1}|+|E_{k_2}|$. In particular, let $F_1\subset C$ be a consecutive set containing $E_{j_1}\cup E_{k_1}$ with $|F_1| = \min(|E_{j_1}|+|E_{k_1}|, |C|)$.  
Then
\begin{align*}
\prod\limits_{i=1}^u p_{E_i} &= b_{ E_{j_1},  E_{k_1}}^{F_1} p_{F_1}   \prod\limits_{i\neq j_1, k_1} p_{E_i}  + \mbox{nonnegative terms}
\end{align*}
with $b_{ E_{j_1},  E_{k_1}}^{F_1}\neq 0$. As in the proof of Theorem~\ref{thm:AuBnotconsec}, we relabel the sets 
$$
F_1, E_1,\dots, \widehat{E}_{j_1},  \widehat{E}_{k_1}, \dots, E_u \text{ by } E_1^{(2)}, \dots E_{u-1}^{(2)},
$$
in which we omit sets with a $\widehat{ }$\ .
By construction of $F_1$,  $|C|\leq \sum_i  |E_i^{(2)}|$ and $\cup_i E_i^{(2)}\subseteq C$. Thus there is a pair of sets $E_{j_1}^{(2)}$ and $E_{k_1}^{(2)}$ whose union is consecutive. We continue inductively, obtaining  a sequence of consecutive sets $F_1,\dots, F_{u-2}\subseteq C$ such that $b_{E^{(s)}_{j_s},E^{(s)}_{k_s}}^{F_s} \neq 0$ and $|F_s| = \min(|E_{j_s}|+|E_{k_s}|, |C|) $ for all $s$.
By picking out the coefficient of $p_C$ in the product, we obtain:
$$
b_{E_1,\dots, E_u}^C = b_{ E_{j_1},  E_{k_1}}^{F_1} b_{E^{(2)}_{j_2},E^{(2)}_{k_2}}^{F_2} \dots b_{E^{(u-2)}_{j_{u-2}},E^{(u-2)}_{k_{u-2}}}^{F_{u-2}} b^C_{E^{(u-1)}_{1}, E^{(u-1)}_{2}} + \mbox{{\small nonnegative terms}}
$$
where the nonnegative terms in the sum are similarly products of coefficients. The  first term is nonzero because its factors are all nonzero by construction.
Thus $b_{E_1,\dots, E_u}^C\neq 0$.     \end{proof}

\begin{proof}[of Theorem~\ref{thm:positive}.] 
Suppose $b_{A,B} ^C = at^d$ with $a >0$. 
By Lemma~\ref{le:vanishingcoefficientsupport}, $A\cup B\subseteq C$. Let $C = C_1\cup\cdots \cup C_m$ be a union of maximal consecutive subsets $C_k$. Then by Theorem~\ref{thm:disconnectedgeneral}, 
$$
b_{A,B}^C = \prod_{k=1}^mb_{A\cap C_k, B \cap C_k}^{C_k}.
$$
  The hypothesis implies $b_{A\cap C_k, B \cap C_k}^{C_k} \neq 0$,  and thus by degree considerations (or Lemma~\ref{le:vanishingcoefficientsupport}), $|C_k| \leq |A\cap C_k|+ |B\cap C_k|$. 

Now suppose the converse. For $k=1, \dots, m$, let $A^k = C_k\cap A$ 
and $B^k=  C_k\cap B$. 
Note that $A^k\cup B^k\subseteq C_k$ by construction and $|C_k|\leq |A^k|+|B^k|$ by assumption. Then the coefficient $b_{A^k, B^k}^{C_k}\neq 0$ by Corollary~\ref{le:consecutiveCvanishing} as $C_k$ is consecutive. 

We show that $b_{A,B}^C\neq 0$. By Lemma~\ref{breakup}, 
$$p_A p_B = \prod_k p_{A^k} \prod_k p_{B^k} = \prod_k (p_{A^k}p_{B^k})$$
  since $A= \cup A^k$ and $B= \cup_k B^k$ are disjoint unions. 
 Each product $p_{A^k}\cdot p_{B^k}$ has at least one nonzero summand in its expansion, 
 since $b_{A^k, B^k}^{C_k} \neq 0$.
 It follows that the expansion of the product  $p_A p_B$ 
 has a nonzero term
 $$
 \prod_k \left(b_{A^k, B^k}^{C_k} p_{C_k}\right) = \prod_k b_{A^k, B^k}^{C_k} \prod_k p_{C_k} = \prod_k b_{A^k, B^k}^{C_k} p_{C},
 $$ 
 where the last equality follows from Lemma~\ref{breakup} as $C_k$ are all disjoint. It is possible that additional terms in the product contribute to the coefficient of $p_C$, however any additional terms contribute a nonnegative multiple of $t^d$, where $d = |A|+|B|-|C|$ by Corollary~\ref{cor:nonnegativecoeff}. 
As a result, the coefficient $b_{A,B}^C$ has at least one strictly positive contribution, and thus $ b_{A,B}^C= at^d$ with $a> 0$.
  \end{proof}

\section{Proof of Theorem~\ref{identity}}\label{se:proofofidentity}

Fix $m,n,w,x,y,z \in \mathbb Z$ with $x,y,z,w, m,n \geq 0$ and $w+x = y+z$. Note that Theorem~\ref{identity} holds trivially whenever $x, y, z$ or $w$ is less than $0$. 

We construct an explicit bijection between two sets of sizes given by the right hand and left sides of \eqref{eq:combinatorialidentity} in Theorem \ref{identity}. 
We carry this out as follows: we define two sets $\mathcal S$ and $\mathcal V$ whose sizes obviously correspond to the left and right hand sides of the identity in Theorem~\ref{identity}. We  construct bijections
\begin{align*}
BL^{-}: \mathcal S \rightarrow \widetilde{\mathcal S} &&\mbox{and} && BL^{\star}: \mathcal V \rightarrow \widetilde{\mathcal V},
\end{align*}
for sets $\widetilde{S}$ and $\widetilde{V}$ that will be rather clearly in one-to-one correspondence with one another. The bijections $BL^{-}$ and $BL^{\star}$ are compositions of {\em bike lock moves}, which we introduce in Section~\ref{sse:identitybikelockmoves}.

\subsection{Two sets with the right size}\label{sse:identity2sets}

We begin by describing a set $\mathcal S$ that indexes the right hand side of Theorem~\ref{identity}. Let $\mathcal S$ be the set of $2\times (w+m+n)$ matrices ${F\choose G}$ where the row $F$ rows is a sequence consisting of six letters and a placeholder, denoted by $O$, $P$, $Q$, $R$, $S$, $T$ and $-$, respectively, while row $G$ is a sequence consisting of only two letters and a placeholder, $U$, $C$, and  $-$.  We refer to the number of each letter or symbol in the matrix using the absolute value, e.g. $|P|$ refers to the number of $P$s occurring in ${F \choose G}$.

We  insist that the following relationships hold among the numbers of each letter:
\begin{itemize}
\item[$\bullet$]   $|O|+|P| = m$
\item[$\bullet$]   $|T|+|U|=n$
\item[$\bullet$]   $|Q|+|R|+ |S| = w$
\item[$\bullet$]   $|Q|-|P| = y-x$
\item[$\bullet$]  $|S|-|T| = z-x$
\item[$\bullet$]  $|C|+|O|+|U| = w+n+m$
\item[$\bullet$]  Letters are left aligned in both sequences, so that any placeholders $-$ occur to the right of all the letters, ensuring each sequence has length $w+m+n$. 
\end{itemize}

For a given pair ${F \choose G}$, let $i:= |P|$ and $j:= |T|$, then $|O|=m-i$ and $|U|=n-j$. It follows that $|Q| = y-x +i$, and $|S|=z-x+j$, so the number of letters in $F$ is $ |O|+|P|+|Q|+|R|+|S|+|T| = m+w+j$, and these letters are followed by $n-j$ placeholders. Similarly, the number of letters in $|G|$ is $|U|+|C| =  n+w+i$, and the letters are followed by $m-i$ placeholders. We tabulate the counts of each letter in Table~\ref{ta:counts} for ${F \choose G}$.

\begin{table}[h] 
  \caption{\label{ta:counts} The number of each letter in ${F \choose G}\in \mathcal S$, when $i=|P|$ and $j=|T|$.}
\begin{center}
\begin{tabular}{|c|c|c|c|c|c|c|c|c|c|}
\hline
    & $O$     & $P$   & $Q$       & $R$       & $S$       & $T$   & $U$     & $C$    & $-$   \\ \hline
$F$ & $m-i$ & $i$ & $y-x+i$ & $x-i-j$ & $z-x+j$ & $j$ &    &   &    $n-j$     \\
$G$ &       &     &         &         &         &     & $n-j$ & $w+i+j$ & $m-i$\\
\hline
\end{tabular}
\end{center}
\end{table}

By allowing $i=|P|$  and $j=|T|$ to vary from $0$ to $m$ and $n$, respectively, we obtain a count of the number of matrices ${F \choose G}$ satisfying these conditions. Among the $w+m+j$ letters in $F$, we choose where to place $i$ entries in of $P$, $j$ entries of $T$, $y-x+i$ entries of $Q$, $z-x+j$ entries of $S$, and $m-i$ entries for $O$. The remaining non-letter entries of $F$ are placeholders and have no part in the count as they must be placed at the end of the sequence. Similarly, among the $w+n+i$ letters in $G$, we choose where to place the $n-j$ copies of $U$. The remaining letters are all $C$s, and the entries of $G$ that aren't letters are placeholders at the end of the sequence. We have shown:

$$
|\mathcal S| =
\sum_{i,j}
\frac{
(w+m+j)!
}
{
i!
(y-x+i)!
(x-i-j)!
(z-x+j)!
j!
(m-i)!
}
\cdot
\frac{
(w+n+i)!
}
{
(n-j)!
(w+i+j)!
}.
$$
This expression is the right hand side of the equation in Theorem~\ref{identity}.

Now we define a set $\mathcal V$ that indexes the left hand side of Theorem~\ref{identity}.  
Let $\mathcal V$ be the set of 4-tuples of sequences $V = (v_1, v_2, v_3, v_4)$ with each $v_i$ a sequence of $1$s, $0$s, and $\star$s, with any $\star$s occurring to the right of all numbers.  We additionally require that
\begin{itemize}
\item[$\bullet$]   $v_1$ consists of $w$ $1$s, $m$ $0$s, and $n$ $\star$s
\item[$\bullet$]   $v_2$ consists of $x$ $1$s, $y-x+m$ $0$s, and $z-x+n$ $\star$s
\item[$\bullet$]   $v_3$ consists of $y$ $1$s, $z-x+n=w-y+n$ $0$s and $m$ $\star$s
\item[$\bullet$]   $v_4$ consists of $z$ $1$s, $n$ $0$s, and $y-x+m = w-z+m$ $\star$s
\item[$\bullet$]   Numbers are left aligned in all 4 sequences, so any placeholders $\star$ occur to the right of all the numbers, ensuring each sequence has length $w+m+n$.
\end{itemize}

One quickly observes that
$$
|\mathcal V| =
{w+m \choose w}
{y+m \choose x}
{w+n \choose y}
{z+n \choose z},
$$
since the $\star$ entries are all placed to in the final spots for each sequence. Observe this is the left hand side of the equality in Theorem~\ref{identity}.

For future use, we tabulate these values in Table~\ref{ta:vcounts}.
\begin{table}[h]\caption{\label{ta:vcounts} Counts of $0$s, $1$s, and $\star$s in each of $v_1,\dots, v_4$, where $w+x=y+z$}
\begin{center}
\begin{tabular}{|c|c|c|c|}
\hline
      & $1$ & $0$   & $\star$  \\ \hline
$v_1$ & $w$ & $m$  & $n$   \\
$v_2$ & $x$ & $y-x+m$ & $z-x+n$ \\
$v_3$ & $y$ & $z-x+n$ & $m$ \\
$v_4$ & $z$ & $n$   & $y-x+m$\\
\hline
\end{tabular}
\end{center}
\end{table}

\subsection{Bike Lock Moves}\label{sse:identitybikelockmoves}

The bijections we construct depend on a series of {\em bike lock moves} on $r\times c$ matrices. 
Each move is indexed by a column $k$, and specifies a set of set of rows on which it will operate (which generally depends on the matrix itself). Each affected row is will rotate its entries from $k$ to $c$ cyclically, by sending the entry in column $i$ to $i+1$, while the entry in column $c$ will move to column $k$.

\begin{defi}
For each $k$ with $1\leq k\leq c$, a {\em bike lock move} $BL_k$ on a set of matrices $\mathcal M_c$ with $c>0$ columns is a map $\mathcal M_c\rightarrow \mathcal M_c$ such that, for all $M\in \mathcal M_c$,
\begin{enumerate}
\item $BL_k(M)$ is identical to $M$ except in a specified subset of rows $R_{BL_k(M)}$. 
\item $BL_k(M)$ cyclically permutes the entries in row $\ell\in R_{BL_k(M)}$ as follows:
\begin{itemize}
\item[$\bullet$] An entry in column $m<k$ is fixed.
\item[$\bullet$] An entry in column $m$ with $k\leq m< c$ of $M$  sent to column $m+1$ in the same row.
\item[$\bullet$]  If $m=c$, the entry is sent to the $k$th column of the same row.
\end{itemize}
\end{enumerate}
\end{defi}

Observe that each bike lock move is determined by its row set.

\begin{example}  Consider the $4\times 5$ matrix $M$ on the left below.  A  bike lock move $BL_3$ on a $4\times 5$ matrix with  $R_{BL_3}(M) = \{1,3\}$ can be seen as follows. Impacted entries are highlighted in red.
\begin{center}
\begin{tikzpicture}[scale=.6, box/.style={rectangle,draw=black,thick, minimum size=1cm, scale=.6}]
   
\node[box,fill=red!40] at (3,2){d}; 
\node[box,fill=red!40] at (4,2){e}; 
\node[box,fill=red!40] at (5,2){f};  
\node[box,fill=red!40] at (3,4){a}; 
\node[box,fill=red!40] at (4,4){b}; 
\node[box,fill=red!40] at (5,4){c};  
\node[box, fill=blue!40] at (1,1){$a_{41}$}; 
\node[box, fill=blue!40] at (2,1){$a_{42}$}; 
\node[box, fill=blue!40] at (3,1){$a_{43}$}; 
\node[box, fill=blue!40] at (4,1){$a_{44}$}; 
\node[box, fill=blue!40] at (5,1){$a_{45}$}; 
\node[box, fill=blue!40] at (1,2){$a_{31}$}; 
\node[box, fill=blue!40] at (2,2){$a_{32}$}; 
\node[box, fill=blue!40] at (1,3){$a_{21}$}; 
\node[box, fill=blue!40] at (2,3){$a_{22}$}; 
\node[box, fill=blue!40] at (3,3){$a_{23}$}; 
\node[box, fill=blue!40] at (4,3){$a_{24}$}; 
\node[box, fill=blue!40] at (5,3){$a_{25}$}; 
\node[box, fill=blue!40] at (1,4){$a_{11}$}; 
\node[box, fill=blue!40] at (2,4){$a_{12}$};

\draw[->, thick] (6,2.5) -- (9,2.5);

\foreach \x in {10,11,12, 13, 14}{
    \foreach \y in {1,2, 3, 4}
        \node[box] at (\x,\y){};
}
\node[box,fill=red!40] at (12,2){f}; 
\node[box,fill=red!40] at (13,2){d}; 
\node[box,fill=red!40] at (14,2){e};  
\node[box,fill=red!40] at (12,4){c}; 
\node[box,fill=red!40] at (13,4){a}; 
\node[box,fill=red!40] at (14,4){b};   
\node[box, fill=blue!40] at (10,1){$a_{41}$}; 
\node[box, fill=blue!40] at (11,1){$a_{42}$}; 
\node[box, fill=blue!40] at (12,1){$a_{43}$}; 
\node[box, fill=blue!40] at (13,1){$a_{44}$}; 
\node[box, fill=blue!40] at (14,1){$a_{45}$}; 
\node[box, fill=blue!40] at (10,2){$a_{31}$}; 
\node[box, fill=blue!40] at (11,2){$a_{32}$}; 
\node[box, fill=blue!40] at (10,3){$a_{21}$}; 
\node[box, fill=blue!40] at (11,3){$a_{22}$}; 
\node[box, fill=blue!40] at (12,3){$a_{23}$}; 
\node[box, fill=blue!40] at (13,3){$a_{24}$}; 
\node[box, fill=blue!40] at (14,3){$a_{25}$}; 
\node[box, fill=blue!40] at (10,4){$a_{11}$}; 
\node[box, fill=blue!40] at (11,4){$a_{12}$};
\end{tikzpicture}.
\end{center}
\end{example}{

\begin{rem} \label{rem:samecounts}
Bike lock moves rotate elements starting in a specified column; they do not change the set of entries on each row, nor the number of any repeated entries. 
\end{rem}

We capture an immediate but more subtle version of this critical property of bike lock moves in the following lemma. Let $(M)_k$ indicate the $k$th column of the matrix $M$.
\begin{lemm}\label{le:samecounts}
Let $M$ be an $r\times c$ matrix, and $BL_k$ a bike lock move with $k\leq c$.  Then $M$ and $BL_k(M)$ satisfy the following properties:
\begin{enumerate}
\item The set of entries in the $\ell$th row of $M$ is the same as the set of entries in the $\ell$th row of $BL_k(M)$.
\item $(M)_\ell = (BL_k(M))_\ell$ for $\ell = 1, \dots, k-1$. 
\item If $\ell\not\in R_{BL_k(M)}$, then the $\ell$th row of $BL_k(M)$ is identical to the $\ell$th row of $M$.
\item If $\ell\in R_{BL_k(M)}$, each entry in the $\ell$th row and $j$th column of $M$ appears in the $\ell$th row and $j+1$st column of $BL_k(M)$,  for $j=k, \dots, c-1$ . 
In particular, these entries occur in the same (column) order.
\end{enumerate}
\end{lemm}

\subsection{Bike lock moves on $\mathcal S$}\label{sse:identitybikelockmovesonS}

We define a specific type of bike lock move, and apply a composition of them to elements of $\mathcal S$. 
The idea of the composition of bike lock moves is intuitive but the execution is rather technical. 
Applied to a $2\times 9$ matrix ${F \choose G}\in\mathcal S$,
$$
\begin{pmatrix} R & Q & O & S & P & R & T & R &  - \\
C & C & U & C & C & C & C & C &-  \end{pmatrix},
$$
for example, the sequence of bike lock moves 
``shuffle'' in the $-$s at the right of the matrix in order to line up the consonants $P$, $Q$, $R$, $S$ and $T$ in the top row with $C$s in the bottom row, and line up the vowels $O$ and $U$ with the $-$s:
$$
\begin{pmatrix}R & Q & O &-& S & P & R & T & R \\
C & C & - & U & C & C & C & C & C   \end{pmatrix}.
$$
Details for this example are carried out in Example~\ref{ex:S}.

\begin{defi}\label{def:-blm}
 The $-$ bike lock move $BL^{-}_k$ is defined on the set of $2\times c$ matrices whose entries in the first row are in the set $\{O,P,Q,R,S,T, -\}$, and whose entries in the second row are in $\{C,U,-\}$. Let $m_{ij}$ refer to the $(i,j)$-entry of $M$. Define:
\begin{equation}\label{eq:blcolorrows}
R_{BL^{-}_k(M)}=
\begin{cases} \{2\} \quad\mbox{if $m_{1k}=O$,}\\
\{1\} \quad\mbox{if $m_{2k} =U$ and $m_{1k}\neq O$,}\\
\emptyset \quad\mbox{else.}
\end{cases}
\end{equation} 
By definition,  $BL_k^-$  cyclicly rotates the entries in $R_{BL^{-}_k(M)}$ in columns 
$k, k+1,\dots, c$ one column to the right, with the entry in the last column sent to column $k$. 
\end{defi}

Let $BL^{-}$ be the composition
\begin{align*}
BL^{-} &:= BL^{-}_{w+m+n}\circ BL^{-}_{w+m+n-1}\circ\cdots \circ BL^{-}_2\circ BL^{-}_1.
\end{align*}
We restrict the domain to $\mathcal S$, and let 
$$\widetilde{\mathcal S} := \{BL^{-}(S) :\ S\in \mathcal S\}.$$

\begin{example}\label{ex:S} Let $S =\begin{pmatrix} R & Q & O & S & P & R & T & R &  - \\
C & C & U & C & C & C & C & C &-  \end{pmatrix}$.
We find the result of a series of bike lock moves
$$
BL^{-}(S) = BL^{-}_{9}\circ BL^{-}_{8}\circ\cdots \circ BL^{-}_2\circ BL^{-}_1 (S).
$$ The bike lock moves $BL^{-}_2\circ BL^{-}_1$ do not change $S$, since in the first two columns there is no $O$ in the first row or $U$ in the second. When applying $BL_3^{-}$, the third column ${O \choose U}$ indicates by \eqref{eq:blcolorrows} that we must shift the second row to the right:
\begin{center}
\begin{tikzpicture}[scale=.55, box/.style={rectangle,draw=black,thick, minimum size=1cm, scale=.55}]
\node[box,fill=blue!40] at (1,2){R};
\node[box,fill=blue!40] at (2,2){Q};
\node[box,fill=blue!20] at (3,2){O};
\node[box,fill=blue!40] at (4,2){S};
\node[box,fill=blue!40] at (5,2){P};
\node[box,fill=blue!40] at (6,2){R};
\node[box,fill=blue!40] at (7,2){T};
\node[box,fill=blue!40] at (8,2){R};
\node[box,fill=blue!40] at (9,2){-};
\node[box,fill=blue!40] at (1,1){C};
\node[box,fill=blue!40] at (2,1){C};
\node[box,fill=blue!20] at (3,1){U};
\node[box,fill=blue!40] at (4,1){C};
\node[box,fill=blue!40] at (5,1){C};
\node[box,fill=blue!40] at (6,1){C};
\node[box,fill=blue!40] at (7,1){C};
\node[box,fill=blue!40] at (8,1){C};
\node[box,fill=blue!40] at (9,1){-};
\draw[->,thick] (10,1.5) -- (12,1.5) node[midway, above] {$BL_3^{-}$};
\node[box,fill=blue!40] at (13,2){R};
\node[box,fill=blue!40] at (14,2){Q};
\node[box,fill=blue!40] at (15,2){O};
\node[box,fill=blue!40] at (16,2){S};
\node[box,fill=blue!40] at (17,2){P};
\node[box,fill=blue!40] at (18,2){R};
\node[box,fill=blue!40] at (19,2){T};
\node[box,fill=blue!40] at (20,2){R};
\node[box,fill=blue!40] at (21,2){-};
\node[box,fill=blue!40] at (13,1){C};
\node[box,fill=blue!40] at (14,1){C};
\node[box,fill=red!40] at (15,1){-};
\node[box,fill=red!40] at (16,1){U};
\node[box,fill=red!40] at (17,1){C};
\node[box,fill=red!40] at (18,1){C};
\node[box,fill=red!40] at (19,1){C};
\node[box,fill=red!40] at (20,1){C};
\node[box,fill=red!40] at (21,1){C};
\end{tikzpicture}
\end{center}
where we have indicated the shifted row in red.
When applying $BL_4^{-}$ to the result, the fourth column is ${S\choose U}$ so we shift the first row. 
\begin{center}
\begin{tikzpicture}[scale=.55, box/.style={rectangle,draw=black,thick, minimum size=1cm, scale=.55}]
\node[box,fill=blue!40] at (1,2){R};
\node[box,fill=blue!40] at (2,2){Q};
\node[box,fill=blue!40] at (3,2){O};
\node[box,fill=blue!20] at (4,2){S};
\node[box,fill=blue!40] at (5,2){P};
\node[box,fill=blue!40] at (6,2){R};
\node[box,fill=blue!40] at (7,2){T};
\node[box,fill=blue!40] at (8,2){R};
\node[box,fill=blue!40] at (9,2){-};
\node[box,fill=blue!40] at (1,1){C};
\node[box,fill=blue!40] at (2,1){C};
\node[box,fill=blue!40] at (3,1){-};
\node[box,fill=blue!20] at (4,1){U};
\node[box,fill=blue!40] at (5,1){C};
\node[box,fill=blue!40] at (6,1){C};
\node[box,fill=blue!40] at (7,1){C};
\node[box,fill=blue!40] at (8,1){C};
\node[box,fill=blue!40] at (9,1){C};
\draw[->,thick] (10,1.5) -- (12,1.5) node[midway, above] {$BL_4^{-}$};
\node[box,fill=blue!40] at (13,2){R};
\node[box,fill=blue!40] at (14,2){Q};
\node[box,fill=blue!40] at (15,2){O};
\node[box,fill=red!40] at (16,2){-};
\node[box,fill=red!40] at (17,2){S};
\node[box,fill=red!40] at (18,2){P};
\node[box,fill=red!40] at (19,2){R};
\node[box,fill=red!40] at (20,2){T};
\node[box,fill=red!40] at (21,2){R};
\node[box,fill=blue!40] at (13,1){C};
\node[box,fill=blue!40] at (14,1){C};
\node[box,fill=blue!40] at (15,1){-};
\node[box,fill=blue!40] at (16,1){U};
\node[box,fill=blue!40] at (17,1){C};
\node[box,fill=blue!40] at (18,1){C};
\node[box,fill=blue!40] at (19,1){C};
\node[box,fill=blue!40] at (20,1){C};
\node[box,fill=blue!40] at (21,1){C};
\end{tikzpicture}.
\end{center}
The remaining columns have no $U$s or $O$s, so this matrix is left unchanged the bike lock moves $BL_9^{-}\circ \dots \circ BL_5^{-}$. Thus $\widetilde{S} = \begin{pmatrix}R & Q & O &-& S & P & R & T & R \\
C & C & - & U & C & C & C & C & C   \end{pmatrix}$.
\end{example}

We now prove a basic property of $BL^-$ applied to elements of $\mathcal S$.  For any ${F\choose G}\in \mathcal S$, define
$$
N^-_{k}=BL_{k}^-\circ\dots\circ BL_1^-{F\choose G}.
$$
By convention $N^-_0= {F\choose G}$.
\begin{lemm}\label{le:insertion-}
For $k = 1,\dots, w+m+n$, the  bike lock move $BL_k^-$ applied to $N^-_{k-1}$ either leaves it unchanged, or inserts $-$ into the $k$th column.
\end{lemm}
\begin{proof}
Let ${F \choose G}\in \mathcal S$. 
There are $m-i$ $O$s in the first row $F$ (see Table~\ref{ta:counts}), and therefore the row specification of \eqref{eq:blcolorrows}, results indicates there are  $m-i$ bike lock moves in the composition $BL^{-}$ impacting the second row, $G$. There are $m-i$ placeholders $-$ in $G$, so each of these bike lock moves will shift a $-$  from the end of $G$ to some earlier part of the sequence.

Similarly, there are  $n-j$ $U$s in $G$, and thus by \eqref{eq:blcolorrows} at most $n-j$ individual bike lock moves that impact the row $F$. We argue that {\em exactly} $n-j$ bike lock moves in the composition $BL^-$ cycle $F$ by showing that each $U$ results in a cycle of the first row. 

Referencing \eqref{eq:blcolorrows}, the first row is cycled to the right by $BL_k^-$ whenever $(N^-_{k-1})_k= {*\choose U}$ and $*\neq O$.  
If $(N^-_{k-1})_k = {O\choose U}$, then $BL_k^-$ cycles the second row of $N^-_{k-1}$, with the result that $(N^-_k)={O\choose -}$ and  $(N^-_k)_{k+1} = {*\choose U}$; in particular, the same number of $U$s occur in columns $k+1, \dots, w+m+n$ of $N^-_k$ as occur in columns $k, \dots, w+m+n$ of $N^-_{k-1}$. 

For some $\ell\geq k$, $(N^-_\ell)_{\ell+1} = {*\choose U}$ with $*\neq O$, as the existence of $U$ in the second row guarantees some non-$O$ entries on the first row (see Table~\ref{ta:counts}).  
Thus the second row will be cycled by $BL_\ell^-$. We have shown that, for each $U$ occurring in $G$, there is a shift to the right of the original sequence $F$. 
Since there are $n-j$ placeholders $-$ at the end of $F$, each move results in the insertion of $-$ into the column associated with the bike lock move. 
 \end{proof}

\begin{coro}\label{co:orderstayssame}
The letters of $N^-_k$ are in the same order as the letters of ${F\choose G}$ for all $k=0,\dots, w+m+n$.
\end{coro}
\begin{proof}[of Corollary \ref{co:orderstayssame}]
Suppose $BL_\ell^-$ acts nontrivially on $N^-_{\ell-1}$ for some $\ell\leq k$. By
Lemma~\ref{le:insertion-}, $BL_\ell^-$ inserts a $-$ into the $\ell$th column. Lemma~\ref{le:samecounts} implies that all letters in columns $\ell+1, \dots, w+m+n-1$ in the impacted row are shifted to the right one column. Thus all letters remain in the same order after each subsequent bike lock move.
 \end{proof}

\begin{coro}\label{co:FGinj}
The composition $BL^{-}$
is bijective map from   $\mathcal S$ to $\widetilde{\mathcal S}$. 
\end{coro}

\begin{proof}[of Corollary \ref{co:FGinj}]
Let ${F \choose G}\in \mathcal S$. By Corollary~\ref{co:orderstayssame}, the order of the letters of $BL^-{F \choose G}$ in each row are the same as the order of the letters in ${F \choose G}$.  Observe the letters of ${F \choose G}$ are left-aligned. If $BL^-{F' \choose G'}$ for some ${F' \choose G'}\in \mathcal S$, then the letters of ${F' \choose G'}$ are also left-aligned, and occur in the same order as ${F \choose G}$, so that ${F' \choose G'}={F \choose G}$. Therefore, $BL^-$ is injective. Recall $\widetilde{\mathcal S}$ is the image of $BL^-$.
 \end{proof}

We now characterize $\widetilde{\mathcal S}$.
\begin{prop}\label{prop:FG} 
Elements of $\widetilde{\mathcal S}$ are exactly $2\times (w+m+n)$ matrices $M$ satisfying the following:
\begin{enumerate}
\item The columns of $M$ consist only of  7 types:
$$
 \begin{pmatrix} -\\ U\end{pmatrix},  \begin{pmatrix} O\\ -\end{pmatrix}, \begin{pmatrix} P\\ C\end{pmatrix}, \begin{pmatrix} Q\\ C\end{pmatrix}, \begin{pmatrix} R\\ C\end{pmatrix}, \begin{pmatrix} S\\ C\end{pmatrix}, \begin{pmatrix} T\\ C\end{pmatrix}.
 $$
\item  There are no pairs of adjacent columns in $M$ of the form
$
\begin{pmatrix}- & O \\U & - \end{pmatrix}.
$
\item The number of times each letter or placeholder appears in each row of $M$ is given in Table~\ref{ta:counts} for some $0\leq i \leq m$ and $0\leq j\leq n$.
\end{enumerate}
\end{prop}

We prove Proposition \ref{prop:FG} in a series of lemmas. 

\begin{lemm}\label{Ssatisfies} Elements of $\widetilde{\mathcal S}$  
satisfy the three conditions of Proposition~\ref{prop:FG}. 
\end{lemm}
\begin{proof}[of Lemma~\ref{Ssatisfies}.]
By Lemma~\ref{le:insertion-}, each nontrivial bike lock move inserts a $-$ into the corresponding column. $BL_k^-(N^-_{k-1})$ has a nontrivial row set exactly when there is an $O$ or a $U$ in the $k$th column of $N^-_{k-1}$. Thus all columns in $BL^-{F\choose G}$ with an $O$ or a $U$ are of the form ${O\choose -}$ or ${-\choose U}$. All other columns are possible, and listed in the proposition, proving Property 1. 

Observe that $-$s occur before letters in $N_k$ only in columns $1, \dots, k$. Thus $(N_k)_{k+1}$ is not ${ -\choose U}$ for any $k$, unless no letters follow on the first row, in which case the column ${ -\choose U}$ cannot be followed by $ {O\choose -}$. On the other hand, a column of the form ${O\choose U}$ results in a shift on the second row.  As a result, the column ${ -\choose U}$ is never followed by $ {O\choose -}$. This establishes Property 2. 

Finally, Lemma~\ref{le:samecounts} ensures that the counts of $BL^-{F\choose G}$ are the same as those of ${F\choose G}$. These counts are given in Table~\ref{ta:counts}, establishing Property 3. 
 \end{proof}

 We now show that any matrix $M$ satisfying these conditions is $BL^{-}{F \choose G}$ for some  ${F \choose G} \in S$.  Consider any 
matrix $M$ satisfying the conditions of Proposition~\ref{prop:FG} for some $i, j$. 
Observe that the first row of $M$ consists of entries in  $\{O, P, Q, R, S, T,-\}$ and the second row consists of entries in  $\{U, C, -\}$.
In each row of $M$, remove all placeholders, left align all letters and place the placeholders to the right of the last letter.  Note that this operation does not change the number of individual letters listed in each row. The resulting matrix is of the form ${F \choose G}$, with the number of letters of each type given in Table~\ref{ta:counts}. Therefore ${F \choose G}$ satisfies the bulleted listed in \S5.1, implying ${F \choose G}\in \mathcal S$. 

We verify that 
$$
M=BL_{w+m+n}^{-}\circ\dots\circ BL_2^{-}\circ BL_1^{-}{F \choose G}
$$
using an inductive argument on the columns of each matrix. We begin with some properties of the series of applications of bike lock moves on ${F\choose G}$.

\begin{lemm}\label{le:leftaligned}
All letters of $N^-_k$ in columns $k+1, \dots, w+m+n$ are left aligned, meaning that all letters occur before any $-$ in these columns. 
\end{lemm}
\begin{proof}[of Lemma~\ref{le:leftaligned}]
Observe that all letters of ${F\choose G}$ occur to the left of all copies of $-$. 
The application of $BL_1^-$ to ${F\choose G}$ results in either no change, or a cyclic shift to the right by one row, resulting in an entry of the first column ${F\choose G}$ moving to the second column and all other entries moving to the right, with the last entry of the row cycling to the first column. If all entries of $F$ or $G$ are $-$, then a rotation of that row will has entries that are vacuously left-aligned from the second column. If either begins with a letter, then a cycling of that row will move that letter to the right one unit, possibly inserting a $-$ in the first column. The resulting matrix remains left-aligned from column 2. 

Similarly, suppose the letters of $N^-_{k-1}$ are left-aligned among columns $k, k+1, \dots, w+m+n$ with $-$ occurring at the the end of the matrix and/or possibly in the first $k-1$ columns in $N^-_{k-1}$. The application of $BL_k^-$ to $N^-_{k-1}$ has either no effect, or it rotates one row in columns $k, k+1,\dots, w+m+n$ by one unit to the right with the entry in column $w+m+n$ moving to column $k$. If there is no effect, then clearly $N^-_k = BL_k^-(N^-_{k-1})$ is left-aligned in columns $k+1, \dots, w+m+n$. If a rotated row of $N^-_{k-1}$ has a letter in column $k$, then that letter is moved to the $k+1$st column and thus $N^-_k$ is left-aligned from column $k+1$. If the entry of a rotated row of $N^-_{k-1}$ is $-$, then $N^-_{k-1}$ has only $-$ in rows $k, k+1, \dots w+m+n$, since it is left aligned from column $k$.  Therefore $N^-_k$ has only $-$ in that row in columns $k+1, \dots, w+m+n$, so its letters in these columns are vacuously left-aligned. 
 \end{proof}

\begin{lemm}\label{le:countpersistence}
Let ${F\choose G}$ have counts of letters in Table~\ref{ta:counts}.
For $k=0,1, \dots, w+n+m$, 
\begin{itemize}
\item[$\bullet$] The number of $O$s in the first row of $N^-_{k}$ in columns $k+1,\dots, w+m+n$ is the same as the number of $-$s in the second row of $N^-_{k}$ in columns $k+1,\dots, w+m+n$, and 
\item[$\bullet$]  The number of $U$s in the second row of $N^-_{k}$ in columns $k+1,\dots, w+m+n$ is the same as the number of $-$s in the first row of $N^-_{k}$ in columns $k+1,\dots, w+m+n$.
\end{itemize}
\end{lemm}
\begin{proof}[of Lemma~\ref{le:countpersistence}]
Observe that these two properties hold for ${F\choose G}$ by a quick check on Table~\ref{ta:counts}.
If ${F \choose G}_1 = {O\choose -}$, then all entries of $G$ are $-$ which implies by Table~\ref{ta:counts} that all entries of $F$ are $O$. Similarly if ${F \choose G}_1 = {-\choose U}$, all entries of $F$ are $-$  since the letters of $F$ are left-aligned, and thus by Table~\ref{ta:counts}, all entries of $G$ are $U$. In both cases,  $ N^-_k={F\choose G}$ for all $k$, so the statement holds.

If ${F\choose G}$ consists of two consonants, then $\left(BL_1^{-}{F \choose G}\right)_1 = {F\choose G}$ and hence the number of $O$, $U$s, and $-$ in each row and in columns $2,\dots, w+m + n$, is the same for $N^-_1$ and ${F\choose G}$.  

If ${F \choose G}_1 = {O\choose *}$ for any $*\neq -$, then $\left(BL_1^{-}{F \choose G}\right)_1 = {O\choose -}$ since  $BL_1^-$ applied to ${F\choose G}$ rotates of the second row, and  Table~\ref{ta:counts} ensures there is a $-$ at the end of $G$ (since there is an $O$ in $F$). Therefore  the first row of $BL_1^{-}{F \choose G}$ has one fewer $O$ in columns $2,\dots, w+n+m$ than ${F\choose G}$, and one fewer $-$ in the second row in those columns. Since $1\not\in R_{BL_1^-{F\choose G}}$, Lemma \ref{le:samecounts}(3) implies 
The number of $-$ occurring in the first row of ${F\choose G}$ is the same as the number in $BL_1^{-}{F \choose G}$.  By Lemma~\ref{le:samecounts}(4), since $2\in R_{BL_1^-{F\choose G}}$, the number of $U$ in columns $2, \dots, w+m+n$ in $BL_1^{-}{F \choose G}$ is also unchanged.

If ${F \choose G}_1 = {*\choose U}$ for $*\neq -$ and $*\neq O$, then $\left(BL_1^{-}{F \choose G}\right)_1 = {-\choose U}$ since Table~\ref{ta:counts} ensures there is a $-$ at the end of $F$. Therefore the first row of $BL_1^{-}{F \choose G}$ has one fewer $-$ in columns $2,\dots, w+n+m$ than ${F\choose G}$, and one fewer $U$ in the second row in those columns. Since $*\neq O$, $1\in R_{BL_k^-{F\choose G}}$. Only $-$ are rotated into the first column. By Lemma~\ref{le:samecounts} the count of $O$ in the first row and the count of $-$ in the second row in ${F\choose G}$, columns $2,\dots, w+n+m$, are the same as those in 
 $BL_1^{-}{F \choose G}$.
 
Now suppose that the equalities hold for $N^-_{k-1}$. If $\left(N^-_{k-1}\right)_k = {O\choose -}$ (or ${-\choose U}$, then all entries of $N^-_{k-1}$ in columns $k, k+1,\dots w+m+n$ in the second row (or first row) are $-$, since the letters of $N^-_{k-1}$ are left-aligned (see Lemma~\ref{le:leftaligned}). By the inductive assumption, all entries of $N^-_{k-1}$ in columns $k, k+1,\dots w+m+n$ in the first row (or second row) are $O$ (or $U$).  Then  $N^-_k= N^-_{k-1}$ and there are both one fewer $-$ and one fewer $O$ (or $U$) in subsequent columns, preserving the equality of the counts. 

If $(N^-_{k-1})_k= {O\choose *}$ for any $*\neq -$, then $BL^-_k$ requires the rotation of the second row. The inductive assumption ensures that there is a $-$ at the end of the second row of $N^-_{k-1}$. Therefore the first row of $N^-_k$ has one fewer $O$ and the second row has one fewer $-$ in columns $k+1,\dots, w+n+m$  than $N^-_{k-1}$ has in columns $k, \dots, w+n+m$ . 

To check the other equality, if $(N^-_{k-1})_k= {O\choose *}$, then $BL_k^-$ rotates the second row, so that the count of $U$s in $N^-_k$ in columns $k+1, \dots, w+m+n$ equals the count of $U$s in $N^-_{k-1}$ in columns $k, \dots, w+n+m$. It follows that the lemma holds for $N^-_k$ when $(N^-_{k-1})_k= {O\choose *}$. 

Similarly, if $(N^-_{k-1})_k= {*\choose U}$ for $*\neq -$ and $*\neq O$, it must be the case that $\left(BL_k^{-}(N^-_{k-1})\right)_1 = {-\choose U}$ since Table~\ref{ta:counts} ensures there is a $-$ at the end of the first row of $N^-_{k-1}$. Thus there is one fewer $-$ in the first row and one fewer $U$ in the second row of $N^-_k$ in columns $k+1,\dots, w+m+n$, compared to the counts of the same in $N^-_{k-1}$ in columns $k, k+1, \dots, w+m+n$. As both values are reduced by $1$, they remain equal. The counts of $O$s in the first row and $-$ in the second row do not change as they are not present in the $k$th column of $N^-_{k-1}$ or $N^-_k$. 

It follows that the lemma holds for $N^-_k$ in all cases. 
 \end{proof}

 We use an inductive argument to show that  
$(M)_k = (N^-_k)_k$ for all $k$. 
The base case is established in the following lemma. 
\begin{lemm}\label{le:basecase}
Let $(M)_1$ denote the first column of $M$. Then $(M)_1 = (N^-_1)_1$.
\end{lemm}
\begin{proof}[of Lemma~\ref{le:basecase}]
If the entries of column $(M)_1$ are consonants, then $BL_1^{-}$ does not change ${F \choose G}$. 
 Therefore, 
$(N^-_1)_1=(BL^{-}_1{F\choose G} )_1=(M)_1
$
 in this case. 

Suppose $(M)_1={O\choose -}$. 
It follow that the first column of ${F\choose G}$ is ${O \choose U}$, ${O\choose C}$, or ${O \choose -}$
In all cases, the bike lock move $BL^{-}_{1}$  rotates the second row of ${F\choose G}$ (see \eqref{eq:blcolorrows}). Table~\ref{ta:counts} guarantees a $-$ at the end of the second row of ${F\choose G}$ because there exists an $O$ in the first row. It follows that
 $\left(N^-_1\right)_1 ={O\choose -}= \left(M\right)_1.$

Suppose $(M)_1 = {- \choose U}$. 
The first column of ${F\choose G}$ is thus ${* \choose U}$, where $*$ is an element of $\{O, P, Q, R, S, T, -\}$.  If $*=O$, then the first non-placeholder in the first row of $M$ would be $O$. If this occurs in column $\ell$, then $(M)_\ell = {O\choose -}$, as this is the only permitted column with an $O$ in the first row. For the same reason, $(M)_{\ell-1}= {-\choose U}$, and so $M$ contains a disallowed pair  $\begin{pmatrix} -&O\\ U&-\end{pmatrix}$. 

On the other hand, if $*$ is one of $\{P, Q, R, S, T, -\}$, then by  \eqref{eq:blcolorrows}  $BL^{-}_{1}$ cycles the first row starting in column 1, and introduces the last element in $F$ to column 1. Table~\ref{ta:counts} ensures this symbol is $-$ since the $U$ in the second row ensures a $-$ at the end of $F$. Thus
$
\left(BL_1^{-}{F \choose G}\right)_1= {-\choose U} = \left(M\right)_1.
$
 \end{proof}

Having established the base case, we assume that 
$(N^-_{k-1})_\ell = (M)_\ell$ for $\ell\leq k-1$ and show that 
$$
\left(N^-_k\right)_\ell = \left(BL_k^-\circ BL_{k-1}^{-}\circ\dots \circ BL_1^{-}{F \choose G}\right)_\ell=\left(M\right)_\ell, \quad \mbox{for $\ell\leq k$}
$$
in each three cases of the possible columns of $M$ in Proposition~\ref{prop:FG}: when $(M)_k$ consists of consonants, when  $(M)_k = {O \choose -}$ and when $(M)_k= {-\choose U}$. 

Observe that $(N^-_{k-1})_\ell = (M)_\ell$ for all $\ell\leq k-1$ implies that $(N^-_{k})_\ell = (M)_\ell$ for all $\ell\leq k-1$ as $BL_k^-$ does not change any of the first $k-1$ columns (see Lemma~\ref{le:samecounts}, Property 2). Thus we need only show that $(N^-_k)_k = (M)_k$.

\begin{lemm} \label{le:Mkconsonant}
Suppose 
$(M)_k$ consists of consonants, and  $(N^-_{k-1})_\ell = (M)_\ell$ for all $\ell\leq k-1$. Then 
$(N^-_k)_k = (M)_k$. 
\end{lemm}
\begin{proof}[of Lemma~\ref{le:Mkconsonant}]
Suppose $(M)_k = {m_{1k}\choose m_{2k}}$.  By Remark~\ref{rem:samecounts}, the number of each symbol that occur in $M$ and in $N^-_{k-1}$ are the same. The two matrices agree on the first $k-1$ columns., so $m_{1k}$ and $m_{2k}$ appear in the first and second rows of $N^-_{k-1}$ in columns $\ell_1$ and $\ell_2$, respectively, where $\ell_1, \ell_2\geq k$. 
By Corollary~\ref{co:orderstayssame}, the order of the letters are the same in $M$ and in $N^-_{k-1}$, so $m_{1k}$ and $m_{2k}$ are the the first letters to appear in $N^-_{k-1}$ in column $k$ or later, in their respective rows. By Lemma~\ref{le:leftaligned}, the letters in columns $k, k+1, \dots, w+m+n$ of  $N^-_{k-1}$  are left aligned, and thus $\ell_1=\ell_2=k$. 
It follows that
$(N^-_{k-1})_k = (M)_k$, and thus $(N^-_{k-1})_k$ consists of consonants. By Definition \ref{def:-blm}, 
$N^-_k = 
N^-_{k-1},$  and thus $(N^-_k)_k = (N^-_{k-1})_k =(M)_k$, as desired. \end{proof}

\begin{lemm}\label{MkO} 
Suppose 
$(M)_k$ consists of ${O \choose -}$ and

 $(N^-_{k-1})_\ell = (M)_\ell$ for all $\ell\leq k-1$. Then 
$(N^-_k)_k = (M)_k$. 
\end{lemm}

\begin{proof}[of Lemma~\ref{MkO}]
By Corollary~\ref{co:orderstayssame}, the letters of $F$ are in the same order as the letters of $M$. Thus the first row entry of 
$(N^-_k)_k$ 
is either $O$ or $-$. 

\vspace{.1in}
\noindent{\em Case 1.} Suppose $(N^-_k)_k = {O\choose *}$ for $*$ either $C$ or $U$. 
Then either $(N^-_{k-1})_k = {O\choose *}$, in which case $BL_k^-$ does not change the first entry, or 
 $(N^-_{k-1})_k = {*\choose U}$ for $*$ some symbol not $O$.  Observe that in the latter case, the bike lock move $BL_k^-$ whose row shifts are specified in \eqref{eq:blcolorrows} results in a shift of the first row. 
Following Lemma~\ref{le:countpersistence}, there are exactly as many copies of $U$ in the second row, columns $k, \dots, w+m+n$ as there are $-$ in the first row in these columns of $N^-_{k-1}$. By Lemma~\ref{le:leftaligned}, there is a $-$ at the end of the first row of $N^-_{k-1}$. 
 As a consequence, $\left(N^-_k\right)_k=\left(BL_{k}^{-}(N^-_{k-1})\right)_k = {-\choose U},$ contrary to assumption. 
 Thus we may assume that $(N^-_{k-1})_k = {O\choose *}$. 
Using Lemma~\ref{le:countpersistence} again,
 there are as many $-$ in the second row in columns $k, \dots, w+m+n$ of $N^-_{k-1}$ as there are $O$s in the first row, so that
$$
\left(N^-_k\right)_k = \left(BL_k^-(N^-_{k-1})\right)_k ={O\choose -} = (M)_k.
$$ 
This establishes that $(N^-_k)_k = {O\choose *}$ implies $(N^-_k)_k = (M)_k$.

\vspace{.1in}
\noindent{\em Case 2.}
Suppose
$\left(N^-_k\right)_k = {-\choose *}$ for some symbol $*$. 
Since $N^-_k = BL_k^-(N^-_{k-1})$, the specification of row shifts in \eqref{eq:blcolorrows} of $BL_k^-$ implies $(N^-_{k-1})_k = {-\choose *}$ for some symbol $*$, or 
$(N^-_{k-1})_k = {*\choose U}$ for a symbol $*$ that is not $O$.  

If $(N^-_{k-1})_k = {*\choose U}$ with $*\neq O$,
then by Lemmas~\ref{le:countpersistence} and \ref{le:leftaligned}, 
there is a $-$ at the end of $N^-_{k-1}$ in the first row. As a result, $(N^-_k)_k = (BL_k^-(N^-_{k-1})_k=  {-\choose U}$. On the other hand, if $(N^-_{k-1})_k = {-\choose *}$, but $*\neq U$, then $(N^-_k)_k = {-\choose *}$ as $BL_k^-$ has no effect. In either case, by Corollary~\ref{co:orderstayssame}, $N^-_k$ must have an $O$ in the first row of some column $\ell > k$  and $-$ in rows $k, k+1, \dots, \ell-1$, since $M$ has an $O$ in the first row in column $k$. But then the letters of $N^-_k$ are not left-aligned from from column $k+1$, contrary to  Lemma~\ref{le:leftaligned}.

These two cases establish that $(M)_k = {O \choose -}$ implies $(N^-_k)_k  =  (M)_k$. 
  \end{proof}

\begin{lemm}\label{MkU} 
Suppose 
$(M)_k={- \choose U}$ and
 $(N^-_{k-1})_\ell = (M)_\ell$ for all $\ell\leq k-1$. Then 
$(N^-_k)_k = (M)_k$. 
\end{lemm}

\begin{proof}[of Lemma~\ref{MkU}]
The letters of the second row of $N^-_{k-1}$ are in the same order as those of $M$ by Corollary~\ref{co:orderstayssame}. Thus the second entry of $(N^-_k)_k$ is either $U$ or $-$. 

\vspace{.1in}
\noindent{\em Case 1.} Suppose $(N^-_k)_k = {*\choose U}$. By Definition \ref{def:-blm}, $(N^-_{k-1})_k$ has either an $O$ in the first row or a $U$ in the second row. If $(N^-_{k-1})_k={O\choose *}$, then $BL_k^-$ rotates the second row of $N^-_{k-1}$ starting in column $k$. By Lemmas~\ref{le:countpersistence} and ~\ref{le:leftaligned}, there is a $-$ at the end of the second row of $N^-_{k-1}$ Therefore, using \eqref{eq:blcolorrows}, $(N^-_k)_k = {O\choose -}$, contrary to assumption. 

Alternatively $(N^-_{k-1})_k = {*\choose U}$ with $*\neq O$. Then Lemmas~\ref{le:countpersistence} and ~\ref{le:leftaligned} imply that there is a $-$ at the end of the first row of $N^-_{k-1}$. It follows that $$\left(N^-_{k}\right)_k = {-\choose U} = \left(M\right)_k.$$

\vspace{.1in}
\noindent{\em Case 2.} Suppose $(N^-_k)_k = {*\choose -}.$  If $(N^-_{k-1})_k$ also has a $-$ in the $k$th column second row, then all the entries in the second row of $N^-_{k-1}$ in columns $k, k+1, \dots w+m+n$ are  $-$ as letters are left aligned (see Lemma~\ref{le:leftaligned}). However, this contradicts the fact that $(M)_k$ has a $U$ in the second row, as the letters must be the same as those in $M$ in columns $k, k+1, \dots w+m+n$  (by Lemma~\ref{le:countpersistence}). 

We may therefore assume that the second row of $(N^-_{k-1})_k$ is not $-$. In this case $BL_k^-$ moves the second row to ensure that $(N^-_k)_k = {*\choose -}$. However $BL_k^-$ moves the second row if and only if the first entry is $O$, so  $(N^-_{k-1})_k = {O\choose *}$. It follows that $(N^-_k)_k = {O\choose -}$. Then by Corollary~\ref{co:orderstayssame}, the first letter occurring in the first row in columns $k+1, \dots w+m+n$  of $M$ must be $O$.
If this occurs in column $\ell$ with $\ell>k$,  then $(M)_\ell = {O\choose -}$, since this is the only permitted column with an $O$. For the same reason, $(M)_{\ell-1}= {-\choose U}$, and so $M$ contains a disallowed pair  $\begin{pmatrix} -&O\\ U&-\end{pmatrix}$.
We conclude that Case 2 cannot occur. 
 \end{proof}

We finally have the ingredients to prove Proposition~ \ref{prop:FG}.
\begin{proof}[of Proposition~ \ref{prop:FG}]
If $S\in \widetilde{S}$, then $S$ satisfies the three properties of the proposition, by Lemma~ \ref{Ssatisfies}. On the other hand, if $M$ satisfies these three properties, then 
construct ${F \choose G}$ by removing all $-$ from each row, left aligning all letters, and placing all $-$ at the end of the corresponding row, as done earlier.

Observe that $(BL^-{F\choose G})_\ell = (N^-_k)_\ell$ whenever $\ell\leq k$ (see Lemma~\ref{le:samecounts}). By Lemma~\ref{le:basecase}, the first columns of $M$ and $N^-_1$ agree. Therefore, the first columns of $M$ and $BL^-{F\choose G}$ agree. 

By way of induction we assume that the first $k-1$ columns of $M$ and  $BL^-{F\choose G}$ agree. Then by Lemma~\ref{le:samecounts}, the first $k-1$ columns of $M$ and of $N^-_{k-1}$ agree. 
Lemmas~\ref{le:Mkconsonant}, \ref{MkO} and \ref{MkU} imply that the $k$th columns of $M$ and $N^-_k$ agree, and hence that the $k$th columns of $M$ and of $BL^-{F\choose G}$ agree. We conclude that all columns of $M$ agree with all columns of  $BL^-{F\choose G}$, i.e. $M = BL^-{F\choose G}$. 
 \end{proof}

\subsection{Bike lock moves on elements of $\mathcal V$}\label{subsec:bikelockV}
\begin{defi}\label{def:starblm}
We define a bike lock move $BL^{\star}_k$ on 
the set of $4\times c$ matrices $M$ with entries in $\{0, 1, \star\}$, with row shifts listed in Table~\ref{eq:blV}.
By definition,  $BL_k^\star$  cyclicly rotates the entries in each row of $R_{BL^{\star}_k(M)}$ and columns 
$k, k+1,\dots, c$ one column to the right, with the entry in the last column sent to column $k$. 
\end{defi}

\begin{table}[hbt!]\caption{\label{eq:blV} Rows moved by $BL_k^{\star}$, depending on the $k$th column.}
\begin{center}
\begin{tabular}{|cccc|c|}
\hline
& $(M)_k$ &&& $R_{BL^{\star}_k(M)}$  \\ 
\hline
&&&&\\
$\begin{pmatrix} 0 \\ 1\\ 1 \\ 0\end{pmatrix}$ & $ \begin{pmatrix} 1 \\ 1 \\ 1 \\ 0\end{pmatrix}$ & $\begin{pmatrix} 1 \\ \star \\ 0 \\ 0\end{pmatrix}$ && $\{1\}$   \\ 
&&&&\\
$\begin{pmatrix} 1 \\ 0 \\ 0 \\ 1\end{pmatrix}$ & $\begin{pmatrix} 1 \\ 1 \\ 0 \\ 1\end{pmatrix}$ & $\begin{pmatrix} \star \\ 1 \\ 0 \\ 0\end{pmatrix}$  && $\{2\}$ \\
&&&&\\
$\begin{pmatrix} 0 \\ 1 \\ 0 \\ 1\end{pmatrix}$ & $\begin{pmatrix} 0 \\ 1 \\ 1 \\ 1\end{pmatrix}$ & $\begin{pmatrix} 0 \\ 0 \\ 1 \\ \star\end{pmatrix}$ && $\{3\}$  \\
&&&&\\
$\begin{pmatrix} 1 \\ 0 \\ 1 \\ 0\end{pmatrix}$ & $\begin{pmatrix} 1 \\ 0 \\ 1 \\ 1\end{pmatrix}$ & $\begin{pmatrix} 0 \\ 0 \\ \star \\ 1\end{pmatrix}$ &&$\{4\}$   \\
&&&&\\
$\begin{pmatrix} 1 \\ 0 \\ 0 \\ 0\end{pmatrix}$ & $\begin{pmatrix} 0 \\ 1 \\ 0 \\ 0\end{pmatrix}$ & $\begin{pmatrix} 1 \\ 1 \\ 0 \\ 0\end{pmatrix}$ &&$\{1,2\}$   \\
&&&&\\
$\begin{pmatrix} 0 \\ 0 \\ 0 \\ 0\end{pmatrix}$ & $\begin{pmatrix} 0 \\ 0 \\ 1 \\ 0\end{pmatrix}$ & $\begin{pmatrix} 0 \\ 0 \\ 0 \\ 1\end{pmatrix}$ & $\begin{pmatrix} 0 \\ 0 \\ 1 \\ 1\end{pmatrix}$ &$\{3,4\}$  \\
&&&&\\
 & else & & & $\emptyset$\\
 \hline
\end{tabular}
\end{center}
\end{table}
The rows $R_{BL^{\star}_k(M)}$ that  $BL^{\star}_k$  shifts depend on the columns of $M,$ as indicated in 
Table~\ref{eq:blV}. Let $(M)_k$ denote the $k$th column of $M$.

\begin{example} 
We show that 
$$
BL_5^{\star}\circ BL_4^{\star}\circ BL_3^{\star}\circ BL_2^{\star}\circ BL_1^{\star} \begin{pmatrix}
0 & 1 & 0 &  \star & \star \\
0 & 0 & 0 &  \star & \star \\
 0 & 1 & 0 &  \star & \star \\
 0 & 0 & \star &  \star & \star
\end{pmatrix} = \begin{pmatrix}
0 & \star & 1 &  0 & \star \\
0 & \star & 0 &  0 & \star \\
 \star & 0 & 1 &  \star & 0\\
  \star &0  & \star &  \star& 0
\end{pmatrix}.
$$
Apply each bike lock move referring to $R_{BL_k^{\star}(M)}$ in \eqref{eq:blV} for the appropriate rows to shift. In each case we highlight the column that determines the row shift, and color the impacted cells that have changed with each bike lock move.

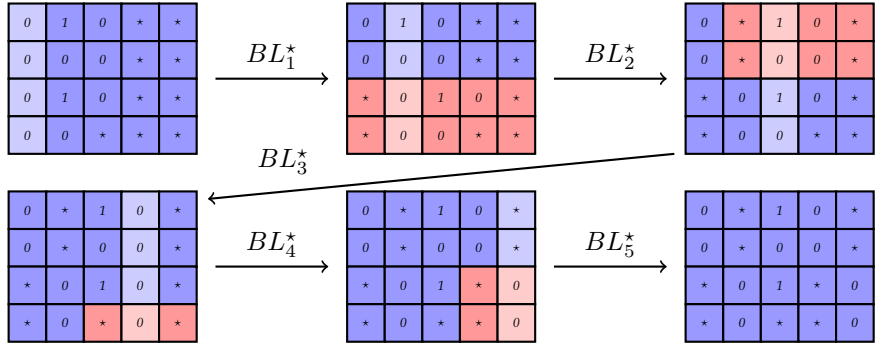
\begin{center}
\begin{tikzpicture}[scale=.50, box/.style={rectangle,draw=black,thick, minimum size=1cm, scale=.50}]
\node[box,fill=blue!20] at (1,4){0};
\node[box,fill=blue!40] at (2,4){1};
\node[box,fill=blue!40] at (3,4){0};
\node[box,fill=blue!40] at (4,4){$\star$};
\node[box,fill=blue!40] at (5,4){$\star$};
\node[box,fill=blue!20] at (1,3){0};
\node[box,fill=blue!40] at (2,3){0};
\node[box,fill=blue!40] at (3,3){0};
\node[box,fill=blue!40] at (4,3){$\star$};
\node[box,fill=blue!40] at (5,3){$\star$};
\node[box,fill=blue!20] at (1,2){0};
\node[box,fill=blue!40] at (2,2){1};
\node[box,fill=blue!40] at (3,2){0};
\node[box,fill=blue!40] at (4,2){$\star$};
\node[box,fill=blue!40] at (5,2){$\star$};
\node[box,fill=blue!20] at (1,1){0};
\node[box,fill=blue!40] at (2,1){0};
\node[box,fill=blue!40] at (3,1){$\star$};
\node[box,fill=blue!40] at (4,1){$\star$};
\node[box,fill=blue!40] at (5,1){$\star$};
\draw[->,thick] (6,2.5) -- (9,2.5) node[midway, above] {$BL_1^{\star}$};
\node[box,fill=blue!40] at (10,4){0};
\node[box,fill=blue!20] at (11,4){1};
\node[box,fill=blue!40] at (12,4){0};
\node[box,fill=blue!40] at (13,4){$\star$};
\node[box,fill=blue!40] at (14,4){$\star$};
\node[box,fill=blue!40] at (10,3){0};
\node[box,fill=blue!20] at (11,3){0};
\node[box,fill=blue!40] at (12,3){0};
\node[box,fill=blue!40] at (13,3){$\star$};
\node[box,fill=blue!40] at (14,3){$\star$};
\node[box,fill=red!40] at (10,2){$\star$};
\node[box,fill=red!20] at (11,2){0};
\node[box,fill=red!40] at (12,2){1};
\node[box,fill=red!40] at (13,2){0};
\node[box,fill=red!40] at (14,2){$\star$};
\node[box,fill=red!40] at (10,1){$\star$};
\node[box,fill=red!20] at (11,1){0};
\node[box,fill=red!40] at (12,1){0};
\node[box,fill=red!40] at (13,1){$\star$};
\node[box,fill=red!40] at (14,1){$\star$};
\draw[->,thick] (15,2.5) -- (18,2.5) node[midway, above] {$BL_2^{\star}$};
\node[box,fill=blue!40] at (19,4){0};
\node[box,fill=red!40] at (20,4){$\star$};
\node[box,fill=red!20] at (21,4){1};
\node[box,fill=red!40] at (22,4){0};
\node[box,fill=red!40] at (23,4){$\star$};
\node[box,fill=blue!40] at (19,3){0};
\node[box,fill=red!40] at (20,3){$\star$};
\node[box,fill=red!20] at (21,3){0};
\node[box,fill=red!40] at (22,3){0};
\node[box,fill=red!40] at (23,3){$\star$};
\node[box,fill=blue!40] at (19,2){$\star$};
\node[box,fill=blue!40] at (20,2){0};
\node[box,fill=blue!20] at (21,2){1};
\node[box,fill=blue!40] at (22,2){0};
\node[box,fill=blue!40] at (23,2){$\star$};
\node[box,fill=blue!40] at (19,1){$\star$};
\node[box,fill=blue!40] at (20,1){0};
\node[box,fill=blue!20] at (21,1){0};
\node[box,fill=blue!40] at (22,1){$\star$};
\node[box,fill=blue!40] at (23,1){$\star$};
\draw[->, thick] (18.2,.5)--(5.8,-.7) node[right=10mm, above=2mm] {$BL_3^{\star}$};
\node[box,fill=blue!40] at (1,-1){0};
\node[box,fill=blue!40] at (2,-1){$\star$};
\node[box,fill=blue!40] at (3,-1){1};
\node[box,fill=blue!20] at (4,-1){0};
\node[box,fill=blue!40] at (5,-1){$\star$};
\node[box,fill=blue!40] at (1,-2){0};
\node[box,fill=blue!40] at (2,-2){$\star$};
\node[box,fill=blue!40] at (3,-2){0};
\node[box,fill=blue!20] at (4,-2){0};
\node[box,fill=blue!40] at (5,-2){$\star$};
\node[box,fill=blue!40] at (1,-3){$\star$};
\node[box,fill=blue!40] at (2,-3){0};
\node[box,fill=blue!40] at (3,-3){1};
\node[box,fill=blue!20] at (4,-3){0};
\node[box,fill=blue!40] at (5,-3){$\star$};
\node[box,fill=blue!40] at (1,-4){$\star$};
\node[box,fill=blue!40] at (2,-4){0};
\node[box,fill=red!40] at (3,-4){$\star$};
\node[box,fill=red!20] at (4,-4){0};
\node[box,fill=red!40] at (5,-4){$\star$};
\draw[->,thick] (6,-2.5) -- (9,-2.5) node[midway, above] {$BL_4^{\star}$};
\node[box,fill=blue!40] at (10,-1){0};
\node[box,fill=blue!40] at (11,-1){$\star$};
\node[box,fill=blue!40] at (12,-1){1};
\node[box,fill=blue!40] at (13,-1){0};
\node[box,fill=blue!20] at (14,-1){$\star$};
\node[box,fill=blue!40] at (10,-2){0};
\node[box,fill=blue!40] at (11,-2){$\star$};
\node[box,fill=blue!40] at (12,-2){0};
\node[box,fill=blue!40] at (13,-2){0};
\node[box,fill=blue!20] at (14,-2){$\star$};
\node[box,fill=blue!40] at (10,-3){$\star$};
\node[box,fill=blue!40] at (11,-3){0};
\node[box,fill=blue!40] at (12,-3){1};
\node[box,fill=red!40] at (13,-3){$\star$};
\node[box,fill=red!20] at (14,-3){0};
\node[box,fill=blue!40] at (10,-4){$\star$};
\node[box,fill=blue!40] at (11,-4){0};
\node[box,fill=blue!40] at (12,-4){$\star$};
\node[box,fill=red!40] at (13,-4){$\star$};
\node[box,fill=red!20] at (14,-4){0};
\draw[->,thick] (15,-2.5) -- (18,-2.5) node[midway, above] {$BL_5^{\star}$};
\node[box,fill=blue!40] at (19,-1){0};
\node[box,fill=blue!40] at (20,-1){$\star$};
\node[box,fill=blue!40] at (21,-1){1};
\node[box,fill=blue!40] at (22,-1){0};
\node[box,fill=blue!40] at (23,-1){$\star$};
\node[box,fill=blue!40] at (19,-2){0};
\node[box,fill=blue!40] at (20,-2){$\star$};
\node[box,fill=blue!40] at (21,-2){0};
\node[box,fill=blue!40] at (22,-2){0};
\node[box,fill=blue!40] at (23,-2){$\star$};
\node[box,fill=blue!40] at (19,-3){$\star$};
\node[box,fill=blue!40] at (20,-3){0};
\node[box,fill=blue!40] at (21,-3){1};
\node[box,fill=blue!40] at (22,-3){$\star$};
\node[box,fill=blue!40] at (23,-3){0};
\node[box,fill=blue!40] at (19,-4){$\star$};
\node[box,fill=blue!40] at (20,-4){0};
\node[box,fill=blue!40] at (21,-4){$\star$};
\node[box,fill=blue!40] at (22,-4){$\star$};
\node[box,fill=blue!40] at (23,-4){0};
\end{tikzpicture}
\end{center}
\end{example}

We prove a series of properties of $BL_k^\star$ that will allow us to completely describe $\widetilde{\mathcal V}:=\{BL^{\star}(V): \  V\in \mathcal V\}$.  Let 
$$
N^\star _k:= BL_k ^\star \circ \cdots \circ BL_1^\star (V)
$$
with the convention $N^\star _0 = V$.

\begin{lemm}\label{le:leftalignedstar}
All $0$s and $1$s of $N^\star _k$ in columns $k+1, \dots, w+m+n$ are left aligned, meaning that all numbers occur before any $\star$ in these columns.  
\end{lemm}

\begin{proof}
The 0s and 1s in $V$ are all left aligned by definition.  If a row of $V$ consists of all $\star$s it is vacuously left aligned.  Assume inductively that for all $\ell \leq k$ that the $0$s and $1$s in $N^\star _\ell$ in columns $k +1, \ldots , w+m+n$ are left aligned.  We consider $BL_{k+1}^\star (N^\star _k)$.  By hypothesis the numbers in the rows of columns $k+1$ through $w+m+n$ are left aligned; if $BL_k^\star (N^\star _k)$ is trivial they remain left aligned and in particular the numbers in columns $k+2$ to $w+m+n$ remain left aligned.  

If $BL_{k+1}^\star (N^\star _k)$ rotates one of the rows of $N^\star _k$, then in inserts the last entry in the rotated row into $ (N^\star _k)_{k+1}$ and shifts the remaining entries to the right by one.  Thus the numbers in the affected row remain left aligned in columns $k+2$ and any unchanged rows also preserve the property.
 \end{proof}

\begin{coro}\label{rightalignedstar}
All the $\star$s of $N^\star _k$ are right aligned, meaning that if the $\ell$th row of $N^\star _k$ is a $\star$, then so is every entry of row $\ell$ in columns $k+2$ to $w+m+n$.
\end{coro}

\begin{proof}
Since the numbers in rows $k+1$ to $w+m+n$ are left aligned by Lemma~\ref{le:leftalignedstar}, the entries to the right of all of the numbers in a row in columns $k+2$ to $w+m+n$ must all be $\star$s.
  \end{proof}

\begin{lemm}\label{le:countpersistencestar}
Suppose $V \in \mathcal V$, or $V$ has counts of $0$s, $1$s, and $\star$s given in Table~\ref{ta:vcounts}.
For $k=0,1, \dots, w+n+m$, 
\begin{itemize}
\item[$\bullet$] The number of $\star$s in the first row of $N^\star _k$ in columns $k+1,\dots, w+m+n$ is the same as the number of $0$s in the fourth row of $N^\star _k$ in columns $k+1,\dots, w+m+n$;
\item[$\bullet$]  The number of $\star$s in the second row of $N^\star _k$ in columns $k+1,\dots, w+m+n$ is the same as the number of $0$s in the third row of $N^\star _k$ in columns $k+1,\dots, w+m+n$;
\item[$\bullet$]  The number of $\star$s in the third row of $N^\star _k$ in columns $k+1, \dots, w+m+n$ is the same as the number of $0$s in the first row of $N^\star _k$; and
\item[$\bullet$] The number of $\star$s in the fourth row of $N^\star _k$ in columns $k+1, \dots, w+m+n$ is the same as the number of $0$s in the second row of $N^\star _k$.
\end{itemize}
\end{lemm}

\begin{proof}
By referencing Table~\ref{ta:vcounts} the above statement is true for $k = 0$.  Assume by induction that the statements in the lemma hold for all $0 \leq \ell <k$ and consider $BL_k^\star(N^\star _{k-1})$.  We consider each row separately.  

By Lemma~\ref{le:leftalignedstar}, if there is a $\star$ in the first row of $(N^\star _{k-1})_k$ then all the remaining entries of the first row must also be $\star$s and therefore by the inductive assumption all the remaining entries in the fourth row of $N^\star _{k-1}$ must be $0$.  

Note that $1$ or $2$ is in the row set of $BL_k^\star$ applied to $N_{k-1}^\star$ implies neither $3$ nor $4$ is in the row set (see Table~\ref{eq:blV}). In particular, in these cases, the third and fourth rows of  $N_k^\star$ and $N_{k-1}^\star$ are identical. 

If $1 \in R_{BL_k^\star(N_{k-1}^\star)}$ then by referring to Table~\ref{eq:blV} we see that there must be a $0$ in the fourth row of $(N^\star _{k-1})_k$.  By the inductive hypothesis there must be a $\star$ in the first row in columns $k$ through $w+m+n$, by Corollary~\ref{rightalignedstar} a $\star$ occur in the last entry of the first row.  Thus $BL_k^\star$ rotates a $\star$ into the first row of $(N^\star _{k-1})_k$.  Since the fourth row of $N_k^\star$ is identical to that of $N_{k-1}^\star$,
%Hence 
both the number of $\star$s in the first row and $0$s in the fourth row of columns $k+1$ to $w+m+n$ of $N^\star _k$ decrease by one. 

Similarly, if $2 \in R_{BL_k^\star(N_{k-1}^\star)}$, then Table~\ref{eq:blV} implies there is a $0$ in the third row of $(N^\star _{k-1})_k$, and also that $3\not \in R_{BL_k^\star}$. By the inductive assumption there must be a $\star$ in the second row among columns $k, \dots, w+m+n$ of $N_{k-1}^\star$ of  $(N^\star _{k-1})_k$, and by Corollary~\ref{rightalignedstar}, such a $\star$ is found at the end of the second row. Furthermore, since the third row is not cycled by $BL_k^\star$, the matrix $N_k^\star$ has one fewer $0$ in row 3, and one fewer $\star$ in row $2$, in columns $k+1,\dots, w+m+n$, compared to the number of each in columns $k, \dots, w+m+n$ of $N_{k-1}^\star$. 
Thus the properties of the Lemma hold for $N_k^\star$.

A similar argument applies to prove the case when $3$ or $4$ is in the row set of $BL_k^\star$ applied to $N_{k-1}^\star$.
 \end{proof}

\begin{coro}\label{cor:insertionstar}
For $k = 1,\dots, w+m+n$, the  bike lock move $BL_k^\star$ applied to $N^\star _{k-1}$ either leaves it unchanged, or inserts $\star$ into the $k$th column.
\end{coro}

\begin{proof}  If $BL^\star _k$ shifts the first row of $N^\star _{k-1}$ then there is a $0$ in the fourth row of $(N^\star _{k-1})_k$.  By Lemma~\ref{le:countpersistencestar} and Corollary~\ref{rightalignedstar} there is a $\star$ in the first row of $N^\star _{k-1}$ in the $w+m+n$ column.  Thus $BL^\star _k$ rotates a $\star$ into the first row of $(N^\star _{k-1})_k$.

Similar arguments apply to prove the cases when $BL^\star _k$ shifts the second, third, and fourth rows of $N^\star _{k-1}$.  Thus if $BL ^\star _k$ shifts the $\ell$th  of $N^\star _{k-1},$  it cycles a $\star$ into the $\ell$th row of $(N^\star _{k-1})_k$. 
  \end{proof}

\begin{coro}\label{co:orderstayssamestar}
The numbers of $N^\star _{k}$ are in the same order as the numbers of 
$V$ for all $k=0,\dots, w+m+n$.
\end{coro}

\begin{proof}
By Corollary~\ref{cor:insertionstar} if $BL^\star _k$ shifts row $\ell$ of $N^\star _{k-1}$ then it cycles the entries of row $\ell$ in columns $k$ to $w+m+n$ to the right by 1.  By Corollary~\ref{cor:insertionstar} $BL^\star _k$ always shifts a $\star$ into the $k$th column of $N^\star _{k-1}$ and so the original order of the numbers is preserved.
  \end{proof}

\begin{lemm}\label{lemma:Vinj}
The composition 
 $$BL^{\star}:= BL^{\star}_{w+m+n}\circ \cdots \circ BL^{\star}_2 \circ BL^{\star}_1$$
 is a bijective map from $\mathcal V$ to  $\widetilde{\mathcal V}:=\{BL^{\star}(V): \  V\in \mathcal V\}.$
\end{lemm}
\begin{proof}[of Lemma \ref{lemma:Vinj}]
We verify that $BL^\star$ is injective.  Suppose that $BL^\star (V) = BL^\star (V')$ for some $V, V' \in \mathcal V$.  By Corollary~\ref{co:orderstayssamestar} the order of the $0$s and $1$s is preserved from $V$ to $BL^\star (V)$ and $V'$ to $BL^\star (V')$, implying that the sequence of $0$s and $1$s in each row of $V$ and $V'$ are the same since $BL^\star(V) = BL^\star (V')$.  Since the $0$s and $1$s in $V$ and $V'$ are left aligned, it must be the case that $V = V'$.  Hence $BL^\star$ injects onto its image.
  \end{proof}

We now characterize the elements of $\widetilde{\mathcal V}$ by a careful accounting of what each bike lock move $BL_k^{\star}$  does to columns of $ N^\star _k:= BL_{k}^{\star}\circ  \dots BL_1^{\star}(V)$ for $V\in \mathcal V$.  

\begin{prop}\label{lemma:V}
Elements of $\widetilde{\mathcal V}$ are exactly $4\times (w+m+n)$ matrices $M$ satisfying the following:
\begin{enumerate}
\item \label{prop:Vpropsforms}The columns consist only of  7 types:
$$
\begin{pmatrix}
\star \\
1 \\
1 \\
0
\end{pmatrix},
\begin{pmatrix}
1 \\
\star \\
0 \\
1
\end{pmatrix},
\begin{pmatrix}
0 \\
1 \\
\star \\
1
\end{pmatrix},
\begin{pmatrix}
1 \\
0 \\
1 \\
\star
\end{pmatrix},
\begin{pmatrix}
\star \\
\star \\
0 \\
0
\end{pmatrix},
\begin{pmatrix}
0 \\
0 \\
\star \\
\star
\end{pmatrix},
\text{ or }
\begin{pmatrix}
1 \\
1 \\
1 \\
1
\end{pmatrix}.
 $$
\item\label{prop:Vadjacentexclusion}  There are no pairs of adjacent columns of the form
$
\begin{pmatrix}
\star & 0 \\
\star & 0 \\
0 & \star \\
0 & \star
\end{pmatrix}.
$
\item The number of times each $1$, $0$, or $\star$ appears in each row of $M$ is given in Table~\ref{ta:vcounts}.
\end{enumerate}
\end{prop}

We prove Proposition~\ref{lemma:V} via a series of lemmas.

\begin{lemm}\label{lem:Vsatisfiesprops}
Elements of $\widetilde{\mathcal V}$  
satisfy the three conditions of Proposition~\ref{lemma:V}. 
\end{lemm}

\begin{proof}[Proof of Lemma~\ref{lem:Vsatisfiesprops}]
We first show that elements $BL^\star(V)\in \widetilde V$ satisfy Property 1, noting that $(BL^\star(V) )_k = (N^\star_k)_k$. 

If the $k$th column of $N^\star_{k-1}$ consists of only $0$s and $1$s, then by referencing Table~\ref{eq:blV} and by Corollary~\ref{cor:insertionstar} the possibilities for the $k$th column of $N^\star _k$ are
$$
(BL^{\star}(V))_k = (N^\star _k)_k=
\begin{pmatrix}
\star \\
1 \\
1 \\
0
\end{pmatrix},
\begin{pmatrix}
1 \\
\star \\
0 \\
1
\end{pmatrix},
\begin{pmatrix}
0 \\
1 \\
\star \\
1
\end{pmatrix},
\begin{pmatrix}
1 \\
0 \\
1 \\
\star
\end{pmatrix},
\begin{pmatrix}
\star \\
\star \\
0 \\
0
\end{pmatrix},
\begin{pmatrix}
0 \\
0 \\
\star \\
\star
\end{pmatrix},
\text{ or }
\begin{pmatrix}
1 \\
1 \\
1 \\
1
\end{pmatrix}.
 $$
Suppose the $k$th column of $N^\star_{k-1}$ contains a $\star$.  We consider each row separately.  
If $\star$ is the entry of the first row of $(N^\star_{k-1})_k$, then every subsequent entry in the first row is $\star$ since the $0$s and $1$s of $N^\star_{k-1}$ are left aligned (see Corollary~\ref{rightalignedstar}).  Hence by Table~\ref{ta:vcounts} and Lemma~\ref{le:countpersistencestar}, every entry in the fourth row, columns $k, \dots, w+m+n$, must be a $0$.  Since there are no $\star$s in the fourth row in columns $k, \dots, w+m+n$, no entry in the second row of $N^\star_{k-1}$ in these columns can be a $0$.  Therefore the second entry of $(N^\star_{k-1})_k$ is $1$ or $\star$.

If the entry in the second row of $(N^\star_{k-1})_k$ is a $\star$, by a similar argument  the entry in the third row of $(N^\star_{k-1})_k$ must be a 0 and so $(N^\star_{k-1})_k = (\star, \star, 0, 0)^T$.

Suppose $(N^\star_{k-1})_k$ has $1$ in its second row. Since there are no $0$s in the first row of $N^\star_{k-1}$, columns $k, \dots, w+m+n$, there are no $\star$s in the third row in these columns.  Thus the only  possibilities for $(N^\star_{k-1})_k$ are $(N^\star_{k-1})_k=(\star, 1, 1, 0)^T$, which is one of the 7 types listed in the first property, or $(N^\star_{k-1})_k=(\star, 1, 0, 0)^T$.  If $(N^\star_{k-1})_k= (\star, 1, 0, 0)^T$, then $BL^\star_k$ rotates the second row of $N^\star_{k-1}$ and Corollary~\ref{cor:insertionstar} ensures that $(BL^\star(V))_k = (N^\star_k)_k =(\star, \star, 0, 0)^T,$ also one of the types listed in Property 1.

Similar arguments show that if the second, third, or fourth rows of $(N^\star_{k-1})_k$  are $\star$ then either $(N^\star_{k-1})_k$ is already one of the 7 types or $BL_k^\star$ shifts a $\star$ into an appropriate row so that $(BL^\star(V))_k = (N^\star _k)_k$ is one of the types listed in Property 1.

To prove the second property we check that if the $k$th column of $M \in \widetilde{\mathcal V}$ is $(\star, \star, 0, 0)^T$ then the $k+1$st column cannot be $(0, 0, \star, \star)^T$.  There are three possiblities for $(N^\star _{k-1})_k$ that could lead to $(M)_k=(N^\star _k)_k = (\star, \star, 0,0)^T$.  

First suppose that $R_{BL^\star_k (N^\star _{k-1})} = \emptyset$, in which case $(N^\star _{k-1})_k = (\star, \star, 0,0)^T$.  By Corollary~\ref{rightalignedstar} and Lemma~\ref{le:countpersistencestar} $(N^\star _{k-1})_{\ell} = (\star, \star, 0, 0)^T$ for all $k \leq \ell \leq w+m+n$; in particular $(N^\star _{k+1})_{k+1} \neq (0, 0, \star, \star)^T$.

The second possibility is that $R_{BL^\star_k (N^\star _{k-1})} = \{1\}$.  By Table~\ref{eq:blV} the only choice for $(N^\star _{k-1})_k$ is $(1, \star, 0, 0)^T$.  By Lemma~\ref{rightalignedstar} every subsequent entry in the second row of $N^\star_{k-1}$ must be $\star$, which excludes $(N^\star _{k+1})_{k+1} = (0, 0, \star, \star)^T$.

The last possibility is that $R_{BL^\star_k (N^\star _{k-1})} = \{2\}$.  Again by referencing Table~\ref{eq:blV} it must be the case that $(N^\star _{k-1})_k = (\star, 1, 0, 0)^T$.  Thus once again by Lemma~\ref{rightalignedstar} every entry in the second row of $N^\star _{k-1}$ to the right of the $k$th column must also be a $\star$ and so $(M)_{k+1}=(N^\star _{k+1})_{k+1} \neq (0, 0, \star, \star)^T$.

Finally, Property 3 is immediately satisfied by Corollary~\ref{co:orderstayssamestar}. 
  \end{proof}

We now verify that any $M$ satisfying the properties of Proposition~\ref{lemma:V} is $BL^\star(V)$ for some $V \in \mathcal V$, and hence $M \in \widetilde{\mathcal V}$.

\begin{lemm}\label{lem:propsimply}
Let $M$ be any matrix satisfying Properties 1, 2, and 3 in Proposition~\ref{lemma:V}, and let $V$ be the matrix obtained by right justifying the $\star$s in $M$.  Then $BL^\star (V) = M$.
\end{lemm}

\begin{proof}
Observe that $V \in \mathcal V$ due to Property 3.
Suppose that the first $k-1$ columns of $BL^{\star}(V)$ and $M$ agree, and consider the $k$th column ($k$ could be 1).  If the $k$th column of $BL^{\star}(V)$ is any of $(\star,1,1,0)^T$, $(1,\star,0,1)^T$, $(0,1,\star,1)^T$, $(1,0,1,\star)^T$, or $(1,1,1,1)^T$, so too must be the $k$th column of $M$, to have preserved the order of the $0$s and $1$s when removing $\star$s from $M$ to form $V$.

Suppose the $k$th column of $BL^{\star}(V)$ is $(\star,\star,0,0)^T$.  An appearance of a $1$ in the 3rd or 4th row of the $k$th column of $M$ would disrupt the order of  $0$s and $1$s. Thus the $k$th column of $M$ must contain only $\star$ or $0$s in the 3rd and 4th rows. Only $(\star,\star,0,0)^T$ and $(0, 0, \star,\star)^T$ satisfy this condition. 

If  $M$ has $(\star,\star,0,0)^T$ in the $k$th column, we're done. If not, then it must have $(0, 0, \star,\star)^T$  in the $k$th column. Then the appearance of a $1$ in the first or second row of the $(k+1)$st column of $BL^{\star}(V)$ would disrupt the order of the $0$s and $1$s. Thus the $(k+1)$st column of $BL^{\star}(V)$ must contain only $\star$ or $0$s in the 1st or 2nd rows. Of the seven possibilities, only $(\star,\star,0,0)^T$ and $(0, 0, \star,\star)^T$ satisfy this condition. By the same reasoning, $BL^{\star}(V)$ has $(\star,\star,0,0)^T$  in all subsequent columns, until the first occurrence of $(0, 0, \star,\star)^T$, guaranteed to occur by a simple count. It follows that $BL^{\star}(V)$ has two adjacent columns of the form disallowed by Property 2.

A similar argument works if the $k$th column of $BL^{\star}(V)$ is $(0,0,\star,\star)^T$. If  $M$ has $(0,0,\star,\star)^T$ in the $k$th column, we're done. If not, then it must have $(\star,\star,0,0)^T$  in the $k$th column. By the same reasoning, $M$ has $(\star,\star,0,0)^T$  in all subsequent columns, until the first occurrence of $(0, 0, \star,\star)^T$, guaranteed to occur by a simple count. 
Then $M$ has two adjacent columns of the form disallowed by Property 2.
  \end{proof}

\begin{proof}[Proof of Proposition~\ref{lemma:V}]  
By Lemma~\ref{lem:Vsatisfiesprops} any $M \in \widetilde{\mathcal V}$ satisfies the conditions listed in Proposition~\ref{lemma:V}; by Lemma~\ref{lem:propsimply} any matrix $M$ satisfying the properties is $BL^\star (V)$ for some $V \in \mathcal V$ and hence $M \in \widetilde{\mathcal V}$.
  \end{proof}

\subsection{Bijection between $\mathcal V$ and $\mathcal S$}\label{sse:identitybijection}
Finally, we complete the proof of Theorem~\ref{identity} by establishing the bijection between sets of the right size.
\begin{prop}\label{prop:bijection}
There is a bijection between sets $\mathcal V$ and $\mathcal S$.
\end{prop}
\begin{proof}
We establish a bijection between the sets $\widetilde {\mathcal V}$ and $\widetilde{\mathcal S}$, by mapping the seven vectors $
\begin{pmatrix}
\star \\
1 \\
1 \\
0
\end{pmatrix},
\begin{pmatrix}
1 \\
\star \\
0 \\
1
\end{pmatrix},
\begin{pmatrix}
0 \\
1 \\
\star \\
1
\end{pmatrix},
\begin{pmatrix}
1 \\
0 \\
1 \\
\star
\end{pmatrix},
\begin{pmatrix}
\star \\
\star \\
0 \\
0
\end{pmatrix},
\begin{pmatrix}
0 \\
0 \\
\star \\
\star
\end{pmatrix},
\begin{pmatrix}
1 \\
1 \\
1 \\
1
\end{pmatrix}
 $
  listed in Lemma~\ref{lemma:V} for elements of $\widetilde{\mathcal V}$ to the seven 2-vectors 
  $$ \begin{pmatrix} T\\ C\end{pmatrix},
  \begin{pmatrix} S\\ C\end{pmatrix}, \begin{pmatrix} P\\ C\end{pmatrix}, 
   \begin{pmatrix} Q\\ C\end{pmatrix},
 \begin{pmatrix} -\\ U\end{pmatrix},  \begin{pmatrix} O\\ -\end{pmatrix}, 
 \begin{pmatrix} R\\ C\end{pmatrix}
 $$
  of elements of $\widetilde{\mathcal S}$, respectively.   Note that the excluded configurations of $\widetilde{\mathcal V}$ correspond exactly to the excluded configurations of $\widetilde{\mathcal S}$

Lemma \ref{co:FGinj} and Lemma \ref{lemma:Vinj} establish bijections from $\mathcal S$ to $\widetilde{\mathcal S}$ and from $\mathcal V$ to $\widetilde {\mathcal V}$, respectively.  Thus there is a bijection $\mathcal S \rightarrow \mathcal V$.
 \end{proof}
It follows that $|\mathcal S| = |\mathcal V|$.  Since $|\mathcal V|$ is given by the left-hand side of Equation~\eqref{eq:combinatorialidentity} and $|\mathcal S|$ is given by the right side of Equation~\eqref{eq:combinatorialidentity}, we have 
concluded the proof of Theorem~\ref{identity}.

As an immediate corollary, we obtain Verdermonde's Identity. Let $n=0$ in Theorem~\ref{identity}, and substitute $a=x$, $b = y-x+m$, $s=m$, and $r=i$.

\begin{coro}[Vandermonde]
Let $a, b\in \IZ$. Then
$$
{a+b \choose s} = \sum_{r}{a \choose r} {b \choose s-r}.
$$
\end{coro}

\nocite{*}

\bibliographystyle{spmpsci}
\bibliography{bibfile}

\begin{thebibliography}{10}
\providecommand{\url}[1]{{#1}}
\providecommand{\urlprefix}{URL }
\expandafter\ifx\csname urlstyle\endcsname\relax
  \providecommand{\doi}[1]{DOI~\discretionary{}{}{}#1}\else
  \providecommand{\doi}{DOI~\discretionary{}{}{}\begingroup
  \urlstyle{rm}\Url}\fi

\bibitem{Abe-Horiguchi-Kuwata-Zeng}
Abe, H., Horiguchi, T., Kuwata, H., Zeng, H.: Geometry of {P}eterson {S}chubert
  calculus in type \uppercase{$A$} and left-right diagrams (2021).
\newblock ArXiv:2104.02914 [math.AG]

\bibitem{abe2019hessenberg}
Abe, T., Horiguchi, T., Masuda, M., Murai, S., Sato, T.: Hessenberg varieties
  and hyperplane arrangements.
\newblock J. Reine Angew. Math. \textbf{2020}(764), 241--286 (2020)

\bibitem{AJS}
Andersen, H.H., Jantzen, J.C., Soergel, W.: Representations of quantum groups
  at a $p$th root of unity and of semisimple groups in characteristic $p$:
  independence of $p$.
\newblock Ast\'erisque \textbf{220}, 321 (1994)

\bibitem{sbilley}
Billey, S.C.: Kostant polynomials and the cohomology ring for
  \uppercase{$G/B$}.
\newblock Duke Math. J. \textbf{96}(1), 205--224 (1999)

\bibitem{Brion00}
Brion, M.: Poincar\'e duality and equivariant (co)homology.
\newblock Michigan Math. J. \textbf{48}, 77--92 (2000)

\bibitem{BC}
Brion, M., Carrell, J.B.: The equivariant cohomology ring of regular varieties.
\newblock Michigan Math. J. \textbf{52}, 189--203 (2002)

\bibitem{CS}
Chang, T., Skjelbred, T.: The topological \uppercase{S}chur lemma and related
  results.
\newblock Ann. Math. \textbf{100}(2), 307--321 (1974)

\bibitem{Dr}
Drellich, E.: Monk's rule and \uppercase{G}iambelli's formula for
  \uppercase{P}eterson varieties of all \uppercase{L}ie types.
\newblock J. Algebraic Combin. \textbf{41}(2), 539--575 (2015)

\bibitem{Fukukawa2015}
Fukukawa, Y., Harada, M., Masuda, M.: The equivariant cohomology rings of
  \uppercase{P}eterson varieties.
\newblock J. Math. Soc. Japan \textbf{67}(3), 1147--1159 (2015)

\bibitem{Goldin-Milhalcea-Singh}
Goldin, R., Mihalcea, L., Singh, R.: Positivity of {P}eterson {S}chubert
  calculus (2021).
\newblock ArXiv:2106.10372 [math.AG]

\bibitem{GKM}
Goresky, M., Kottwitz, R., MacPherson, R.: Equivariant cohomology,
  \uppercase{K}oszul duality, and the localization theorem.
\newblock Invent. Math. \textbf{131}(1), 25--83 (1997)

\bibitem{gr}
Graham, W.: Positivity in equivariant \uppercase{S}chubert calculus.
\newblock Duke Math. J. \textbf{109}(3), 599--614 (2001)

\bibitem{HHM}
Harada, M., Horiguchi, T., Masuda, M.: The equivariant cohomology rings of
  \uppercase{P}eterson varieties in all \uppercase{L}ie types.
\newblock Canad. Math. Bull. \textbf{58}, 80--90 (2014)

\bibitem{HT}
Harada, M., Tymoczko, J.S.: A positive \uppercase{M}onk formula in the
  $\uppercase{S}^1$-equivariant cohomology of type $\uppercase{A}$
  \uppercase{P}eterson varieties.
\newblock Proc. London Math. Soc. \textbf{103}(1), 40--72 (2011)

\bibitem{insko.tymoczko:intersection.theory}
Insko, E., Tymoczko, J.: Intersection theory of the {P}eterson variety and
  certain singularities of {S}chubert varieties.
\newblock Geom. Dedicata \textbf{180}, 95--116 (2016)

\bibitem{Ko96}
Kostant, B.: Flag manifold quantum cohomology, the \uppercase{T}oda lattice,
  and the representation with highest weight $\rho$.
\newblock Sel. Math. New Ser. \textbf{2}(1), 43--91 (1996)

\bibitem{peterson:notes}
Peterson, D.: Quantum cohomology of $\uppercase{G}/\uppercase{P}$ (1997).
\newblock Lecture Notes, M.I.T.

\bibitem{rietsch2003totally}
Rietsch, K.: Totally positive \uppercase{T}oeplitz matrices and quantum
  cohomology of partial flag varieties.
\newblock J. Am. Math. Soc. \textbf{16}(2), 363--392 (2003)

\bibitem{rietsch2008}
Rietsch, K.: A mirror symmetric construction of $q\uppercase{H^*_T(G/P
  )}_{(q)}$.
\newblock Adv. Math. \textbf{217}, 2401--2442 (2008)

\bibitem{sz}
Sz{\'e}kely, L.A.: Common origin of cubic binomial identities; a generalization
  of \uppercase{S}ur{\'a}nyi's proof on \uppercase{L}e \uppercase{J}en
  \uppercase{S}hoo's formula.
\newblock J. Comb. Theory Ser. A. \textbf{40}(1), 171--174 (1985)

\bibitem{T1}
Tymoczko, J.S.: Paving \uppercase{H}essenberg varieties by affines.
\newblock Sel. Math. New Ser. \textbf{13}(2), 353--367 (2007)

\end{thebibliography}



\end{document}